\documentclass[pdf,a4paper,11pt,oneside]{article}

\usepackage[english]{babel}
\usepackage{authblk}

\usepackage{amsfonts}
\usepackage{amssymb}
\usepackage{indentfirst}
\usepackage{mathrsfs}
\usepackage{amsfonts,amsmath,latexsym,bm}
\usepackage{amsthm}
\usepackage{float}
\newcommand{\tabcaption}{\def\@captype{table}\caption}
\usepackage{graphicx}
\usepackage{cancel}
\usepackage{multirow}
\usepackage{color}
\usepackage{caption}
\usepackage{subfigure}
\usepackage[colorlinks=true, allcolors=blue]{hyperref}

\newtheorem{lem}{Lemma}[section]
\newtheorem{thm}{Theorem}[section]
\newtheorem{propo}{Proposition}[section]
\newtheorem{rem}{Remark}[section]

\numberwithin{equation}{section}

\title{Spatial second-order positive and asymptotic preserving filtered $P_N$ schemes for nonlinear radiative transfer equations}
\author[a,b]{Xiaojing Xu}
\author[c]{Song Jiang}
\author[c,d]{Wenjun Sun %\footnote{Corresponding author}
}
\affil[a]{\scriptsize School of Mathematics and Information Science, Guangxi University, Nanning 530004, Guangxi,
China}
\affil[b]{\scriptsize School of Mathematics and Physics,
Southwest University of Science and Technology, Mianyang 621010, Sichuan, China
 }
\affil[c]{\scriptsize Institute of Applied Physics and Computational Mathematics, Fenghao East Road 2, Beijing 100094, China }
\affil[d]{\scriptsize Center for Applied Physics and Technology, College of Engineering, Peking University, Beijing 100871, China }

\begin{document}
\setlength{\arraycolsep}{0.5mm}

%\captionsetup[figure]{font=small,labelfont={bf},labelformat={default},labelsep=period,name={Fig.}}
%\captionsetup[table]{labelfont={bf},labelformat={default},labelsep=period,name={Table}}

\date{}
\maketitle

\vskip -10mm
{\small\qquad  E-mail : xuxiaojing0603@126.com \ \  jiang@iapcm.ac.cn \ \  sun$\_$wenjun@iapcm.ac.cn }

\begin{abstract}
A spatial second-order scheme for the nonlinear radiative transfer equations is introduced in this paper. The discretization scheme is based on the
filtered spherical harmonics ($FP_N$) method for the angular variable and the unified gas kinetic scheme (UGKS) framework
for the spatial and temporal variables respectively. In order to keep the scheme positive and second-order accuracy,
firstly, we use the implicit Monte Carlo (IMC) linearization method \cite{IMC-1971} in the construction of the UGKS
numerical boundary fluxes. This is an essential point in the construction.
 Then, by carefully analyzing the constructed second-order fluxes involved in the macro-micro decomposition, which is induced
 by the $FP_N$ angular discretization, we establish the sufficient conditions that guarantee the
 positivity of the radiative energy density and material temperature. Finally, we employ linear scaling limiters for the angular variable
 in the $P_N$ reconstruction and for the spatial variable in the piecewise linear slopes reconstruction respectively,
 which are shown to be realizable and reasonable to enforce the sufficient conditions holding. Thus, the desired scheme,
 called the $PPFP_N$-based UGKS, is obtained.
Furthermore, we can show that in the regime $\epsilon\ll 1$ and the regime $\epsilon=O(1)$, the second-order fluxes can be simplified.
And, a simplified spatial second-order scheme, called the $PPFP_N$-based SUGKS, is thus presented, which possesses all the
properties of the non-simplified one. Inheriting the merit of UGKS, the proposed schemes are asymptotic preserving.
By employing the $FP_N$ method for the angular variable, the proposed schemes are almost free of ray effects.
Moreover, the above-mentioned way of imposing the positivity would not destroy both AP and second-order accuracy properties.
To our best knowledge, this is the first time that spatial second-order, positive, asymptotic preserving and almost free of ray effects
schemes are constructed for the nonlinear radiative transfer equations without operator splitting.
Therefore, this paper improves our previous work on the first-order scheme \cite{Xu-Jiang-Sun-2021}
which could not be directly extended to high order, while keeping the solution positive.
Various numerical experiments are included to validate the properties of the proposed schemes.



\end{abstract}

\section{Introduction}
The numerical simulation of the radiative transfer equations has a wide spectrum of applications in science and technology, such as
in astrophysics and inertial confinement fusion research. From the view point of numerical simulation, the radiative transfer equations
result in several difficulties, for example, the integro-differential form and high-dimensionality (time, spatial and angular variables).
%% and strong nonlinearity.
In particular, the multi-scale features characterized by wide-ranging optical thickness of background
materials induce in numerical challenges.
%%%%%%%%%%
Indeed, the optical thickness of background materials has a great impact on the behavior of radiation transfer.
For a material with low opacity, the radiation propagates in a transparent way, while for a material with high opacity, the
 radiation behaves like a diffusion process.
 In order to resolve the kinetic-scale-based radiative transfer equations in numerical simulation,
the spatial mesh size in many numerical methods usually should be comparable to the photon's mean-free path, which is very small in the optically thick regions, leading to huge computational costs.

The other numerical challenges include preserving positivity and eliminating ray-effects. Namely,
the positivity of the exact solution to the radiation transfer equations (i.e., the radiation intensity and the material temperature)
should be preserved numerically. The appearance of negative numerical solution may cause the breakdown or instability of the numerical simulation.
The ray effects caused by the usual discrete ordinate ($S_N$) method would destroy the symmetry when it is used to solve problems involving
isolated sources within optically thin media.

In the recent years, the numerical simulation of the radiative transfer equations has been developed rapidly.
An efficient approach to cope with the above-mentioned difficulties is the multi-scale numerical scheme
that is asymptotic preserving (AP), positive preserving (PP) and free of or less susceptible to ray effects,
whose construction  remains a challenging issue even now.
In fact,
the  schemes  with  AP property alone have been studied extensively in the literature \cite{ S.Jin-AP-a-class-2010, S.Jin-AP-a-review-2010,   E.W. Larsen-1987,   A.W. Larsen-1989,  L. Mieussens-linear kinetic-2013, T.Xiong-AP-2022}, including micro-macro decomposition, even-odd decomposition, implicit-explicit (IMEX) method
and so on.
%There are also
%many research papers dedicated to discuss how to make a scheme
%positive preser
There are also numerical schemes only focusing on the positive preserving property, such as in
 \cite{positive-shceme-1994, Laiu-positive-FPN-2016, Ling-Dan-PP-DG-2018, positive-scheme-1997, Da-Ming-Yuan-PP-DG-2016},
 the positive solution can be obtained through using a linear scaling limiter, or a flux limiter, or adopting special function
spaces, etc. On the other hand, in order to mitigate the ray effects of the $S_N$ method, alternative angular discretization methods,
such as the angular finite element method\cite{angular-FEM-1,angular-FEM-2}, wavelets method\cite{wavelet-method}, angular adaptivity  method\cite{adaptive-2015, angular-adaptive}, spherical harmonic ($P_N$) method\cite{PN-M.Benassi-1984} have been proposed.
Besides, the numerical schemes developed in \cite{PP-AP-scheme-2020, Jingwei-Hu-2019, Xu-Jiang-Sun-2022} are both asymptotic preserving
and positive preserving, while the schemes in \cite{Laiu-positivity-limiter-FPN-2019, Laiu-positive-FPN-2016} are positive and ray effects
free as well.

The original unified gas kinetic scheme (UGKS) for continuum and rarefied flows in \cite{Xu Kun-UGKS-2010}, which has been adopted to
 the linear radiation equation \cite{L. Mieussens-linear kinetic-2013} and then extended to the nonlinear (frequency-dependent)
 radiative transfer equations
 (\cite{Sun-Jiang-Xu-2015,Sun-Jiang-Xu-2015-frequency-dependent, Sun-Jiang-Xu-2018-implicit-frequency-dependent,
 Sun-Jiang-Xu-2017-equilibriun-and-non,Sun-Jiang-Xu-2017-unstructed mesh}),
 is proved to be asymptotic preserving. Numerical results, however, show that it is not positive preserving.
Moreover, the aforementioned UGKS may suffer from ray effects, as it is based on the $S_N$ method for the angular variable discretization.
In order to mitigate the ray effects, we have extended the angular $S_N$ method to the angular finite element and the filtered spherical
harmonics ($FP_N$) approximation in our previous works \cite{Xu-Sun-Jiang-2020, Xu-Jiang-Sun-2021}.
Due to the rotational invariance of the $FP_N$ method, the scheme in \cite{Xu-Jiang-Sun-2021} is almost free of ray effects.
Remind that although the filtering technique \cite{McClarren-Hauck-FPN-2010-2, Radice-FPN(new-framework)-2013} can suppress spurious
oscillations in the $P_N$ approximation, the negative solution of the $P_N$ approximation \cite{Pn-negative} may still occur.
Thus, we have studied the PP property in \cite{Xu-Jiang-Sun-2021} as well. Different from the positive $FP_N$ schemes in
\cite{Laiu-positivity-limiter-FPN-2019, Laiu-positive-FPN-2016} where the PP property was achieved by imposing a positivity limiter,
or by a positive moment closure that needs solving an additional expensive optimization problem, the PP property
in \cite{Xu-Jiang-Sun-2021} is obtained by giving the sufficient conditions to guarantee the positivity of the radiative energy
density and material temperature, while the sufficient conditions are enforced through a much cheap linear scaling limiter.
As far as we know, the scheme in \cite{Xu-Jiang-Sun-2021} is almost the only one that possesses all the
mentioned properties for the nonlinear radiative transfer equations; but unfortunately, the accuracy in the spatial variable
is only of first order. We remark that the scheme in \cite{Laiu-AP-PP-FPN-2019} also possesses all the above
desired properties, but is for the linear radiative transfer equation. To obtain the PP property
with spatial second-order accuracy, the authors in \cite{Xu-Jiang-Sun-2022} constructed both PP and AP-UGKS schemes for
{\it linear} radiation transfer equation only, based on the framework of UGKS and the angular $S_N$ method,
thus suffering from the ray effects.

In this paper, we shall extend our  previous PP and AP $FP_N$-based UGKS for the nonlinear gray radiative transfer equations \cite{Xu-Jiang-Sun-2021} from spatial first-order accuracy to the second-order. Compared with the linear radiative transfer equation
in \cite{Xu-Jiang-Sun-2022}, for the case of the nonlinear gray radiative transfer equations, in order to achieve all the desired properties,
particularly the positive preserving property, the extension is not trivial, and some additional techniques have to be introduced,
for example, firstly, the linearized method of the implicit Monte Carlo (IMC) \cite{IMC-1971} is used here to construct
the numerical boundary fluxes. Then, due to the requirement of piecewise linear reconstruction to get the spatial second-order accuracy,
which brings technical difficulties in finding the sufficient conditions to preserve the positivity of the numerical solutions.
Such difficulties have to be carefully dealt with. Finally, to make the positivity preserving conditions hold, we apply
an additional linear scaling limiter to modify the reconstructed slopes,
while for the first-order scheme in \cite{Xu-Jiang-Sun-2021}, such a limiter is not needed.

As in \cite{Xu-Jiang-Sun-2022}, by a detailed analysis of the numerical boundary fluxes,
a simplification method of the fluxes in regimes $\epsilon\ll 1$ and $\epsilon=O(1)$
is presented to minimize the computational costs but without sacrificing the second-order accuracy.
To our best knowledge, there seems no scheme for the nonlinear radiative transfer equations without operator splitting in the literature
that is of spatial second-order accuracy, positive and asymptotic preserving, and almost free of ray effects.

The remainder of this paper is organized as follows. In Section 2 we introduce the nonlinear gray radiative transfer equations, while
in Section 3 we present the construction of a spatial second-order $FP_N$-based UGKS. Through a detailed analysis in Section 4,
we find the sufficient conditions to preserve the positivity of the $FP_N$-based UGKS and impose these conditions by using linear scaling
limiters to obtain the positive UGKS. We then analyze the asymptotic preserving property of this positive scheme, and consequently,
a positive and asymptotic preserving filtered spherical harmonics method ($PPFP_N$-based UGKS) is hence proposed in this paper.
In Section 5, we show that in the regime $\epsilon\ll 1$ and the regime $\epsilon=O(1)$, the boundary fluxes can be simplified.
Therefore, a simplified $PPFP_N$-based UGKS ($PPFP_N$-based SUGKS) is obtained in these regimes.
Section 6 presents a number of numerical tests to validate the current schemes. Finally, a conclusion is given in Section 7.

\section{Model problems}

We consider the following nonlinear radiative transfer equations, which are given in the scaled form:
\begin{equation}\label{Gray Radiative transfer equaitons}
\left\{\begin{array}{ll}
\displaystyle{ \frac{\epsilon^2}{c}\frac{\partial I}{\partial t}+
\epsilon {\bm \Omega}\cdot\nabla_{\bf r} I=\sigma\Big( \frac{1}{4\pi}ac T^4-I\Big),}\\[3mm]
\displaystyle{ \epsilon^2C_{\nu}\frac{\partial T}{\partial t}
=\sigma\Big( \int_{S^2} I d {\bm \Omega}-acT^4\Big).}
\end{array}\right.
\end{equation}
The equations describe the radiative transfer and the energy exchange between radiation and background materials.
Here ${\bm r}$ denotes the scaled space variable $(x,y,z)$ in ${\mathbb R}^3$,
$t$ is the scaled time variable, ${\bm \Omega}$ is the angular variable $(\xi,\eta,\zeta)$ in $S^2$
(the unit sphere in ${\mathbb R}^3$). Functions $I({\bm r},{\bm \Omega},t)$ and $T({\bm r},t)$
are the radiation intensity and material temperature, respectively.
$\sigma({\bm r}, T)$ is the scaled opacity, the radiation constant $a$ and speed of light $c$ also are scaled,
$\epsilon>0$ is the Knudsen number, and $C_{\nu}({\bm r},t)$ is the scaled heat capacity.

 Equations \eqref{Gray Radiative transfer equaitons} are a relaxation model for the radiation intensity
to the local thermodynamic equilibrium, in which the emission source is a Planckian at local material temperature.
It has been shown in \cite{E.W. Larsen-Asymptotic analysis-1983} that as the parameter $\epsilon\rightarrow 0$, away from
boundaries and initial layers, the intensity $I$ approaches to a Planckian at local temperature $T^{0}$,
with $T^{0}$ satisfies a nonlinear diffusion equation, i.e.,
\begin{equation}\label{nonlinear diffusion limiting equation}
 \qquad C_{\nu}\frac{\partial}{\partial t}  T^{(0)} + a \frac{\partial}{\partial t}(T^{(0)})^4 = \nabla_{\bf r} \cdot \frac{ac}{3\sigma}\nabla_{\bf r} (T^{(0)})^4,\qquad
 I^{(0)}=\frac{1}{4\pi}ac(T^{(0)})^4.
\end{equation}

In the following, we denote $\phi=acT^4 $, and rewrite \eqref{Gray Radiative transfer equaitons} as
 \begin{equation}\label{Rewritten Gray Radiative transfer equaitons}
\left\{\begin{array}{ll}
\displaystyle{ \frac{\epsilon^2}{c}\frac{\partial I}{\partial t}+
\epsilon {\bm \Omega}\cdot\nabla_{\bf r} I=\sigma\Big( \frac{1}{4\pi}\phi-I\Big),}\\[3mm]
\displaystyle{ \epsilon^2 \frac{\partial \phi}{\partial t}=\beta\sigma(\int_{S^2} I d {\bm \Omega} - \phi),}
\end{array}\right.
\end{equation}
where $\beta\equiv \beta (x,t)= 4acT^3/C_{\nu}$.

\section{A spatial second-order  $FP_N$-based UGKS }
\label{section: A FPN-based UGKS}

 In this section,  we first revisit the filtered spherical harmonics ($FP_N$) discretization for the angular variable,
 and unified gas kinetic scheme (UGKS) for the spatial and temporal variables in \cite{Xu-Jiang-Sun-2021}. And then,
 the construction of the spatial second-order fluxes is introduced in details, which is the crucial step to get the final positive
 preserving scheme. As a consequence, a spatial second-order $FP_N$-based UGKS is constructed for the equations
 \eqref{Rewritten Gray Radiative transfer equaitons}.

\subsection{Discretization of the angular, spatial and temporal variables}

Denote  by  $\psi_{\ell}^m(\bm \Omega)$ ($\ell\geq 0$, $-\ell\leq m \leq \ell $) the real normalized spherical harmonic
function of degree $\ell$ and order $m$ that satisfies
 ${\bm\langle}\psi_{\ell}^m\psi_{\ell'}^{m'}  {\bm\rangle}=\delta_{\ell,\ell'}\delta_{m,m'},$
with ${\bm\langle} \cdot {\bm\rangle}=\int_{S^2} \cdot d{\bm \Omega}$ and $\delta_{\ell,\ell'}$ the Kronecker delta  function.
For any nonnegative integer $N$, let
$$\vec{\bm \psi}=(\psi^0_0,~ \breve{\vec{{\bm \psi}}}' )',  \qquad \qquad
\vec{\bm I}=(I^0_{0},~\breve{\vec{\bm I}}')',$$
with $\breve{\vec{{\bm \psi}}}= ( \psi_{1}^{-1},\psi_{1}^{0}, \psi_{1}^{1},...,\psi_{N}^{N})'$  the  spherical harmonic vector
and
$I^0_{0}={\bm\langle}\psi^0_0I {\bm\rangle}$, $\breve{\vec{\bm I}}={\bm\langle}\breve{\vec{{\bm \psi}}}I {\bm\rangle}$  the corresponding
moments, here the upper script $'$ is the transpose operation.
Moreover, denoting $\rho=\int_{S^2} I d {\bm \Omega}$ and noting that $I^0_0 = \frac{1}{2\sqrt{\pi}} \rho$, one has
 \begin{equation}\label{decomposition-of-I}
 \vec{\bm I} =(\frac{1}{2\sqrt{\pi}} \rho, ~\breve{\vec{\bm I}}' )'.
 \end{equation}

With the  $FP_N$ angular discretization and macro-micro decomposition  as in \cite{Xu-Jiang-Sun-2021}, we can obtain
\begin{equation}\label{macro equation}
\left\{\begin{array}{ll}
\displaystyle{\frac{\epsilon^2}{c}\frac{\partial \rho}{\partial t}+
\epsilon \nabla_{\bf r}\cdot \langle {\bm \Omega} I\rangle=\sigma(\phi-\rho),}\\[3mm]
\displaystyle{\epsilon^2 \frac{\partial \phi}{\partial t}=\beta\sigma(\rho- \phi), }
\end{array}\right.
\end{equation}
and
\begin{equation}\label{micro equation}
\displaystyle{ \frac{\epsilon^2}{c}\frac{\partial \breve{\vec{\bm I}}}{\partial t}+
\epsilon \nabla_{\bf r}\cdot {\bm\langle} \breve{\vec{\bm \psi}}{\bm \Omega} I{\bm\rangle} = - \sigma \breve{\vec{\bm I}} - \frac{\epsilon^2}{c} \sigma_f \breve{{\bf F}}\breve{\vec{\bm I}},}
\end{equation}
which are referred to as the macro and micro equations, respectively.
Here
$\nabla_{\bf r}\cdot {\bm\langle} \vec{\bm \psi}{\bm \Omega} I{\bm\rangle}
= \partial_x {\bm\langle} \vec{\bm \psi} \xi I {\bm\rangle}
+ \partial_y {\bm\langle} \vec{\bm \psi} \eta I {\bm\rangle}
+ \partial_z {\bm\langle} \vec{\bm \psi} \zeta I {\bm\rangle}.$
$\sigma_f\geq0$ is a filter parameter,
$\breve{{\bf F}}$ is a diagonal filtering matrix with elements $\breve{{\bf F}}_{(\ell,m),(\ell,m)} = -\text{ln}(f(\frac{\ell}{N+1}))$,
and $f:{\mathbb R}^+\rightarrow [0,1]$ is a filter function with $f(0)=1$.

To discretize the spatial and temporal variables of \eqref{macro equation}-\eqref{micro equation}, similar to that in
\cite{Xu-Jiang-Sun-2021}, only the two-dimensional Cartesian spatial case is considered. Therefore, one has
${\bf r}=(x,y)$, ${\bm \Omega}=(\xi,\eta)$ and
$\nabla_{\bf r}\cdot {\bm\langle} \vec{\bm \psi}{\bm \Omega} I{\bm\rangle}
= \partial_x {\bm\langle} \vec{\bm \psi} \xi I {\bm\rangle}
+ \partial_y {\bm\langle} \vec{\bm \psi} \eta I {\bm\rangle}$
in the 2-D setting. And also a
uniform partition of the space and time variables is considered for simplicity. Denote by $\Delta x, \Delta y, \Delta t$
the mesh sizes and by $x_i, y_j, t_n$ ($i,j,n\in \mathbb{Z}$) the corresponding mesh nodes.
 Let $(i,j)$ be the cell $\{(x,y): x_{i-1/2}< x < x_{i+1/2}, ~  y_{j-1/2} < y < y_{j+1/2}\}$
with $x_{i\pm1/2}=(x_i+x_{i\pm1})/2$, $y_{j\pm1/2}=(y_j+y_{j\pm1})/2$ being the cell interfaces of $(i,j)$.
  For any function $\varphi(x,y,t)$, denote
  $$ \varphi_{i,j}^n
  =\frac{1}{\Delta x \Delta y}\int_{y_{j-\frac{1}{2}}}^{y_{j+\frac{1}{2}}}
  \int_{x_{i-\frac{1}{2}}}^{x_{i+\frac{1}{2}}} \varphi(x,y,t_n)dx dy. $$
Then, from \cite{Xu-Jiang-Sun-2021}, one has the  following conservative finite volume scheme:
\begin{equation}\label{discrete macro equation}
\left\{\begin{array}{ll}
\rho^{n+1}_{i,j} =\rho^{n}_{i,j} +\frac{\Delta t}{\Delta x} \left( \Phi^{n+1}_{i-1/2,j}-\Phi^{n+1}_{i+1/2,j} \right) +\frac{\Delta t}{\Delta y}
\left(\Upsilon^{n+1}_{i,j-1/2}-\Upsilon^{n+1}_{i,j+1/2} \right) \\[3mm]
\qquad\qquad  +\frac{\sigma_{i,j}^{n+1} c \Delta t}{\epsilon^2}( \phi^{n+1}_{i,j}-\rho^{n+1}_{i,j}) ,\\[4mm]
\displaystyle{\phi_{i,j}^{n+1}=\phi_{i,j}^{n}+ \frac{\beta^{n+1}_{i,j}\sigma^{n+1}_{i,j}\Delta t}{\epsilon^2}}
\left( \displaystyle{\rho^{n+1}_{i,j}-\phi^{n+1}_{i,j} }\right),
\end{array}\right.
\end{equation}
and
\begin{eqnarray}\label{discrete micro equation}\nonumber
\breve{\vec{\bm I}}^{n+1}_{i,j}  &=& \breve{\vec{\bm I}}^{n}_{i,j} +\frac{\Delta t}{\Delta x}
\left( \vec{\bm G}_{i-1/2,j}-\vec{\bm G}_{i+1/2,j} \right) + \frac{\Delta t}{\Delta y} \left( \vec{\bm H}_{i,j-1/2}-\vec{\bm H}_{i,j+1/2} \right) \\
& & - \frac{\sigma_{i,j}^{n+1}c \Delta t}{\epsilon^2} \breve{\vec{\bm I}}^{n+1}_{i,j}
-\Delta t \sigma_f \breve{{\bf F}} \breve{\vec{\bm I}}^{n+1}_{i,j},
\end{eqnarray}
where the interface fluxes are given  by
{\small
\begin{flalign}\label{definition of fluxes}
\hspace{-1mm}
\begin{split}
%\begin{align}
\Phi^{n+1}_{i-1/2,j}=
\frac{c}{\epsilon \Delta t}
\int_{t_n}^{t_{n+1}}\hspace{-1.5mm}{\bm\langle}\xi I(x_{i-\frac{1}{2}},y_j,{\bm \Omega},t) {\bm\rangle} dt, ~~
\vec{\bm G}_{i-1/2,j}=
\frac{c}{\epsilon \Delta t}
\int_{t_n}^{t_{n+1}} \hspace{-1.5mm}{\bm\langle}\xi \breve{\vec{\bm \psi}} I(x_{i-\frac{1}{2}},y_j,{\bm \Omega},t) {\bm\rangle} dt,\\
\Upsilon^{n+1}_{i,j-1/2}=
\frac{c}{\epsilon \Delta t}
\int_{t_n}^{t_{n+1}}\hspace{-1.5mm}{\bm\langle}\eta I(x_{i},y_{j-\frac{1}{2}},{\bm \Omega},t) {\bm\rangle} dt, ~~
\vec{\bm H}_{i,j-1/2}=
\frac{c}{\epsilon \Delta t}
\int_{t_n}^{t_{n+1}}\hspace{-1.5mm}{\bm\langle}\eta \breve{\vec{\bm \psi}} I(x_{i},y_{j-\frac{1}{2}},{\bm \Omega},t) {\bm\rangle} dt,
\end{split}
\end{flalign}}
and will be determined in the next subsection so as to update the system
\eqref{discrete macro equation}-\eqref{discrete micro equation}.

\subsection{Method for the fluxes' construction}
\label{section of Constructing interface fluxes}
This subsection is devoted to constructing the interface fluxes $\Phi^{n+1}_{i-1/2,j}$, $\Upsilon^{n+1}_{i-1/2,j}$ and
 $\vec{\bm G}_{i-1/2,j}$, $\vec{\bm H}_{i,j-1/2}$ in \eqref{definition of fluxes}.
 In order to get a positive-preserving scheme, the construction of the fluxes here is slightly different from that
 in \cite{Xu-Jiang-Sun-2021}.

 As done in the IMC scheme \cite{IMC-1971}, we discretize the time variable of the second equation
 of \eqref{Rewritten Gray Radiative transfer equaitons} firstly, and obtain
\begin{equation}\nonumber
\epsilon^2\frac{\phi^{n+1}-\phi^n }{\Delta t }
=\beta^{n+1} \sigma^{n+1}\Big( \rho^{n+1} - \phi^{n+1}\Big),
\end{equation}
which implies
\begin{eqnarray}\label{presen of phi in flux constrution}
\nonumber
\phi^{n+1}&=&\frac{\epsilon^2}{ \epsilon^2+\Delta t(\beta\sigma)^{n+1}}\phi^n + \frac{\Delta t(\beta\sigma)^{n+1} }{\epsilon^2+\Delta t(\beta\sigma)^{n+1}}\rho^{n+1}\\
&\equiv& \kappa^{n+1}\phi^n + (1-\kappa^{n+1})\rho^{n+1},
\end{eqnarray}
where $\kappa=\frac{\epsilon^2}{\epsilon^2 +\Delta t\beta\sigma}.$
Replacing $\phi$ by $\phi^{n+1}$ in the first equation of \eqref{Rewritten Gray Radiative transfer equaitons}, we
substitute \eqref{presen of phi in flux constrution} into it to arrive at
\begin{equation}\label{new equationn for flux}
\frac{\epsilon^2}{c}\frac{\partial I}{\partial t}+
\epsilon {\bm \Omega}\cdot\nabla_{\bf r} I=\sigma\Big( \frac{1}{4\pi}\left(\kappa^{n+1}\phi^n + (1-\kappa^{n+1})\rho^{n+1}\right)-I\Big).
\end{equation}
In order to get the $x$-direction numerical fluxes $\Phi^{n+1}_{i-1/2,j}$ and $\vec{\bm G}_{i-1/2,j}$, we
 consider the solution of the following initial value problem at cell interface $x=x_{i-\frac{1}{2}}, y=y_j$:
 \begin{equation}\label{inital value problem on boundary}
\left\{\begin{array}{ll}
\displaystyle{\frac{\epsilon}{c}\frac{\partial I}{\partial t}+
\xi \frac{\partial I}{\partial x}=\frac{\sigma^{n+1}_{i-1/2,j}}{\epsilon}\Big( \frac{1}{4\pi}\left(\kappa^{n+1}\phi^n + (1-\kappa_{i-1/2,j}^{n+1})\rho^{n+1}\right)-I\Big),}\\[2mm]
\displaystyle{I(x,y_j,{\bm \Omega},t)|_{t=t_n}=I^0(x,y_j,{\bm \Omega}),}
\end{array}\right.
\end{equation}
where $\sigma^{n+1}_{i-1/2,j}, \kappa_{i-1/2,j}^{n+1}$ are supposed to be approximate constants of $\sigma, \kappa$ around $t_{n+1}$ at the corresponding cell interface.
Then,  the time-dependent solution of \eqref{inital value problem on boundary} can be written as
\begin{eqnarray}\nonumber\label{cell initial value problem solution I}
&& I(x_{i-\frac{1}{2}},y_j,{\bm \Omega},t)\\
&& =e^{-\nu^{n+1}_{i-1/2,j}(t-t_n)}I^0\left(x_{i-1/2}-\frac{c\xi}{\epsilon}(t-t_n), y_j, {\bm \Omega} \right)\\\nonumber
&& \; +\frac{\nu^{n+1}_{i-1/2,j}}{4\pi}
\int_{t_n}^t e^{-\nu^{n+1}_{i-1/2,j}(t-s)}
%\frac{c\sigma^{n+1}_{i-1/2,j}}{4\pi\epsilon^2}
\left(\kappa^{n+1}\phi^n + (1-\kappa_{i-1/2,j}^{n+1})\rho^{n+1}\right)\left(x_{i-1/2}-\frac{c\xi}{\epsilon}(t-s),y_j\right) d s,
\end{eqnarray}
where $\nu=c\sigma/\epsilon^2$.

Now, in order to fully determine the solution in \eqref{cell initial value problem solution I}, it remains to give approximations of the initial data $I^0(x,y_j,{\bm \Omega})$,  $(\kappa^{n+1}\phi^n)(x,y_j)$ and $\rho^{n+1}(x,y_j)$
 around the cell interface $(x_{i-1/2},y_j)$.
In UGKS, they are obtained through polynomial reconstructions. Here, we use the following piecewise linear reconstructions:
\begin{equation}\label{I 0 reconstrution}
I^0(x,y_j,{\bm \Omega})=\vec{\bm \psi}\cdot (\vec{\bm I}_{i,j}^{n}+\delta_x \vec{\bm I}_{i,j}^{n}(x-x_i)), \qquad
%\text{for}
~ x\in( x_{i-1/2}, x_{i+1/2} ),
\end{equation}
\begin{equation}\label{phi reconstruction}
(k^{n+1}\phi^n)(x,y_j)
=(\kappa^{n+1}\phi^n)_{i-1/2,j} + \delta_x(\kappa^{n+1}\phi^n)_{i-1/2,j}(x-x_{i-1/2}), \quad
%\text{for}
~ x\in( x_{i-1}, x_{i} ),
\end{equation}
\begin{equation}\label{rho reconstrution}
\rho^{n+1}(x,y_j)
=\rho^{n+1}_{i-1/2,j} +\delta_x\rho_{i-1/2,j}^{n+1}(x-x_{i-1/2}), \qquad
%\text{for}
~ x\in( x_{i-1}, x_{i} ),
\end{equation}
where $\delta_x \vec{\bm I}=(\delta_x I_0^0, \delta_x I_1^{-1},...,\delta_x I_N^{N})'$ are the slopes,  and the second order MUSCL slope limiter is used to remove
(possible) spurious oscillations. And $\varphi_{i-1/2,j}, \delta_x\varphi_{i-1/2,j}$ (for $\varphi = \kappa^{n+1}\phi^n$ or $\rho^{n+1}$)
are the cell interface value and the finite difference approximation respectively, which are given by
\begin{equation}\nonumber
\varphi_{i-1/2,j} =\frac{\varphi_{i,j}+ \varphi_{i-1,j} }{2},
\quad\qquad\qquad
 \delta_x\varphi_{i-1/2,j}=\frac{\varphi_{i,j} - \varphi_{i-1,j} }{\Delta x}.
\end{equation}
We remark here that since the piecewise linear reconstructions have been used for the spatial variable,
the scheme is of second-order accuracy in space.
%% Moreover, the temporal variable reconstruction in \eqref{phi of initial value problem} is also piecewise constant.
%This means that the scheme is of first-order accuracy
%in the temporal variable.
We shall verify the spatial second-order accuracy numerically in Section \ref{section: Numerical experiments}.

Substituting \eqref{I 0 reconstrution}-%\eqref{phi reconstruction} and
\eqref{rho reconstrution} into \eqref{cell initial value problem solution I},
we obtain the approximation formula of the interface intensity $I(x_{i-1/2},y_j,{\bm \Omega},t)$ that depends on the macro quantities $\rho$ and $\phi$. %at time $t_{n+1}$.

With the above obtained approximation of $I(x_{i-1/2},y_j,{\bm \Omega},t)$ in \eqref{cell initial value problem solution I},
the expressions for the numerical fluxes $\Phi^{n+1}_{i-1/2,j}$ and $\vec{\bm G}_{i-1/2,j}$  can be computed, which  are
%are computed as
\begin{eqnarray}\label{flux Phi}\nonumber
\Phi^{n+1}_{i-1/2,j}
&=&
 \frac{c}{\epsilon \Delta t}
\int_{t_n}^{t_{n+1}} {\bm\langle}\xi I(x_{i-\frac{1}{2}},y_{j},{\bm \Omega},t) {\bm\rangle} dt\\\nonumber
&=&
\tilde{\alpha}^{n+1}_{i-1/2,j}\left\{ {\bm\langle}\xi\vec{\bm \psi} {\bm\rangle}_{1,4}\cdot \left(\vec{\bm I}^n_{i-1,j}+
\frac{\Delta x}{2}\delta_x\vec{\bm I}^n_{i-1,j}\right)
+
{\bm\langle}\xi\vec{\bm \psi} {\bm\rangle}_{2,3}\cdot\left(\vec{\bm I}^n_{i,j}
-\frac{\Delta x}{2}\delta_x\vec{\bm I}^n_{i,j}\right)
\right\}\\\nonumber
& &
+ \tilde{b}^{n+1}_{i-1/2,j}\left(
{\bm\langle}\xi^2\vec{\bm \psi} {\bm\rangle}_{1,4}\cdot\delta_x\vec{\bm I}^n_{i-1,j}
+{\bm\langle}\xi^2\vec{\bm \psi}{\bm\rangle}_{2,3}\cdot\delta_x\vec{\bm I}^n_{i,j}
\right)\\
& &
+\tilde{d}^{n+1}_{i-1/2,j}\frac{4\pi}{3}
\left(\delta_x (\kappa^{n+1}\phi^n)_{i-1/2,j} +(1-\kappa_{i-1/2,j}^{n+1})\delta_x\rho_{i-1/2,j}^{n+1} \right),
\end{eqnarray}
\begin{eqnarray}\label{flux G}\nonumber
\vec{\bf G}_{i-1/2,j}
&=&
\frac{c}{\epsilon \Delta t}
\int_{t_n}^{t_{n+1}} {\bm\langle}\xi \breve{\vec{\bm \psi}} I(x_{i-\frac{1}{2}},y_j,{\bm \Omega},t) {\bm\rangle} dt\\\nonumber
&=&
\tilde{\alpha}^{n+1}_{i-1/2,j}
\left\{ {\bm\langle}\xi\breve{\vec{\bm \psi}}\vec{\bm \psi}' {\bm\rangle}_{1,4}\cdot \left(\vec{\bm I}^n_{i-1,j}+
\frac{\Delta x}{2}\delta_x\vec{\bm I}^n_{i-1,j}\right)+ {\bm\langle}\xi\breve{\vec{\bm \psi}}\vec{\bm \psi}' {\bm\rangle}_{2,3}\cdot\left(\vec{\bm I}^n_{i,j}
-\frac{\Delta x}{2}\delta_x\vec{\bm I}^n_{i,j}\right) \right\}\\\nonumber
& &
+ \tilde{b}^{n+1}_{i-1/2,j}\left(
{\bm\langle}\xi^2\breve{\vec{\bm \psi}}\vec{\bm \psi}' {\bm\rangle}_{1,4}\cdot\delta_x\vec{\bm I}^n_{i-1,j}
+{\bm\langle}\xi^2\breve{\vec{\bm \psi}}\vec{\bm \psi}' {\bm\rangle}_{2,3}\cdot\delta_x\vec{\bm I}^n_{i,j}
\right)\\\nonumber
& &
+\tilde{c}^{n+1}_{i-1/2,j}{\bm\langle}\xi\breve{\vec{\bm \psi}} {\bm\rangle}~
\left( (\kappa^{n+1}\phi^n)_{i-1/2,j} +(1-\kappa_{i-1/2,j}^{n+1})\rho_{i-1/2,j}^{n+1} \right) \\
& &
+ \tilde{d}^{n+1}_{i-1/2,j}{\bm\langle}\xi^2\breve{\vec{\bm \psi}} {\bm\rangle}
\left(\delta_x (\kappa^{n+1}\phi^n)_{i-1/2,j} +(1-\kappa_{i-1/2,j}^{n+1})\delta_x\rho_{i-1/2,j}^{n+1} \right),
\end{eqnarray}
where the coefficients $\tilde{\alpha}^{n+1}_{i-1/2,j}, \tilde{b}^{n+1}_{i-1/2,j},
\tilde{c}^{n+1}_{i-1/2,j},
 \tilde{d}^{n+1}_{i-1/2,j}$ are given by
\begin{eqnarray}\label{Definition of coefficients}\nonumber
&&\tilde{\alpha}(\Delta t, \epsilon, \sigma)=\frac{c}{\epsilon \Delta t \nu}(1-e^{-\nu\Delta t}),\\\nonumber
&&
\tilde{b}(\Delta t, \epsilon, \sigma)= -\frac{c^2}{ \epsilon^2 \nu^2 \Delta t}\left(1-e^{-\nu\Delta t}-\nu\Delta te^{-\nu\Delta t} \right), \\
&&
\tilde{c}(\Delta t, \epsilon, \sigma)=\frac{c^2\sigma}{4\pi \Delta t \epsilon^3 \nu}\left(\Delta t-\frac{1}{\nu}(1-e^{-\nu\Delta t})\right),\\\nonumber
&&
\tilde{d}(\Delta t, \epsilon, \sigma)=-\frac{c^3\sigma}{4\pi \Delta t \epsilon^4 \nu^2}\left(\Delta t( 1+e^{-\nu\Delta t})-\frac{2}{\nu}(1-e^{-\nu\Delta t})\right),
\end{eqnarray}
with $\sigma_{i-1/2,j}^{n+1}=\frac{2\sigma_{i,j}^{n+1}\sigma_{i-1,j}^{n+1}} {\sigma_{i,j}^{n+1}+\sigma_{i-1,j}^{n+1}}$ on the boundary. And
${\bm\langle} \cdot {\bm\rangle}_{1,4}$ denotes the angular integration over the  $\xi>0$  region of the unit sphere,
while ${\bm\langle} \cdot {\bm\rangle}_{2,3}$ over the $\xi<0$  region.

Similarly,  the $y$-direction fluxes $\Upsilon^{n+1}_{i,j-1/2}$ and $\vec{\bm H}_{i,j-1/2}$ can be constructed, and take
the following form:
\begin{eqnarray}\label{flux Upsilon}\nonumber
\Upsilon^{n+1}_{i,j-1/2}
%&=& \frac{c}{\epsilon \Delta t}
%\int_{t_n}^{t_{n+1}} {\bm\langle}\eta I(x_{i},y_{j-\frac{1}{2}},{\bm \Omega},t) {\bm\rangle} dt\\\nonumber
&=&
\tilde{\alpha}^{n+1}_{i,j-1/2}\left\{ {\bm\langle}\eta\vec{\bm \psi} {\bm\rangle}_{1,2}\cdot\left(\vec{\bm I}^n_{i,j-1}+
\frac{\Delta y}{2}\delta_y\vec{\bm I}^n_{i,j-1}\right)
+
{\bm\langle}\eta\vec{\bm \psi}{\bm\rangle}_{3,4}\cdot \left(\vec{\bm I}^n_{i,j}
-\frac{\Delta y}{2}\delta_y\vec{\bm I}^n_{i,j}\right)
\right\}\\\nonumber
& &
+ \tilde{b}^{n+1}_{i,j-1/2}\left(
{\bm\langle}\eta^2\vec{\bm \psi}{\bm\rangle}_{1,2}\cdot\delta_y\vec{\bm I}^n_{i,j-1}
+{\bm\langle}\eta^2\vec{\bm \psi} {\bm\rangle}_{3,4}\cdot \delta_y\vec{\bm I}^n_{i,j}
\right)\\
& &
+\tilde{d}^{n+1}_{i,j-1/2}\frac{4\pi}{3}
\left(\delta_y (\kappa^{n+1}\phi^n)_{i,j-1/2} +(1-\kappa_{i,j-1/2}^{n+1})\delta_y\rho_{i,j-1/2}^{n+1} \right),
\end{eqnarray}
\begin{eqnarray}\label{flux H}
\nonumber
\vec{\bf H}_{i,j-1/2}
%&=& \frac{c}{\epsilon \Delta t}
%\int_{t_n}^{t_{n+1}} {\bm\langle}\eta \breve{\vec{\bm \psi}} I(x_{i},y_{j-\frac{1}{2}},{\bm \Omega},t) {\bm\rangle} dt\\\nonumber
&=&
\tilde{\alpha}^{n+1}_{i,j-1/2}
\left\{ {\bm\langle}\eta\breve{\vec{\bm \psi}}\vec{\bm \psi}' {\bm\rangle}_{1,2}\cdot\left(\vec{\bm I}^n_{i,j-1}+
\frac{\Delta y}{2}\delta_y\vec{\bm I}^n_{i,j-1}\right)+ {\bm\langle}\eta\breve{\vec{\bm \psi}}\vec{\bm \psi}' {\bm\rangle}_{3,4}\cdot\left(\vec{\bm I}^n_{i,j}
-\frac{\Delta y}{2}\delta_y\vec{\bm I}^n_{i,j}\right) \right\}\\\nonumber
& &
+ \tilde{b}^{n+1}_{i,j-1/2}\left(
{\bm\langle}\eta^2\breve{\vec{\bm \psi}}\vec{\bm \psi}' {\bm\rangle}_{1,2}\cdot\delta_y\vec{\bm I}^n_{i,j-1}
+{\bm\langle}\eta^2\breve{\vec{\bm \psi}}\vec{\bm \psi}' {\bm\rangle}_{3,4}\cdot\delta_y\vec{\bm I}^n_{i,j}
\right)\\\nonumber
& &
+\tilde{c}^{n+1}_{i,j-1/2}{\bm\langle}\eta\breve{\vec{\bm \psi}} {\bm\rangle}~
\left( (k^{n+1}\phi^n)_{i,j-1/2} +(1-k_{i,j-1/2}^{n+1})\rho_{i,j-1/2}^{n+1} \right) \\
& &
+ \tilde{d}^{n+1}_{i,j-1/2}{\bm\langle}\eta^2\breve{\vec{\bm \psi}} {\bm\rangle}
\left(\delta_y (\kappa^{n+1}\phi^n)_{i,j-1/2} +(1-\kappa_{i,j-1/2}^{n+1})\delta_y\rho_{i,j-1/2}^{n+1} \right),
\end{eqnarray}
 where for $ \varphi = \kappa^{n+1}\phi^n $ or $\rho^{n+1}$,
   $\varphi_{i,j-1/2} = \frac{\varphi_{i,j}+ \varphi_{i,j-1} }{2}, \delta_y\varphi_{i,j-1/2}=\frac{\varphi_{i,j} - \varphi_{i,j-1} }{\Delta y}$,
$\sigma_{i,j-1/2}^{n+1}=\frac{2\sigma_{i,j}^{n+1}\sigma_{i,j-1}^{n+1}}
{\sigma_{i,j}^{n+1}+\sigma_{i,j-1}^{n+1}}.$
And  ${\bm\langle} \cdot {\bm\rangle}_{1,2}$ denotes the angular integration over the  $\eta>0$  region of the unit sphere,
while ${\bm\langle} \cdot {\bm\rangle}_{3,4}$ over the $\eta<0$  region.

Up to now, we have completed the construction of the numerical fluxes.

\subsection{\texorpdfstring{$FP_N$}{}-based UGKS for \eqref{Rewritten Gray Radiative transfer equaitons}}
%\subsection{$FP_N$-based UGKS for \eqref{Rewritten Gray Radiative transfer equaitons}}
\label{subsec: loop of $FP_N$-based UGKS}

With the numerical fluxes in hand,
we summarize the whole procedure of the currently constructed $FP_N$-based UGKS in the following.

We note that the nonlinear macro equations \eqref{discrete macro equation} can be solved by the iteration method used in \cite{Sun-Jiang-Xu-2015, Xu-Jiang-Sun-2021}. Moreover, substituting \eqref{presen of phi in flux constrution}
 into the first equation of \eqref{discrete macro equation}, we can obtain the
 decoupled iterative macro system.  Now, we state it in the following:
\vspace{2mm}

{\noindent{\bf Loop of the $FP_N$-based UGKS:}}
Given $\vec{\bm I}^n_{i,j}$, $T_{i,j}^n$, one has $\rho^n_{i,j}$ and $\breve{\vec{\bm I}}^n_{i,j}$.
Find $\vec{\bm I}^{n+1}_{i,j}$ (i.e.,  $\rho^{n+1}_{i,j}$, $\breve{\vec{\bm I}}^{n+1}_{i,j}$) and $T_{i,j}^{n+1}$.

\noindent{ 1) Solve the  nonlinear  macro equations \eqref{discrete macro equation} by the following iteration method to get   $\rho_{i,j}^{n+1}$, $T_{i,j}^{n+1}$:} \\[-3mm]

  {\small
 \noindent{~\quad 1.1)} Set the initial iteration values
  $\rho^{n+1,0}_{i,j}, T^{n+1,0}_{i,j}$ (e.g. take $\rho^{n+1,0}_{i,j} = \rho^n_{i,j}, ~T^{n+1,0}_{i,j} = T^n_{i,j}$);
 % $\rho^{n+1,0}_{i,j} = \rho^n_{i,j}$ and $ \phi^{n+1,0}_{i,j} = \phi^n_{i,j}$;

 \noindent{~\quad 1.2)} For $s=0,1,\cdots ,S$,

   \quad Compute the coefficients $\sigma_{i,j}^{n+1,s}, \beta_{i,j}^{n+1,s}, \kappa_{i,j}^{n+1,s}$, $\tilde{\alpha}^{n+1,s}_{i-1/2,j}, \tilde{b}^{n+1,s}_{i-1/2,j}$,
     $\tilde{d}^{n+1,s}_{i-1/2,j}$ with the value of $ T^{n+1,s}_{i,j}$, and
     solve the following linearized equation to get $\rho_{i,j}^{n+1,s+1}$
\begin{eqnarray}\label{nonlinear iteration of discrete macro eqaution-1-UGKS}
\nonumber
\rho^{n+1,s+1}_{i,j}
&=&\rho^{n}_{i,j}
+\frac{\Delta t}{\Delta x}
\left( \Phi^{n+1,s+1}_{i-1/2,j}-\Phi^{n+1,s+1}_{i+1/2,j} \right)
+\frac{\Delta t}{\Delta y}
\left(\Upsilon^{n+1,s+1}_{i,j-1/2}-\Upsilon^{n+1,s+1}_{i,j+1/2} \right)\\
%\\[2mm]
& &\displaystyle{ +
\frac{\sigma_{i,j}^{n+1,s} c \Delta t}
{\epsilon^2}\kappa_{i,j}^{n+1,s}
(\phi^{n}_{i,j}-\rho^{n+1,s+1}_{i,j})},
\end{eqnarray}
where
\begin{equation}\nonumber
\begin{array}{ll}
\Phi^{n+1,s+1}_{i-1/2,j}
=\tilde{\alpha}^{n+1,s}_{i-1/2,j}\left\{ {\bm\langle}\xi\vec{\bm \psi} {\bm\rangle}_{1,4}\cdot\left(\vec{\bm I}^n_{i-1,j}+
\frac{\Delta x}{2}\delta_x\vec{\bm I}^n_{i-1,j}\right)
+{\bm\langle}\xi\vec{\bm \psi} {\bm\rangle}_{2,3}\cdot\left(\vec{\bm I}^n_{i,j}
-\frac{\Delta x}{2}\delta_x\vec{\bm I}^n_{i,j}\right)
\right\} \\
\qquad\qquad\quad ~ +\tilde{b}^{n+1,s}_{i-1/2,j}\left(
{\bm\langle}\xi^2\vec{\bm \psi}{\bm\rangle}_{1,4}\cdot\delta_x\vec{\bm I}^n_{i-1,j}
+{\bm\langle}\xi^2\vec{\bm \psi} {\bm\rangle}_{2,3}\cdot\delta_x\vec{\bm I}^n_{i,j}
\right)
 \\ \qquad\qquad\quad ~ +\tilde{d}^{n+1,s}_{i-1/2,j}\frac{4\pi}{3}
\left(\delta_x (\kappa^{n+1,s}\phi^n)_{i-1/2,j} +(1-\kappa_{i-1/2,j}^{n+1,s})\delta_x\rho_{i-1/2,j}^{n+1,s+1} \right),
\end{array}
\end{equation}
and similarly  for $\Upsilon^{n+1,s+1}_{i,j-1/2}$.

  Substitute $\rho_{i,j}^{n+1,s+1}$ into the following equation to get $\phi_{i,j}^{n+1,s+1}$:
\begin{equation}\label{nonlinear iteration of discrete macro eqaution-2-UGKS}
\displaystyle{\phi^{n+1,s+1}_{i,j}=
\kappa_{i,j}^{n+1,s}\phi^n_{i,j} + (1-\kappa_{i,j}^{n+1,s}) \rho_{i,j}^{n+1,s+1}.}
\end{equation}

\noindent{~\quad 1.3)}
%Compute the relative iteration error, if converges, then stop the iteration and set
 Calculate the iteration error, exit when convergence, and set

$$\rho_{i,j}^{n+1}=\rho_{i,j}^{n+1,s+1}, \quad \phi_{i,j}^{n+1}=\phi_{i,j}^{n+1,s+1},\quad T_{i,j}^{n+1}= (\frac{\phi_{i,j}^{n+1}}{ac})^{\frac{1}{4}}.$$
}

\noindent{ 2) With the obtained $\rho_{i,j}^{n+1}$, $\phi_{i,j}^{n+1}$ and $T_{i,j}^{n+1}$, compute $\sigma_{i,j}^{n+1}$ and the
fluxes $\vec{\bm G}_{i\pm1/2,j}$, $\vec{\bm H}_{i,j\pm1/2}$. Then, $\breve{\vec{\bm I}}^{n+1}_{i,j}$ can be obtained directly from the following linear  micro equations:
\begin{eqnarray}
\nonumber
\breve{\vec{\bm I}}^{n+1}_{i,j} & = & \breve{\vec{\bm I}}^{n}_{i,j} +\frac{\Delta t}{\Delta x}
\left( \vec{\bm G}_{i-1/2,j}-\vec{\bm G}_{i+1/2,j} \right)
+\frac{\Delta t}{\Delta y} \left( \vec{\bm H}_{i,j-1/2}-\vec{\bm H}_{i,j+1/2} \right) \\\nonumber
& & - \frac{\sigma_{i,j}^{n+1}c \Delta t}{\epsilon^2} \breve{\vec{\bm I}}^{n+1}_{i,j}
-\Delta t \sigma_f \breve{{\bf F}}   \breve{\vec{\bm I}}^{n+1}_{i,j}.
\end{eqnarray}
}

{\noindent{ 3) Goto 1) for the next computational time step.}}

{\noindent{\bf End}}

\section{Analysis of the \texorpdfstring{$FP_N$}{}-based UGKS}
%\section{Positive-preserving $FP_N$-based UGKS}

This section is devoted to analyzing the fluxes in the above $FP_N$-based UGKS and
making the scheme be both positive preserving and asymptotic preserving. First, in Subsection \ref{section of pp theorem},
by a detailed analysis of the linearized iteration equations
\eqref{nonlinear iteration of discrete macro eqaution-1-UGKS}-\eqref{nonlinear iteration of discrete macro eqaution-2-UGKS},
we can obtain the sufficient conditions that preserve the positivity of the macro auxiliary quantities $\rho_{i,j}^{n+1,s+1}$ and $\phi_{i,j}^{n+1,s+1}$ in every iteration. This implies the positivity of the
radiation energy density $\rho_{i,j}^{n+1}$ and material temperature $T_{i,j}^{n+1}$. Then, in Subsection \ref{section of limiter},  linear scaling limiters, which are realizable, are given to enforce those sufficient conditions to hold.
Finally, we analyze the asymptotic preserving property of the scheme
%of the $PPFP_N$-based UGKS
 in Subsection \ref{asymptotic analysis}. Thus, the desired positive preserving and  asymptotic preserving  $FP_N$-based
UGKS, called $PPFP_N$-based UGKS, is obtained.

\subsection{Positivity analysis }
\label{section of pp theorem}
For convenience, we  decompose the macro flux $ \Phi^{n+1,s+1}_{i-1/2,j}$  into the following
three parts:
\begin{equation}\nonumber
\Phi^{n+1,s+1}_{i-1/2,j} = \Phi^{n+1,s+1}_{i-1/2,j}(\vec{\bm I}^n) + \Phi^{n+1,s+1}_{i-1/2,j}(\phi^n)
+\Phi^{n+1,s+1}_{i-1/2,j}(\rho^{n+1,s+1}),
\end{equation}
where $ \Phi^{n+1,s+1}_{i-1/2,j}(\vec{\bm I}^n), \Phi^{n+1,s+1}_{i-1/2,j}(\phi^n),
\Phi^{n+1,s+1}_{i-1/2,j}(\rho^{n+1})$ denote the terms
that are related to $\vec{\bm I}^n$, $\phi^n$ and $\rho^{n+1,s+1}$ in the macro flux $\Phi^{n+1,s+1}_{i-1/2,j}$ of
\eqref{nonlinear iteration of discrete macro eqaution-1-UGKS} respectively. And the other macro fluxes, such as $ \Upsilon^{n+1,s+1}_{i,j-1/2}$,
can be handled similarly. Next, we prove a sufficient condition that makes the solution of the linearized equations
\eqref{nonlinear iteration of discrete macro eqaution-1-UGKS}-\eqref{nonlinear iteration of discrete macro eqaution-2-UGKS}
positive. More precisely,

\begin{thm}\label{positive preserving theorem}
Given $\rho^{n}_{i,j}\geq 0$, $\phi_{i,j}^n \geq 0$ and  $\phi_{i,j}^{n+1,s} \geq 0$.
Suppose that $\Delta t$ satisfies
\begin{equation}\label{time step constraints}
\hspace{-2mm}\Delta t \tilde{\alpha}^{n+1,s}(\Delta t,\epsilon, \sigma_{\min} ) %\min_{i,j}\sigma^{n+1,s}_{i,j})
\leq
\frac{\Delta x\Delta y}{2(\Delta x+\Delta y)},~
\max_{i,j}\left(-\frac{\tilde{d}^{n+1,s}_{i\pm 1/2,j\pm 1/2
%~\text{or} ~i\pm 1/2,j
}(\Delta t) }{\sigma_{i,j}^{n+1,s}}
\right)\leq
\frac{ 3c (\Delta x \Delta y)^2 }{ 8\pi \epsilon^2 \left( (\Delta x)^2+(\Delta y)^2 \right) },
\end{equation}
where $\sigma_{\min}=\min_{i,j}\sigma^{n+1,s}_{i,j}$. Furthermore, for any cell $(i,j)$, assume that the initial data
$\vec{\bm I}^n_{i,j}=(\frac{1}{2\sqrt{\pi}}\rho^n_{i,j}, ~\breve{\vec{\bm I}}_{i,j}^{n'})'$ satisfies
\begin{eqnarray}\label{PN reconstruciotn conditon in pp thm}
\nonumber
& &
\frac{1}{4}
\left(\frac{\tilde{\alpha}^{n+1,s}_{i-1/2,j}}{\Delta x}+\frac{\tilde{\alpha}^{n+1,s}_{i+1/2,j}}{\Delta x}
+\frac{\tilde{\alpha}^{n+1,s}_{i,j-1/2}}{\Delta y}+\frac{\tilde{\alpha}^{n+1,s}_{i,j+1/2}}{\Delta y}\right)
\rho^n_{i,j}
+
\left\{
\frac{\tilde{\alpha}^{n+1,s}_{i-1/2,j}}{\Delta x}
{\bm\langle}\xi\breve{\vec{\bm \psi}} {\bm\rangle}_{2,3}
\right.\\
&  & ~\quad %\qquad~
\left.
-\frac{\tilde{\alpha}^{n+1,s}_{i+1/2,j}}{\Delta x} {\bm\langle}\xi\breve{\vec{\bm \psi}} {\bm\rangle}_{1,4}
+ \frac{\tilde{\alpha}^{n+1,s}_{i,j-1/2}}{\Delta y}
{\bm\langle}\eta\breve{\vec{\bm \psi}}{\bm\rangle}_{3,4}
-\frac{\tilde{\alpha}^{n+1,s}_{i,j+1/2}}{\Delta y} {\bm\langle}\eta\breve{\vec{\bm \psi}} {\bm\rangle}_{1,2} \right\}
\breve{\vec{\bm I}}^n_{i,j}
\geq 0,
\end{eqnarray}
and
\begin{equation}\label{slopex I condition in pp thm}\left\{
\begin{array}{ll}
{\bm\langle}\xi\vec{\bm\psi}{\bm\rangle}_{1,4}\cdot\vec{\bm I}^n_{i,j}\pm{\bm\langle}(\frac{\Delta x}{2}\xi + \frac{\tilde{b}^{n+1,s}_{i+1/2,j}}{\tilde{\alpha}^{n+1,s}_{i+1/2,j}}\xi^2)%{i+1/2,j}
\vec{\bm \psi}{\bm\rangle}_{1,4}\cdot
 %\frac{\Delta x}{2}
 \delta_x\vec{\bm I}^n_{i,j}\geq 0,\\
 {\bm\langle}-\xi\vec{\bm\psi}{\bm\rangle}_{2,3}\cdot\vec{\bm I}^n_{i,j}
\pm {\bm\langle}(\frac{\Delta x}{2}\xi + \frac{\tilde{b}^{n+1,s}_{i-1/2,j}}{\tilde{\alpha}^{n+1,s}_{i-1/2,j}}\xi^2)\vec{\bm\psi}{\bm\rangle}_{2,3}\cdot%\frac{\Delta x}{2}
\delta_x\vec{\bm I}^n_{i,j}\geq 0,
\end{array}\right.
\end{equation}
\begin{equation}\label{slopey I condition in pp thm}\left\{
\begin{array}{ll}
 {\bm\langle}\eta\vec{\bm\psi}{\bm\rangle}_{1,2}\cdot\vec{\bm I}^n_{i,j}\pm{\bm\langle}
 (\frac{\Delta y}{2}\eta + \frac{\tilde{b}^{n+1,s}_{i,j+1/2}}{\tilde{\alpha}^{n+1,s}_{i,j+1/2}}\eta^2)\vec{\bm \psi}{\bm\rangle}_{1,2}\cdot
 %\frac{\Delta x}{2}
 \delta_x\vec{\bm I}^n_{i,j}\geq 0,\\
 {\bm\langle}-\eta\vec{\bm\psi}{\bm\rangle}_{3,4}\cdot\vec{\bm I}^n_{i,j}
\pm {\bm\langle}(\frac{\Delta y}{2}\eta + \frac{\tilde{b}^{n+1,s}_{i,j-1/2}}{\tilde{\alpha}^{n+1,s}_{i,j-1/2}}\eta^2)\vec{\bm\psi}{\bm\rangle}_{3,4}\cdot%\frac{\Delta x}{2}
\delta_y\vec{\bm I}^n_{i,j}\geq 0.
\end{array}\right.
\end{equation}
Then, it holds that
 $$\rho^{n+1,s+1}_{i,j}\geq 0,\qquad\qquad \phi^{n+1,s+1}_{i,j}\geq 0.$$
\end{thm}
\begin{proof}
We first give the proof of $\rho^{n+1,s+1}_{i,j} \geq 0$. Once it is obtained, by \eqref{nonlinear iteration of discrete macro eqaution-2-UGKS}
and the condition $\phi_{i,j}^n \geq 0$, and the fact $\kappa=\frac{\epsilon^2}{\epsilon^2+\Delta t \beta\sigma}\in [0,1]$,
we obtain $\phi^{n+1,s+1}_{i,j}\geq 0$ immediately.

Denote by
$$\vec{\bm\rho}^{~s+1}=(\rho^{n+1,s+1}_{1,1},\rho^{n+1,s+1}_{1,2},...,\rho^{n+1,s+1}_{2,1},\rho^{n+1,s+1}_{2,2},...)' $$
the solution of the discrete linearized macro equation \eqref{nonlinear iteration of discrete macro eqaution-1-UGKS}.
 Thus we can rewrite equation \eqref{nonlinear iteration of discrete macro eqaution-1-UGKS} into the following matrix form
\begin{equation}\nonumber
 {\bm A}^s\vec{\bm{\rho}}^{~s+1} = \vec{\bm b}^{s}.
 \end{equation}
Here ${\bm A}^s$ is the coefficient matrix of \eqref{nonlinear iteration of discrete macro eqaution-1-UGKS} and $\vec{\bm b}^{s}$ denotes
the right-hand term.
For the matrix ${\bm A}^s$, one sees that for any $\Delta t>0$, $\epsilon\geq0, \sigma\geq0$, the coefficient function
$\tilde{d}(\Delta t,\epsilon, \sigma)\leq 0$ and $0 \leq \kappa \leq 1$. Then, one easily verifies that
 \begin{equation}\nonumber
 {\bm A}^s(p,q)\leq 0~~ \forall p\neq q, \quad\quad \text{and} \quad\quad \sum_{q}{\bm A}^s(p,q)>0~~ \forall p.
 \end{equation}
 Hence, ${\bm A}^s$ is an $M$-matrix. This implies that $({\bm A}^s)^{-1}$ is non-negative. Therefore,
 if $\vec{\bm b}^{s}\geq 0$, then it holds that $ \vec{\bm{\rho}}^{~s+1} = ({\bm A}^s)^{-1}\vec{\bm b}^{s} \geq 0$.

 Now, in order to get $\rho^{n+1,s+1}_{i,j} \geq 0$, it suffices to show $\vec{\bm b}^{s} \geq 0$, i.e., $b_{i,j}^s \geq 0$, $\forall\, i,j$.
For convenience, we decompose $b_{i,j}^s \geq 0$ into the following two parts:
\begin{equation}\nonumber
b_{i,j}^s =  b_{i,j}^s( \vec{\bm I}^n ) + b_{i,j}^s(\phi^n),
\end{equation}
with
\begin{eqnarray}\nonumber
 \qquad b_{i,j}^s(\vec{\bm I}^n)
&=&
\rho^{n}_{i,j}+
\frac{\Delta t}{\Delta x}
\left( \Phi^{n+1,s+1}_{i-1/2,j}(\vec{\bm I}^n)
-\Phi^{n+1,s+1}_{i+1/2,j}(\vec{\bm I}^n)
\right) + \\ \nonumber
& \ \ \ \ \ & \frac{\Delta t}{\Delta y}
\left( \Upsilon^{n+1,s+1}_{i,j-1/2}(\vec{\bm I}^n)
-\Upsilon^{n+1,s+1}_{i,j+1/2}(\vec{\bm I}^n)
\right), \\\nonumber
%\end{eqnarray}
%\begin{eqnarray}\nonumber
 \qquad b_{i,j}^s(\phi^n)
 &=&
\frac{\Delta t}{\Delta x}
\left( {\Phi}^{n+1,s+1}_{i-1/2,j}(\phi^n)-{\Phi}^{n+1,s+1}_{i+1/2,j}(\phi^n) \right)
+\\\nonumber
& \ \ \ \ & \frac{\Delta t}{\Delta y}
\left({\Upsilon}^{n+1,s+1}_{i,j-1/2}(\phi^n)-{\Upsilon}^{n+1,s+1}_{i,j+1/2} (\phi^n)\right)+
\\\nonumber
&\ \ \ \ &
\frac{\sigma_{i,j}^{n+1,s} c \Delta t}{\epsilon^2} \kappa_{i,j}^{n+1,s}\phi_{i,j}^n.
\end{eqnarray}
Next, we show that
$b_{i,j}^s(\vec{\bm I}^n)\geq 0$ and $b_{i,j}^s(\phi^n)\geq 0$.

First, denotes that
\begin{eqnarray}\nonumber
\Phi^{n+1,s+1}_{i-1/2,j}(\vec{\bm I}^n)
&=&
\tilde{\alpha}^{n+1,s}_{i-1/2,j}\left\{ {\bm\langle}\xi\vec{\bm \psi} {\bm\rangle}_{1,4}\cdot \left(\vec{\bm I}^n_{i-1,j}+
\frac{\Delta x}{2}\delta_x\vec{\bm I}^n_{i-1,j}\right)
+
{\bm\langle}\xi\vec{\bm \psi} {\bm\rangle}_{2,3}\cdot\left(\vec{\bm I}^n_{i,j}
-\frac{\Delta x}{2}\delta_x\vec{\bm I}^n_{i,j}\right)
\right\}\\\nonumber
& &
+ \tilde{b}^{n+1,s}_{i-1/2,j}\left(
{\bm\langle}\xi^2\vec{\bm \psi} {\bm\rangle}_{1,4}\cdot\delta_x\vec{\bm I}^n_{i-1,j}
+{\bm\langle}\xi^2\vec{\bm \psi}{\bm\rangle}_{2,3}\cdot\delta_x\vec{\bm I}^n_{i,j}
\right).
\end{eqnarray}
And the expressions of %for
$\Phi^{n+1,s+1}_{i+1/2,j}(\vec{\bm I}^n), \Upsilon^{n+1,s+1}_{i,j-1/2}(\vec{\bm I}^n)
$ and $\Upsilon^{n+1,s+1}_{i,j+1/2}(\vec{\bm I}^n)$ can be defined similarly.
%can be obtained in the same manner as for $\Phi^{n+1,s+1}_{i-1/2,j}(\vec{\bm I}^n)$.
For $\Delta t>0, \epsilon\geq0$ and $\sigma\geq 0$, the coefficient function
$\tilde{\alpha}(\Delta t,\epsilon, \sigma^{n+1,s}_{i,j})\geq 0$. Applying \eqref{slopex I condition in pp thm} directly, one has
\begin{eqnarray}\nonumber
 \Phi^{n+1,s+1}_{i-1/2,j}(\vec{\bm I}^n) &\geq&
\tilde{\alpha}^{n+1,s}_{i-1/2,j}
{\bm\langle}\xi\vec{\bm \psi}{\bm\rangle}_{2,3}\cdot\left(\vec{\bm I}^n_{i,j}
-\frac{\Delta x}{2}\delta_x\vec{\bm I}^n_{i,j}\right)
\\ \nonumber & &
+\tilde{b}^{n+1,s}_{i-1/2,j}
{\bm\langle}\xi^2\vec{\bm \psi}{\bm\rangle}_{2,3}\cdot\delta_x\vec{\bm I}^n_{i,j},\\\nonumber
-\Phi^{n+1,s+1}_{i+1/2,j}(\vec{\bm I}^n) &\geq&
-\tilde{\alpha}^{n+1,s}_{i+1/2,j} {\bm\langle}\xi\vec{\bm \psi} {\bm\rangle}_{1,4}\cdot\left(\vec{\bm I}^n_{i,j}+
\frac{\Delta x}{2}\delta_x\vec{\bm I}^n_{i,j}\right)
 \\ \nonumber & &
 -\tilde{b}^{n+1,s}_{i+1/2,j}
{\bm\langle}\xi^2\vec{\bm \psi} {\bm\rangle}_{1,4} \cdot \delta_x\vec{\bm I}^n_{i,j}.
\end{eqnarray}
%Similar for  $\Upsilon^{n+1,s+1}_{i\pm1/2,j}(\vec{\bm I}^n)$.
%we obtain
%\begin{eqnarray}\nonumber
%&~&b_{i,j}^s(\vec{\bm I}^n)\\\nonumber %\hspace{-4mm}
%&\geq &\rho^{n}_{i,j}+
%\frac{\Delta t}{\Delta x}
%\left\{\tilde{\alpha}_{i-1/2,j}
%{\bm\langle}\xi\vec{\bm \psi}{\bm\rangle}_{2,3}\cdot\left(\vec{\bm I}^n_{i,j}
%-\frac{\Delta x}{2}\delta_x\vec{\bm I}^n_{i,j}\right)
%-\tilde{\alpha}_{i+1/2,j} {\bm\langle}\xi\vec{\bm \psi} {\bm\rangle}_{1,4}\cdot\left(\vec{\bm I}^n_{i,j}+
%\frac{\Delta x}{2}\delta_x\vec{\bm I}^n_{i,j}\right)
% \right\}
%\\\nonumber
%&~&
%+ \frac{\Delta t}{\Delta y}
%\left\{
%\tilde{\alpha}_{i,j-1/2}
%{\bm\langle}\eta\vec{\bm \psi}{\bm\rangle}_{3,4}\cdot\left(\vec{\bm I}^n_{i,j}
%-\frac{\Delta y}{2}\delta_y\vec{\bm I}^n_{i,j}\right)
%  -\tilde{\alpha}_{i,j+1/2} {\bm\langle}\eta\vec{\bm \psi} {\bm\rangle}_{1,2}\cdot \left(\vec{\bm I}^n_{i,j}+
%\frac{\Delta y}{2}\delta_y\vec{\bm I}^n_{i,j}\right)
% \right\}.
%\end{eqnarray}
Also by virtue of the conditions in \eqref{slopex I condition in pp thm}, it holds that
\begin{eqnarray}\nonumber
&&-{\bm\langle}\xi\vec{\bm \psi}{\bm\rangle}_{2,3}\cdot\frac{\Delta x}{2}\delta_x\vec{\bm I}^n_{i,j}
+\frac{\tilde{b}^{n+1,s}_{i-1/2,j}}{\tilde{\alpha}^{n+1,s}_{i-1/2,j}}
{\bm\langle}\xi^2\vec{\bm \psi}{\bm\rangle}_{2,3}\cdot\delta_x\vec{\bm I}^n_{i,j}\geq
{\bm\langle}\xi\vec{\bm \psi}{\bm\rangle}_{2,3}\cdot\vec{\bm I}^n_{i,j},
\\\nonumber
&&-{\bm\langle}\xi\vec{\bm \psi} {\bm\rangle}_{1,4}\cdot
\frac{\Delta x}{2}\delta_x\vec{\bm I}^n_{i,j}
 - \frac{\tilde{b}^{n+1,s}_{i+1/2,j}}{ \tilde{\alpha}^{n+1,s}_{i+1/2,j}  }
{\bm\langle}\xi^2\vec{\bm \psi} {\bm\rangle}_{1,4} \cdot \delta_x\vec{\bm I}^n_{i,j}\geq
- {\bm\langle}\xi\vec{\bm \psi}{\bm\rangle}_{1,4}\cdot\vec{\bm I}^n_{i,j}.
\end{eqnarray}
Therefore,
$$
\Phi^{n+1,s+1}_{i-1/2,j}(\vec{\bm I}^n) \geq
2\tilde{\alpha}^{n+1,s}_{i-1/2,j}
{\bm\langle}\xi\vec{\bm \psi}{\bm\rangle}_{2,3}\cdot\vec{\bm I}^n_{i,j},\quad
-\Phi^{n+1,s+1}_{i+1/2,j}(\vec{\bm I}^n) \geq
-2\tilde{\alpha}^{n+1,s}_{i+1/2,j} {\bm\langle}\xi\vec{\bm \psi} {\bm\rangle}_{1,4}\cdot\vec{\bm I}^n_{i,j}.
$$
Similarly, by the conditions in \eqref{slopey I condition in pp thm}, we can obtain
$$
\Upsilon^{n+1,s+1}_{i,j-1/2}(\vec{\bm I}^n) \geq
2\tilde{\alpha}^{n+1,s}_{i,j-1/2}
{\bm\langle}\eta\vec{\bm \psi}{\bm\rangle}_{3,4}\cdot\vec{\bm I}^n_{i,j},\quad
-\Upsilon^{n+1,s+1}_{i,j+1/2}(\vec{\bm I}^n) \geq
-2\tilde{\alpha}^{n+1,s}_{i,j+1/2} {\bm\langle}\eta\vec{\bm \psi} {\bm\rangle}_{1,2}\cdot\vec{\bm I}^n_{i,j}.
$$
%for  $\Upsilon^{n+1,s+1}_{i,j\pm1/2}(\vec{\bm I}^n)$.
%Similarly  for  $\delta_y\vec{\bm I}^n_{i,j}$ .  Thus,
Thus,
\begin{eqnarray}\nonumber
b_{i,j}^s(\vec{\bm I}^n)
&\geq&
\rho^{n}_{i,j}+
\frac{\Delta t}{\Delta x}
\left\{ 2 \tilde{\alpha}^{n+1,s}_{i-1/2,j}
{\bm\langle}\xi\vec{\bm \psi}{\bm\rangle}_{2,3}\cdot\vec{\bm I}^n_{i,j}
- 2 \tilde{\alpha}^{n+1,s}_{i+1/2,j} {\bm\langle}\xi\vec{\bm \psi} {\bm\rangle}_{1,4}\cdot \vec{\bm I}^n_{i,j}
 \right\}
\\\nonumber
&~&
+ \frac{\Delta t}{\Delta y}
\left\{
2 \tilde{\alpha}^{n+1,s}_{i,j-1/2}
{\bm\langle}\eta\vec{\bm \psi}{\bm\rangle}_{3,4}\cdot\vec{\bm I}^n_{i,j}
  -2\tilde{\alpha}^{n+1,s}_{i,j+1/2} {\bm\langle}\eta\vec{\bm \psi} {\bm\rangle}_{1,2}\cdot\vec{\bm I}^n_{i,j}
 \right\}.
\end{eqnarray}
 Noticing that
$\vec{\bm I}_{i,j}^n=(\frac{1}{2\sqrt{\pi}} \rho^{n}_{i,j}, ~\breve{\vec{\bm I}}_{i,j}^{n'})',$
 and
$$
\langle\xi \psi_0^0\rangle_{1,4} = \frac{\sqrt{\pi}}{2}, ~\quad
\langle\xi \psi_0^0\rangle_{2,3} = -\frac{\sqrt{\pi}}{2},~\quad
\langle\eta \psi_0^0\rangle_{1,2} = \frac{\sqrt{\pi}}{2}, ~\quad
\langle\eta \psi_0^0\rangle_{3,4} = -\frac{\sqrt{\pi}}{2}, $$
we can rewrite $b_{i,j}^s(\vec{\bm I}^n)$ as
\begin{eqnarray}\nonumber
\hspace{-2mm}b_{i,j}^s(\vec{\bm I}^n)
&\geq&
 \rho_{i,j}^n
- \frac{\Delta t}{2}\left( \frac{\tilde{\alpha}^{n+1,s}_{i-1/2,j}}{\Delta x}
+\frac{\tilde{\alpha}^{n+1,s}_{i+1/2,j}}{\Delta x}
+\frac{\tilde{\alpha}^{n+1,s}_{i,j-1/2}}{\Delta y}
+\frac{\tilde{\alpha}^{n+1,s}_{i,j+1/2}}{\Delta y}
 \right)\rho_{i,j}^n
 \\\nonumber
&~& \hspace{-2.5mm}% \qquad
+2\Delta t
\left\{  \frac{\tilde{\alpha}^{n+1,s}_{i-1/2,j}}{\Delta x}
 {\bm\langle}\xi\breve{\vec{\bm \psi}} {\bm\rangle}_{2,3}
-  \frac{\tilde{\alpha}^{n+1,s}_{i+1/2,j}}{\Delta x}{\bm\langle}\xi\breve{\vec{\bm \psi}} {\bm\rangle}_{1,4}
+ \frac{\tilde{\alpha}^{n+1,s}_{i,j-1/2}}{\Delta y}
{\bm\langle}\eta\breve{\vec{\bm \psi}} {\bm\rangle}_{3,4}
  -\frac{\tilde{\alpha}^{n+1,s}_{i,j+1/2}}{\Delta y} {\bm\langle}\eta\breve{\vec{\bm \psi}} {\bm\rangle}_{1,2}
 \right\}\cdot\breve{\vec{\bm I}}^n_{i,j}  \\\nonumber
&\geq&
 \rho_{i,j}^n
- \Delta t\left( \frac{\tilde{\alpha}^{n+1,s}_{i-1/2,j}}{\Delta x}
+\frac{\tilde{\alpha}^{n+1,s}_{i+1/2,j}}{\Delta x}
+\frac{\tilde{\alpha}^{n+1,s}_{i,j-1/2}}{\Delta y}
+\frac{\tilde{\alpha}^{n+1,s}_{i,j+1/2}}{\Delta y}
 \right)\rho_{i,j}^n
 \\\nonumber
& &\hspace{-2.5mm}
+\frac{\Delta t}{2}\left( \frac{\tilde{\alpha}^{n+1,s}_{i-1/2,j}}{\Delta x}
+\frac{\tilde{\alpha}^{n+1,s}_{i+1/2,j}}{\Delta x}
+\frac{\tilde{\alpha}^{n+1,s}_{i,j-1/2}}{\Delta y}
+\frac{\tilde{\alpha}^{n+1,s}_{i,j+1/2}}{\Delta y}
 \right)\rho_{i,j}^n
 \\\nonumber
& & \hspace{-2.5mm}% \qquad
+2\Delta t
\left\{  \frac{\tilde{\alpha}^{n+1,s}_{i-1/2,j}}{\Delta x}
 {\bm\langle}\xi\breve{\vec{\bm \psi}} {\bm\rangle}_{2,3}
-  \frac{\tilde{\alpha}^{n+1,s}_{i+1/2,j}}{\Delta x}{\bm\langle}\xi\breve{\vec{\bm \psi}} {\bm\rangle}_{1,4}
+ \frac{\tilde{\alpha}^{n+1,s}_{i,j-1/2}}{\Delta y}
{\bm\langle}\eta\breve{\vec{\bm \psi}} {\bm\rangle}_{3,4}
  -\frac{\tilde{\alpha}^{n+1,s}_{i,j+1/2}}{\Delta y} {\bm\langle}\eta\breve{\vec{\bm \psi}} {\bm\rangle}_{1,2}
 \right\}\cdot\breve{\vec{\bm I}}^n_{i,j},
 \end{eqnarray}
which together with \eqref{PN reconstruciotn conditon in pp thm} implies
\begin{eqnarray}\nonumber
 b_{i,j}^s(\vec{\bm I}^n)\geq
\rho_{i,j}^n
- \Delta t\left( \frac{\tilde{\alpha}^{n+1,s}_{i-1/2,j}}{\Delta x}
+\frac{\tilde{\alpha}^{n+1,s}_{i+1/2,j}}{\Delta x}
+\frac{\tilde{\alpha}^{n+1,s}_{i,j-1/2}}{\Delta y}
+\frac{\tilde{\alpha}^{n+1,s}_{i,j+1/2}}{\Delta y}
 \right)\rho_{i,j}^n.
  \end{eqnarray}
Because for $\Delta t>0, \epsilon\geq 0$ and $\sigma\geq0$, the function $\tilde{\alpha}(\Delta t,\epsilon, \sigma)$ is decreasing with
respect to the variable $\sigma$, we see that
$$0<\tilde{\alpha}(\Delta t,\epsilon, \sigma^{n+1,s}_{i,j}) \leq  \tilde{\alpha}(\Delta t,\epsilon,\min_{i,j}\sigma^{n+1,s}_{i,j}).$$
Thus,
 \begin{eqnarray}\nonumber
 b_{i,j}^s(\vec{\bm I}^n) &\geq&
\left(1 -2 \Delta t \frac{\Delta x+\Delta y}{\Delta x\Delta y}
\tilde{\alpha}(\Delta t,\epsilon,\min_{i,j}\sigma^{n+1,s}_{i,j})\right)
\rho_{i,j}^n.
\end{eqnarray}
Form the first constraint condition in \eqref{time step constraints} on time step $\Delta t$, we immediately get
$$ b_{i,j}^s(\vec{\bm I}^n) \geq 0.$$

On the other hand,
\begin{eqnarray}\nonumber
b_{i,j}^s(\phi^n) &=&
%\frac{\Delta t}{\Delta x}
%\left( {\Phi}^{n+1,s+1}_{i-1/2,j}(\phi^n)-{\Phi}^{n+1,s+1}_{i+1/2,j}(\phi^n) \right)
%+\frac{\Delta t}{\Delta y}
%\left({\Upsilon}^{n+1,s+1}_{i,j-1/2}(\phi^n)-{\Upsilon}^{n+1,s+1}_{i,j+1/2} (\phi^n)\right)
%\\\nonumber
%& & +\frac{\sigma_{i,j}^{n+1,s} c \Delta t}{\epsilon^2} \kappa_{i,j}^{n+1,s}\phi_{i,j}^n\\\nonumber
%&=&
\frac{4\pi}{3} \frac{\Delta t}{\Delta x}
\left(
\tilde{d}^{n+1,s}_{i-1/2,j}
\frac{\kappa_{i,j}^{n+1,s}\phi_{i,j}^{n}-\kappa_{i-1,j}^{n+1,s}\phi_{i-1,j}^{n} }{\Delta x}
-\tilde{d}^{n+1,s}_{i+1/2,j}
\frac{\kappa_{i+1,j}^{n+1,s}\phi_{i+1,j}^{n}- \kappa_{i,j}^{n+1,s}\phi_{i,j}^{n} }{\Delta x}\right)\\\nonumber
& &
+\frac{4\pi}{3} \frac{\Delta t}{\Delta y}
\left(
\tilde{d}^{n+1,s}_{i,j-1/2}
\frac{\kappa_{i,j}^{n+1,s}\phi_{i,j}^{n}-\kappa_{i,j-1}^{n+1,s} \phi_{i,j-1}^{n} }{\Delta y}
-\tilde{d}^{n+1,s}_{i,j+1/2}
%k_{i,j+1/2}^{n+1,s}
\frac{\kappa_{i,j+1}^{n+1,s}\phi_{i,j+1}^{n}- \kappa_{i,j}^{n+1,s}\phi_{i,j}^{n} }{\Delta y}\right)
\\\nonumber
& & +\frac{\sigma_{i,j}^{n+1,s} c \Delta t}{\epsilon^2} \kappa_{i,j}^{n+1,s}\phi_{i,j}^n = \Delta t \kappa_{i,j}^{n+1,s} \\\nonumber
& &
\left( \frac{\sigma_{i,j}^{n+1,s} c }{\epsilon^2}
+ \frac{4\pi}{3(\Delta x)^2}
\left(
\tilde{d}^{n+1,s}_{i-1/2,j}
+ \tilde{d}^{n+1,s}_{i+1/2,j}\right)
%\\\nonumber
%& &
+ \frac{4\pi}{3(\Delta y)^2}
\left(
\tilde{d}^{n+1,s}_{i,j-1/2}
+ \tilde{d}^{n+1,s}_{i,j+1/2}
\right)
\right) \phi_{i,j}^{n}\\\nonumber
& &
-\frac{4\pi}{3}\frac{\Delta t}{(\Delta x)^2}
\tilde{d}^{n+1,s}_{i-1/2,j}
\kappa_{i-1,j}^{n+1,s}\phi^n_{i-1,j}-\frac{4\pi}{3}\frac{\Delta t}{(\Delta x)^2}
\tilde{d}^{n+1,s}_{i+1/2,j}
\kappa_{i+1,j}^{n+1,s}\phi^n_{i+1,j}\\\nonumber
& &
-\frac{4\pi}{3}\frac{\Delta t}{(\Delta y)^2}
\tilde{d}^{n+1,s}_{i,j-1/2}
\kappa_{i,j-1}^{n+1,s}\phi^n_{i,j-1}-\frac{4\pi}{3}\frac{\Delta t}{(\Delta y)^2}
\tilde{d}^{n+1,s}_{i,j+1/2}
\kappa_{i,j+1}^{n+1,s}\phi^n_{i,j+1}.
\end{eqnarray}
For $\tilde{d}(\Delta t,\epsilon, \sigma)\leq 0$ and $\kappa=\frac{\epsilon^2}{ \epsilon^2+\Delta t \beta\sigma}\geq 0$,
one has  %Therefore,
\begin{eqnarray}\nonumber
& &b_{i,j}^s(\phi^n)\\\nonumber
&\geq&
\Delta t \kappa_{i,j}^{n+1,s}\left( \frac{\sigma_{i,j}^{n+1,s} c }{\epsilon^2}
+ \frac{4\pi}{3(\Delta x)^2}
\left(
\tilde{d}^{n+1,s}_{i-1/2,j}
+ \tilde{d}^{n+1,s}_{i+1/2,j}\right)
+ \frac{4\pi}{3(\Delta y)^2}
\left(\tilde{d}^{n+1,s}_{i,j-1/2}
+ \tilde{d}^{n+1,s}_{i,j+1/2}
\right)
\right) \phi_{i,j}^{n}
\\\nonumber
&\geq&
\Delta t \kappa_{i,j}^{n+1,s}\sigma_{i,j}^{n+1,s} \left( \frac{ c }{\epsilon^2}
+ \frac{4\pi}{3} \left(\frac{2}{(\Delta x)^2}+\frac{2}{(\Delta y)^2}\right)
\frac{ \min\{\tilde{d}^{ n+1,s}_{i,j\pm 1/2}, \tilde{d}^{ n+1,s}_{i\pm 1/2,j}\} }{\sigma^{n+1,s}_{i,j}}
\right) \phi_{i,j}^{n}.
\end{eqnarray}
From the second time step constraint condition in \eqref{time step constraints}, one obtains
$$ b_{i,j}^s(\phi^n) \geq 0.$$
This completes the proof.
\end{proof}
\begin{rem}\label{remark: time step constraint}
We remark that the conditions (\ref{time step constraints}) on $\Delta t$ are
not restrictive. (i) For the optically thin case ($\epsilon=O(1), \sigma\rightarrow 0$), since $\tilde{\alpha}\to\frac{c}{\epsilon}$,
 the first time step constraint condition in \eqref{time step constraints} can be simplified to
%\begin{equation}\nonumber
$\Delta t \leq \frac{\epsilon}{4c}\cdot\min\{\Delta x, \Delta y\}.$
%\end{equation}
And by the analysis used in \cite[Appendix]{Xu-Jiang-Sun-2022},
the second time step constraint condition in \eqref{time step constraints} can be reduced to
$ \Delta t \leq \frac{\sqrt{6}}{2} \cdot\frac{\epsilon}{c}\cdot \min\{\Delta x, \Delta y \}.$
Thus,  for the optically thin case ($\epsilon=O(1), \sigma\rightarrow 0$), the time step constraint conditions  in \eqref{time step constraints}
reduce to
\begin{equation}\label{time constr after pp theorem}
\Delta t \leq \frac{\epsilon}{4c}\cdot\min\{\Delta x, \Delta y\},
\end{equation}
which happens to be the usual CFL stability condition $\Delta t\leq CFL\cdot \frac{\epsilon}{c} \min\{\Delta x, \Delta y\}$ with $CFL=\frac{1}{4}$.
(ii) For the optically thick case ($\epsilon \rightarrow 0, \sigma=O(1)$), since $\tilde{\alpha}\rightarrow 0$ and $\tilde{d}\to -c/(4\pi\sigma )$,
we see that there is actually no constraint on the time step.
\end{rem}

\subsection{Realizability  of limiters}
\label{section of limiter}

In this subsection, we show that with suitable scaling limiters the conditions
\eqref{PN reconstruciotn conditon in pp thm}-\eqref{slopey I condition in pp thm} in Theorem \ref{positive preserving theorem}
can be realized. First, it is easy to see that if the reconstructed  %\eqref{I 0 reconstrution}
initial data
 in the $x$- and $y$-directions are nonnegative, i.e.,
\begin{equation} \label{eq1 in limiter}
 \vec{\bm \psi} \cdot (\vec{\bm I}_{i,j}^n + \delta_x  \vec{\bm I}_{i,j}^n(x-x_i)) \geq 0,
 \qquad \forall x\in( x_{i-1/2}, x_{i+1/2}),
\end{equation}
\begin{equation} \label{eq2 in limiter}
\vec{\bm \psi} \cdot (\vec{\bm I}_{i,j}^n + \delta_y  \vec{\bm I}_{i,j}^n (y-y_j) ) \geq 0,
\qquad   \forall y\in( y_{j-1/2}, y_{j+1/2}),
\end{equation}
we can verify that the conditions \eqref{PN reconstruciotn conditon in pp thm}-\eqref{slopey I condition in pp thm} hold
under a reasonable time step.

 In fact, with the nonnegative reconstructions \eqref{eq1 in limiter} and \eqref{eq2 in limiter}, by setting $x=x_i,y=y_j$, we first find the nonnegative $P_N$-reconstruction for the angular variable, i.e., $\vec{\bm \psi} \cdot \vec{\bm I}_{i,j}^n \geq 0$.
This immediately implies ${\bm\langle} (1\pm \xi)\vec{\bm \psi} \cdot \vec{\bm I}^n {\bm\rangle}_{1,4~\text{or}~2,3} \geq 0$ and
${\bm\langle} (1\pm \eta)\vec{\bm \psi} \cdot \vec{\bm I}^n {\bm\rangle}_{1,2~\text{or}~3,4} \geq 0$.
 Based on these results, it is easy to verify that the condition \eqref{PN reconstruciotn conditon in pp thm} holds for the case  $\tilde{\alpha}^{n+1,s}_{i\pm1/2,j}=\tilde{\alpha}^{n+1,s}_{i,j\pm1/2}$.
Then, if we follow the nonnegative piecewise linear reconstruction for the spatial variable in
\eqref{I 0 reconstrution}-\eqref{rho reconstrution}
%\eqref{PN reconstruciotn conditon in pp thm}-\eqref{slopey I condition in pp thm},
 and assume that $\Delta t$ satisfies
\begin{equation}\label{time step constraint of limiter section}
 \max_{\substack{i'=i\pm1/2, ~j'=j\pm1/2}}~
  \frac{-\widetilde{b}^{n+1,s}_{i',j'}(\Delta t)} {\widetilde{\alpha}^{n+1,s}_{i',j'}(\Delta t)}
\leq \min\{\Delta x, \Delta y\},
\end{equation}
thus for
\begin{equation}\nonumber
x=\left\{\begin{array}{ll}
x_{i\pm1/2}\pm\frac{\tilde{b}^{n+1,s}_{i+1/2,j}}{\tilde{\alpha}^{n+1,s}_{i+1/2,j}}\xi,\quad if ~ \xi\geq0,
\\[3mm]
x_{i\pm1/2}\mp \frac{\tilde{b}^{n+1,s}_{i-1/2,j}}{\tilde{\alpha}^{n+1,s}_{i-1/2,j}}\xi,\quad if ~ \xi<0,
\end{array}\right.\qquad
y=\left\{\begin{array}{ll}
y_{i\pm1/2}\pm\frac{\tilde{b}^{n+1,s}_{i,j+1/2}}{\tilde{\alpha}^{n+1,s}_{i,j+1/2}}\eta,\quad if ~ \eta\geq0,
\\[3mm]
y_{i\pm1/2}\mp \frac{\tilde{b}^{n+1,s}_{i,j-1/2}}{\tilde{\alpha}^{n+1,s}_{i,j-1/2}}\eta,\quad if ~ \eta<0,
\end{array}\right.
\end{equation}
we have  $x\in( x_{i-1/2}, x_{i+1/2}), y\in( y_{j-1/2}, y_{j+1/2})$.

Substituting the above $x,y$ into \eqref{eq1 in limiter}-\eqref{eq2 in limiter} and multiplying $\xi, \eta$ respectively
in \eqref{eq1 in limiter}-\eqref{eq2 in limiter}, we integrate with respect to the angular variable to verify
that the conditions in \eqref{slopex I condition in pp thm}-\eqref{slopey I condition in pp thm} hold. This implies that,
under the time step constraint \eqref{time step constraint of limiter section}, the
   conditions \eqref{PN reconstruciotn conditon in pp thm}-\eqref{slopey I condition in pp thm} hold automatically
   if the piecewise linear reconstruction of the initial intensities are nonnegative. We remark here that although
   the added time step constraint \eqref{time step constraint of limiter section} is needed, it will be
   clarified in Remark \ref{remark in limiter sec: time step constraint} that \eqref{time step constraint of limiter section} is
   in fact not restrictive.

Generally, the reconstruction of the initial intensity could be negative. In order to make
the conditions \eqref{PN reconstruciotn conditon in pp thm}-\eqref{slopey I condition in pp thm} valid in this case,
the scaling limiters will be used in this paper. Notice that the condition \eqref{PN reconstruciotn conditon in pp thm} also appears
in the positive-preserving needing conditions of our previous constructed first-order scheme \cite{Xu-Jiang-Sun-2021}. Therefore, as done there,
we tune the micro coefficients $\breve{\vec{\bm I}}^n _{i,j}$ uniformly, while keeping
the macro quantity $\rho_{i,j}^n$ unchanged, so as to make the condition \eqref{PN reconstruciotn conditon in pp thm} as well as the following conditions hold:
\begin{eqnarray}\label{limiter PN reconstruciotn conditon}
\nonumber
& & \qquad \frac{1}{4}\rho^n_{i,j}
+{\bm\langle}\xi\breve{\vec{\bm \psi}} {\bm\rangle}_{1,4}\cdot\breve{\vec{\bm I}}^n_{i,j}
 \geq 0,~~\qquad\qquad
  \frac{1}{4}\rho^n_{i,j}
-{\bm\langle}\xi\breve{\vec{\bm \psi}} {\bm\rangle}_{2,3}\cdot \breve{\vec{\bm I}}^n_{i,j}
\geq 0,~~
\\
 & &
 %\qquad
 \qquad \frac{1}{4}\rho^n_{i,j}
+{\bm\langle}\eta\breve{\vec{\bm \psi}} {\bm\rangle}_{1,2}\cdot\breve{\vec{\bm I}}^n_{i,j}
\geq 0,~~
\qquad\qquad
\frac{1}{4}\rho^n_{i,j}
-{\bm\langle}\eta\breve{\vec{\bm \psi}} {\bm\rangle}_{3,4}\cdot\breve{\vec{\bm I}}^n_{i,j}
 \geq 0.
\end{eqnarray}
This can be achieved
 by using the linear scaling (ls) limiter from \cite{{Zhang-Shu-2010-scaling-limiter}}. That is, one replaces $\breve{\vec{\bm I}}^n _{i,j}$
by $\lambda^{1}_{i,j} \breve{\vec{\bm I}}^n _{i,j}$,
 with $\lambda^{1}_{i,j}$ given by
\begin{equation}\label{limiter 1}
\lambda^{1}_{i,j} = \text{argmax}_{ ~\lambda\in[0,1]}\{\lambda: \lambda \breve{\vec{\bm I}}^n _{i,j}
~ \text{satisfies conditions in } ~\eqref{PN reconstruciotn conditon in pp thm} ~\text{and}~ \eqref{limiter PN reconstruciotn conditon}  \}.
\end{equation}

Since $\rho_{i,j}^n\geq 0$, the positive limiter in \eqref{limiter 1} is realizable.
With this modified $P_N$-reconstruciton (here we slightly abuse the notation and
still denote it by $\vec{\bm \psi}\cdot\vec{\bm I}_{i,j}^n$),
we modify the reconstruction slopes $\delta_x\vec{\bm I}_{i,j}^n$ and $\delta_y\vec{\bm I}_{i,j}^n$ to make
\eqref{slopex I condition in pp thm} and \eqref{slopey I condition in pp thm} hold, respectively.
For simplicity, the modifications are employed uniformly, i.e.,
one need replace $\delta_x\vec{\bm I}_{i,j}^n$ and $\delta_y\vec{\bm I}_{i,j}^n$
by $\lambda_{i,j}^{21}  \delta_x\vec{\bm I}^n _{i,j}$ and $\lambda_{i,j}^{22} \delta_y\vec{\bm I}^n _{i,j}$ respectively,
 with $\lambda^{21}_{i,j}, \lambda^{22}_{i,j}$ given by
\begin{equation}\label{limiter 21}
\lambda^{21}_{i,j} = \text{argmax}_{ ~\lambda\in[0,1]}\{\lambda: \lambda  \delta_x\vec{\bm I}^n _{i,j},
~ \text{satisfies conditions in }
\eqref{slopex I condition in pp thm}~\},
\end{equation}
\begin{equation}\label{limiter 22}
\lambda^{22}_{i,j} = \text{argmax}_{ ~\lambda\in[0,1]}\{\lambda: \lambda \delta_y\vec{\bm I}^n _{i,j}
~ \text{satisfies conditions in }
\eqref{slopey I condition in pp thm}~\}.
\end{equation}
%%%
Noticing that the modified $P_N$-reconstruction $\vec{\bm \psi}\cdot\vec{\bm I}_{i,j}^n$
satisfies the inequalities in \eqref{limiter PN reconstruciotn conditon}, which
can be rewritten as %we can rewrite
 $${\bm\langle}\xi\vec{\bm \psi} {\bm\rangle}_{1,4}\cdot\vec{\bm I}^n_{i,j}\geq 0,\quad
{\bm\langle}-\xi\vec{\bm \psi} {\bm\rangle}_{2,3}\cdot\vec{\bm I}^n_{i,j}\geq 0,\quad
{\bm\langle}\eta\vec{\bm \psi} {\bm\rangle}_{1,2}\cdot\vec{\bm I}^n_{i,j}\geq 0,\quad
{\bm\langle}-\eta\vec{\bm \psi} {\bm\rangle}_{3,4}\cdot\vec{\bm I}^n_{i,j}\geq 0.$$
Thus, the limiters in \eqref{limiter 21} and \eqref{limiter 22} are realizable, and the conditions in
\eqref{slopex I condition in pp thm}-\eqref{slopey I condition in pp thm} are hence satisfied.

We remark that using the positive $P_N$ closures in \cite{Positive-PN-2010} or the positive filtered $P_N$ closures in
\cite{Laiu-positive-FPN-2016} can also yield the desired positive scheme. It is, however, computationally expensive.
So, following the construction procedure in \cite{Laiu-AP-PP-FPN-2019} and of our previous first-order scheme \cite{Xu-Jiang-Sun-2021},
here we apply much cheap linear scaling limiters to achieve the positive property.

By incorporating the above positive-preserving ($PP$) limiters into the $FP_N$-based UGKS, we obtain the so-called {\it $PPFP_N$-based UGKS}
by following almost the same procedure as in the "Loop of the $FP_N$-based UGKS" in Subsection \ref{subsec: loop of $FP_N$-based UGKS}
  except that, before solving the linearized macro equations
 \eqref{nonlinear iteration of discrete macro eqaution-1-UGKS}-\eqref{nonlinear iteration of discrete macro eqaution-2-UGKS}
 and the micro equations \eqref{discrete micro equation}, the micro coefficients $\breve{\vec{\bm I}}^n _{i,j}$ in the
fluxes have to be replaced by $\lambda^{1}_{i,j}\breve{\vec{\bm I}}^n_{i,j}$ for each spatial cell $(i,j)$ first, and then replace
$\delta_x\vec{\bm I}_{i,j}^n, \delta_y\vec{\bm I}_{i,j}^n$ by
$\lambda_{i,j}^{21}\delta_x\vec{\bm I}^n_{i,j}, \lambda_{i,j}^{22} \delta_y\vec{\bm I}^n_{i,j}$.
%%%%%%%%%%%%%%%%%%%%%%%%%%%%%%%%%%%%%%%
%%%%%%%%%%%%%%%%%%%%%%%%%%%%%%%%%%%
\begin{rem}\label{remark in limiter sec: time step constraint}
For the optically thin case ($\epsilon=O(1), \sigma\rightarrow 1$),
%($\epsilon=O(1), \sigma\ll 1$)
the condition \eqref{time step constraint of limiter section} reduces to $\Delta t \leq \frac{2\epsilon}{c}\min\{\Delta x, \Delta y\}$, which is  induced by the fact that $\frac{\tilde{b}}{\tilde{\alpha}}\rightarrow -\frac{c\Delta t}{2\epsilon}$. Thus, in combination
with \eqref{time constr after pp theorem}, the final positive preserving time step constraint takes the following form:
\begin{equation}
\Delta t \leq \frac{\epsilon}{4c}\cdot\min\{\Delta x, \Delta y\}.
\end{equation}
On the other hand, for the optically thick case  ($\epsilon\rightarrow 1, \sigma=O(1)$),
 as $\frac{\tilde{b}}{\tilde{\alpha}}$ tends to zero, we see that the condition \eqref{time step constraint of limiter section} actually
 does not impose any constraints on the time step.
\end{rem}

 \subsection{Asymptotic analysis}\label{asymptotic analysis}
In this subsection we analyze the asymptotic preserving (AP) property of the above developed $PPFP_N$-based UGKS scheme.
In fact, the asymptotic behavior of the algorithm in the small $\epsilon$ limit is completely determined
by the coefficient functions, which are given by
%%%%%%%%%%%%%%%%%%%%%%%%%%%%%%%%%%%%%%%%%%%%%%%%%%%%%%%%%%%%%%%%%%%%%%%
\begin{propo}\label{Proposition of coefficients' property}
Let $\sigma$ be positive  and $\Delta t = O(1)$. Then as $\epsilon\to 0$, we have

$\bullet$  $\tilde{\alpha}(\Delta t, \epsilon, \sigma, \nu)$  and $\tilde{b}(\Delta t, \epsilon, \sigma, \nu)$ tend to $0$;

$\bullet$ $\frac{\epsilon}{c} \tilde{c}(\Delta t, \epsilon, \sigma,\nu)$ tends to $\frac{1}{4\pi}$;

$\bullet$  $\tilde{d}(\Delta t, \epsilon, \sigma, \nu)$ tends to $-\frac{c}{4\pi\sigma}$.

\end{propo}

  With Proposition \ref{Proposition of coefficients' property} in hand, we can show the desired asymptotic preserving property of the proposed numerical scheme in the following theorem.

\begin{thm}\label{A-P theorem}
Let $\sigma$ be positive,   and  $(\Delta x, \Delta y, \Delta t) =O(1)$.  Then, as $\epsilon$ tends to zero,
the {\it $PPFP_N$-based UGKS} scheme
%given by \eqref{discrete macro equation}--\eqref{discrete micro equation}
for the equations \eqref{Rewritten Gray Radiative transfer equaitons}
  approaches to a standard positive-preserving implicit
diffusion scheme for the diffusion limit equation \eqref{nonlinear diffusion limiting equation}.
\end{thm}
\begin{proof}
 By Remark  \ref{remark in limiter sec: time step constraint}, we know that as $\epsilon\to 0$, there is no constraint on the time step.  And also from Proposition \ref{Proposition of coefficients' property}, one has
 $\tilde{\alpha}, \tilde{b} \stackrel{\epsilon\rightarrow 0}{\longrightarrow} 0$.
 Therefore, the positive-preserving limiters do not act on the solution at all for this optically thick limit case.
In other words, they can be omitted in this case. Thus, we only need to focus on analyzing the discrete macro and
micro equations \eqref{discrete macro equation}--\eqref{discrete micro equation}.

Recalling that $\kappa=\frac{\epsilon^2}{\epsilon^2+\Delta t\beta\sigma}\stackrel{\epsilon\rightarrow 0}{\longrightarrow} 0$,
and following almost the same process as in the proof of Theorem 4.1 in \cite{Xu-Jiang-Sun-2021}, we can obtain the asymptotic
preserving property of the proposed {\it $PPFP_N$-based UGKS} scheme as $\epsilon\to 0$. This completes the proof.

\end{proof}

\section{Simplified $PPFP_N$-based UGKS}  \label{section: $PPFP_N$-S}
%%%%%%%%%%
The numerical fluxes in the above $PPFP_N$-based UGKS are of spatial second-order accuracy for
\eqref{Rewritten Gray Radiative transfer equaitons}.
Compared with the fluxes of the first-order scheme in \cite{Xu-Jiang-Sun-2021}, the flux expressions in \eqref{flux Phi}-\eqref{flux G} and \eqref{flux Upsilon}-\eqref{flux H} are quite complex. In this section we shall show that in the regime $\epsilon\ll 1$ and the regime
$\epsilon=O(1)$, the fluxes can be simplified, while the positive preserving and asymptotic preserving properties can be still kept.
Thus, a simplified  $PPFP_N$-based UGKS (abbreviated as the $PPFP_N$-based SUGKS) is obtained and the computational costs can be reduced greatly.

In order to simplify the fluxes in \eqref{flux Phi}-\eqref{flux G}, we denote
\begin{equation}\nonumber
\Phi^{n+1}_{i-1/2,j}( \tilde{b} ) = \tilde{b}^{n+1}_{i-1/2,j}\left(
{\bm\langle}\xi^2\vec{\bm \psi}{\bm\rangle}_{1,4} \cdot \delta_x\vec{\bm I}^n_{i-1,j}
 +{\bm\langle}\xi^2\vec{\bm \psi} {\bm\rangle}_{2,3}\cdot \delta_x\vec{\bm I}^n_{i,j} \right),
\end{equation}
\begin{equation}\nonumber
G_{i-1/2,j}(\tilde{b})= \tilde{b}^{n+1}_{i-1/2,j}\left(
{\bm\langle}\xi^2 \psi_{\ell}^m\vec{\bm \psi} {\bm\rangle}_{1,4}\cdot \delta_x\vec{\bm I}^n_{i-1,j}
+{\bm\langle}\xi^2\psi_{\ell}^m\vec{\bm \psi} {\bm\rangle}_{2,3}\cdot \delta_x\vec{\bm I}^n_{i,j}
\right),
\end{equation}
\begin{equation}\nonumber
G_{i-1/2,j}(\tilde{c})=\tilde{c}^{n+1}_{i-1/2,j}{\bm\langle}\xi \psi_{\ell}^m {\bm\rangle}~
\left( (\kappa^{n+1}\phi^n)_{i-1/2,j} +(1-\kappa_{i-1/2,j}^{n+1})\rho_{i-1/2,j}^{n+1} \right),
\end{equation}
where function $G$ %$G_{\ell,m}$
denotes the components of the vector $\vec{\bf G}$ for $1\leq \ell \leq N, -\ell \leq m \leq \ell$.
We point out that the other fluxes can be handled in a similar way.
 In the following, we use  "$\lesssim$" to denote  "$\leq C$" with $C$ being a nonnegative constant independent of
 $\epsilon$, $\Delta x$, $\Delta y$ and $\Delta t$. By a detailed derivation that is given in Appendix \ref{section: appendix},
 we obtain the following lemma:
\begin{lem}\label{lemma: analysis of fluxes}
Assume that $\sigma, C_{\nu}$  and the solution of \eqref{Rewritten Gray Radiative transfer equaitons}  are sufficiently
smooth. Then, it holds that
\begin{eqnarray}\label{eq1-analysis of fluxes}
\Big|\frac{ \Phi^{n+1}_{i-1/2,j}(\tilde{b})- \Phi^{n+1}_{i+1/2,j}(\tilde{b}) }{\Delta x}
+\frac{ \Upsilon^{n+1}_{i,j-1/2}(\tilde{b})- \Upsilon^{n+1}_{i,j+1/2}(\tilde{b}) }{\Delta y}\Big|
\lesssim  \Delta t,
\\[2mm]\label{eq2-analysis of fluxes}
\Big|\frac{ G_{i-1/2,j}(\tilde{b})- G_{i+1/2,j}(\tilde{b}) }{\Delta x}
+\frac{ H_{i,j-1/2}(\tilde{b})- H_{i,j+1/2}(\tilde{b}) }{\Delta y}\Big|
\lesssim \Delta t,
\end{eqnarray}
in the regimes $\epsilon \ll 1$ and $\epsilon=O(1)$.  And
\begin{eqnarray}\label{eq3-analysis of fluxes}
\hspace{-4mm}
\epsilon^2\Big|\frac{ G_{i-1/2,j}(\tilde{c})- G_{i+1/2,j}(\tilde{c}) }{\Delta x} +\frac{ H_{i,j-1/2}(\tilde{c})- H_{i,j+1/2}(\tilde{c}) }{\Delta y}\Big|\lesssim \Delta t \;\;\text{in the regime} ~\epsilon \ll 1, \;\;
\\[2mm]\label{eq4-analysis of fluxes}
\hspace{-4mm}
\Big|\frac{ G_{i-1/2,j}(\tilde{c})- G_{i+1/2,j}(\tilde{c}) }{\Delta x}+\frac{ H_{i,j-1/2}(\tilde{c})- H_{i,j+1/2}(\tilde{c}) }{\Delta y} \Big|\lesssim  \Delta t \;\;\text{in the regime}~ \epsilon=O(1).\;\;
\end{eqnarray}
\end{lem}

Now with the help of Lemma \ref{lemma: analysis of fluxes}, in the regimes $ \epsilon\ll 1$ and $\epsilon=O(1)$, one can estimate the contribution of each term on the left hand sides of \eqref{eq1-analysis of fluxes}-\eqref{eq4-analysis of fluxes}
%each single term  mentioned  in the above lemma
%in the fluxes
to the final solution.
For this, we shall analyze the contribution to the macro solution and micro solution
respectively.
Firstly, by the linearized macro equations \eqref{nonlinear iteration of discrete macro eqaution-1-UGKS}, we use the notion
$$
 \Delta \vec{\bm \rho}^{~s+1}=(\Delta \rho^{n+1,s+1}_{1,1}, \Delta \rho^{n+1,s+1}_{1,2},..., \Delta \rho^{n+1,s+1}_{2,1}, \Delta \rho^{n+1,s+1}_{2,2},...  )'.
$$
 to represent the solution of the following equation in matrix form:
\begin{equation}\nonumber%$$
{\bm A}^s \Delta \vec{\bm \rho}^{~s+1} = \vec{ \bm z}^{~s},
\end{equation}%$$
where ${\bm A}^s$ is the coefficient matrix of \eqref{nonlinear iteration of discrete macro eqaution-1-UGKS}, and
$\vec{ \bm z}^{~s}=( z_{1,1}^{n+1,s},z_{1,2}^{n+1,s},..., z_{2,1}^{n+1,s},z_{2,2}^{n+1,s}, ... )' $
with
$$z_{i,j}^{n+1,s}= \Delta t\left(\frac{ \Phi^{n+1,s+1}_{i-1/2,j}(\tilde{b})- \Phi^{n+1,s+1}_{i+1/2,j}(\tilde{b}) }{\Delta x}
+\frac{ \Upsilon^{n+1,s+1}_{i,j-1/2}(\tilde{b})- \Upsilon^{n+1,s+1}_{i,j+1/2}(\tilde{b}) }{\Delta y}\right).$$
Namely, $\Delta \rho^{n+1,s+1}_{i,j}$ shows the contribution
of the  term in \eqref{eq1-analysis of fluxes}
to the macro solution $\rho_{i,j}^{n+1}$.
Since the coefficient matrix ${\bm A}^s$ can be decomposed into ${\bm A}^s = {\bm E}-{\bm B}^s$
with ${\bm E}$ being the identity matrix, and ${\bm B}^s$ satisfying ${\bm B}^s(i,j)\geq 0(i\neq j)$ and
$\sum_j {\bm B}^s(i,j)\leq 0$ for all $i$, we see that $\|({\bm A}^s)^{-1}\|_{\infty} \leq 1$.
Consequently,
$$\| \Delta \vec{\bm \rho}^{~s+1}\|_{\infty} = \| ({\bm A}^s)^{-1} \vec{ \bm z}^{~s}\|_{\infty} \leq
 \| ({\bm A}^s)^{-1}\|_{\infty} \|\vec{ \bm z}^{~s}\|_{\infty} \leq \|\vec{ \bm z}^{~s}\|_{\infty}. $$
From \eqref{eq1-analysis of fluxes} one gets $\|\vec{ \bm z}^{~s}\|_{\infty} \lesssim (\Delta t)^2$.
Therefore, $\big|\Delta \rho_{i,j}^{n+1,s+1} \big| \lesssim (\Delta t)^2$ follows immediately.
In addition, from \eqref{nonlinear iteration of discrete macro eqaution-2-UGKS}, the contribution to the macro solution $\phi_{i,j}^{n+1}$ reads as
% if we set
$$
\Delta \phi_{i,j}^{n+1,s+1} \equiv%\triangleq
 (1- \kappa_{i,j}^{n+1,s}) \Delta \rho_{i,j}^{n+1,s+1}.
$$
Since $\kappa = \frac{\epsilon^2}{ \epsilon^2+\Delta t \beta\sigma}\in [0,1], $ thus
$ \big|\Delta \phi_{i,j}^{n+1,s+1} \big| \lesssim (\Delta t)^2 $ holds.
%\ccr{ We remind that $\Delta \rho_{i,j}^{n+1,s+1}$   }

Secondly, in view of \eqref{discrete micro equation}, that
\begin{eqnarray}\label{solution I of micro eq}
 \nonumber
\hspace{-6mm}\breve{\vec{\bm I}}^{n+1}_{i,j}  &=&
\left( (1+ \frac{\sigma_{i,j}^{n+1}c \Delta t}{\epsilon^2}){\bf E} +\Delta t \sigma_f \breve{{\bf F}}\right)^{-1}\times \\
& &\qquad
\left\{ \breve{\vec{\bm I}}^{n}_{i,j} + \Delta t \left(
\frac{ \vec{\bm G}_{i-1/2,j} - \vec{\bm G}_{i+1/2,j}} {\Delta x}
+  \frac{ \vec{\bm H}_{i,j-1/2}-\vec{\bm H}_{i,j+1/2} } { \Delta y } \right) \right\}.
\end{eqnarray}
 Recalling that $\breve{{\bf F}}$ is a diagonal matrix, from \eqref{solution I of micro eq}, the contribution of the term in \eqref{eq2-analysis of fluxes} to the micro solution $(I_{\ell}^m)_{i,j}^{n+1}$
  takes the following form
 %if we can define
\begin{eqnarray}
 \nonumber
\Delta_{\tilde{b}} (I_{\ell}^m)_{i,j}^{n+1}
 &\equiv&% &\triangleq&
  \frac{\Delta t}{1+ \frac{\sigma_{i,j}^{n+1}c \Delta t}{\epsilon^2} +\Delta t \sigma_f {\bf F}_{(\ell,m),(\ell,m)} }\times
\\\nonumber
& & \qquad \left(
 \frac{{G}_{i-1/2,j}(\widetilde{b})-G_{i+1/2,j}(\widetilde{b})}{ \Delta x }
+ \frac{ H_{i,j-1/2}(\widetilde{b})- H_{i,j+1/2}(\widetilde{b}) }{\Delta y } \right) .
\end{eqnarray}
In view of \eqref{eq2-analysis of fluxes} and the fact that ${\bf F}_{(\ell,m),(\ell,m)}$ is nonnegative, we immediately get
\begin{equation}\nonumber
\big|\Delta_{\tilde{b}} (I_{\ell}^m)_{i,j}^{n+1} \big| \lesssim (\Delta t)^2.
\end{equation}
Futhermore, for
\begin{eqnarray}
 \nonumber
\Delta_{\tilde{c}} (I_{\ell}^m)_{i,j}^{n+1}
 &\equiv& %&\triangleq&
   \frac{\Delta t}{1+ \frac{\sigma_{i,j}^{n+1}c \Delta t}{\epsilon^2} +\Delta t \sigma_f {\bf F}_{(\ell,m),(\ell,m)} }\times
\\\nonumber
& & \qquad \left(
 \frac{{G}_{i-1/2,j}(\widetilde{c})-G_{i+1/2,j}(\widetilde{c})}{ \Delta x }
+ \frac{ H_{i,j-1/2}(\widetilde{c})- H_{i,j+1/2}(\widetilde{c}) }{\Delta y } \right) .
\end{eqnarray}
Similarly we can obtain $\big|\Delta_{\tilde{c}} (I_{\ell}^m)_{i,j}^{n+1} \big| \lesssim (\Delta t)^2$.
Thus, we summarize the results in the following theorem:
%Thus the following theorem holds.
%%%%%%%%%%%%%%%%%%%%%%%%%%%%%%%%%%%%%%%%%%%%%%%
\begin{thm}\label{therom: dt square}
 Assume that $\sigma, C_{\nu}$  and the solution of \eqref{Rewritten Gray Radiative transfer equaitons}  are sufficiently
smooth. Then, in the regime $\epsilon\ll 1$ and the regime
$\epsilon=O(1)$, it holds that
\begin{equation}\label{result 1 of therom: dt square}
\big|\Delta \phi_{i,j}^{n+1,s+1} \big| \lesssim (\Delta t)^2,\qquad\qquad~
\big|\Delta \rho_{i,j}^{n+1,s+1} \big| \lesssim (\Delta t)^2,
\end{equation}
and
\begin{equation}\label{result 2 of therom: dt square}
\big|\Delta_{\tilde{b}} (I_{\ell}^m)_{i,j}^{n+1} \big| \lesssim (\Delta t)^2,\qquad\qquad~
\big|\Delta_{\tilde{c}} (I_{\ell}^m)_{i,j}^{n+1} \big| \lesssim (\Delta t)^2.
\end{equation}
%Here $\Delta \phi_{i,j}^{n+1,s+1}, \Delta \rho_{i,j}^{n+1,s+1}$ denote the contribution of ${\Phi}^{n+1,s+1}_{i\pm1/2,j}(\widetilde{b}), {\Upsilon}^{n+1,s+1}_{i,j\pm1/2}(\widetilde{b})$ to the macro \ccb{solution},
% while $\Delta_{\tilde{b}} (I_{\ell}^m)_{i,j}^{n+1}$, $\Delta_{\tilde{c}} (I_{\ell}^m)_{i,j}^{n+1}$
%  respectively denote the contribution of  ${G}_{i\pm1/2,j}(\widetilde{b})$,
%${H}_{i,j\pm1/2}(\widetilde{b})$  and ${G}_{i\pm1/2,j}(\widetilde{c})$,
% ${H}_{i,j\pm1/2}(\widetilde{c})$ to the micro \ccb{solution}.
\end{thm}

Because the first-order discretization for the temporal variable is used in the scheme of the above section,
 and Theorem \ref{therom: dt square} implies that, in the regime $\epsilon\ll 1$ and the regime $\epsilon=O(1)$,
 the contribution of the corresponding terms in Theorem \ref{therom: dt square} to the final solutions are the high-order terms
 with respect to $\Delta t$, we can delete these terms from the fluxes directly to simplify the fluxes but without
 reducing the accuracy in the temporal variable. We call this simplification the simplified $PPFP_N$-based UGKS.
 As a result, in the regime $\epsilon\ll 1$ and the regime $\epsilon=O(1)$, for the $x$-direction as example,
 the simplified macro and micro fluxes read as
\begin{eqnarray} %\label{simplified-flux Phi}
\nonumber
\mathring{\Phi}^{n+1}_{i-1/2,j}
&=&
\tilde{\alpha}^{n+1}_{i-1/2,j}\left\{  {\bm\langle}\xi\vec{\bm \psi}  {\bm\rangle}_{1,4}\cdot\left(\vec{\bm I}^n_{i-1,j}+
\frac{\Delta x}{2}\delta_x\vec{\bm I}^n_{i-1,j}\right)
+
{\bm\langle}\xi\vec{\bm \psi} {\bm\rangle}_{2,3}\cdot\left(\vec{\bm I}^n_{i,j}
-\frac{\Delta x}{2}\delta_x\vec{\bm I}^n_{i,j}\right)
\right\}\\\nonumber
& &
+\tilde{d}^{n+1}_{i-1/2,j}\frac{4\pi}{3}
\left(\delta_x (\kappa^{n+1}\phi^n)_{i-1/2,j} +(1-\kappa_{i-1/2,j}^{n+1})\delta_x\rho_{i-1/2,j}^{n+1} \right),
\end{eqnarray}
\begin{eqnarray}%\label{simplified-flux G}
\nonumber
{\mathring{\vec{\bf G}}}_{i-1/2,j}
&=&
\tilde{\alpha}^{n+1}_{i-1/2,j}
\left\{ {\bm\langle}\xi\breve{\vec{\bm \psi}}\vec{\bm \psi}' {\bm\rangle}_{1,4} \cdot
\left(\vec{\bm I}^n_{i-1,j}+
\frac{\Delta x}{2}\delta_x\vec{\bm I}^n_{i-1,j} \right)+ {\bm\langle}\xi\breve{\vec{\bm \psi}}\vec{\bm \psi}' {\bm\rangle}_{2,3} \cdot \left(\vec{\bm I}^n_{i,j}
-\frac{\Delta x}{2}\delta_x\vec{\bm I}^n_{i,j}\right) \right\}\\\nonumber
& &
+ \tilde{d}^{n+1}_{i-1/2,j}{\bm\langle}\xi^2\breve{\vec{\bm \psi}} {\bm\rangle}
\left(\delta_x (\kappa^{n+1}\phi^n)_{i-1/2,j} +(1-\kappa_{i-1/2,j}^{n+1})\delta_x\rho_{i-1/2,j}^{n+1} \right).
\end{eqnarray}
Similar expressions can be taken for $\mathring{\Upsilon}^{n+1}_{i,j-1/2}$, ${\mathring{\vec{\bf H}}}_{i,j-1/2}$ in the $y$-direction.

In view of the simplified fluxes, it is easy to find that they are in fact a direct combination of the simplest second-order
upwind fluxes for the optically thin regime and the second-order fluxes for the diffusion limit. And UGKS just provides the combination
coefficients, i.e, multi-scale coefficients that make the scheme possess the multi-scale property. We also remind that these fluxes
can not be simplified any more in the intermediate regime. And in this case,
it is better to keep all the terms in the fluxes \eqref{flux Phi}-\eqref{flux G} and \eqref{flux Upsilon}-\eqref{flux H}, i.e.,
 the original $PPFP_N$-based UGKS is preferred.

In the regime $\epsilon\ll 1$ and the regime $\epsilon=O(1)$,
replacing all the fluxes in the $PPFP_N$-based UGKS scheme
with the above simplified fluxes, which have no effect on the positive preserving  and asymptotic preserving properties, will result in the desired simplified $PPFP_N$-based UGKS, i.e., the $PPFP_N$-based SUGKS.

\section{Numerical experiments}
\label{section: Numerical experiments}
In this section we test the performance of the proposed second order AP and PP schemes, i.e.,
the $PPFP_N$-based $UGKS$ and $PPFP_N$-based $SUGKS$, which we abbreviate as $PPFP_N$ and $PPFP_N$-$S$ in the following.

As for the time step $\Delta t$, it is taken to satisfy the positive-preserving constraint conditions
\eqref{time step constraints}, \eqref{time step constraint of limiter section} and the CFL stability condition.
Here the CFL condition is determined by using an argument similar to that used in \cite{L. Mieussens-linear kinetic-2013}, and
as a result, the following CFL condition will be used in our numerical experiments:
\begin{equation}\nonumber
\Delta t \leq CFL \cdot \min\{\Delta x, \Delta y\},
\end{equation}
where $CFL<1$ is some problem-dependent constant.
\vspace{3mm}

First in Examples 1 and 2, we consider the following linear radiation transfer equation, which is a simplification of the system
\eqref{Gray Radiative transfer equaitons} in the case that the material temperature and the radiation temperature are the same
\begin{equation}\label{linear kinetic equation}
\frac{\epsilon^2}{c}\partial_t I + \epsilon  {\bm \Omega}\cdot \nabla I=\sigma(\frac{1}{4\pi}\rho-I).
\end{equation}
%\vspace{0.01mm}

 {\bf \noindent {\sc Example 1} (The line source problem).}
This problem serves as a benchmark test in studying various angular approaches \cite{Line-source-benchmark-2013} in dealing with the
ray effects. We present the problem specifications in Table \ref{Problem specification of the line source}.
%In this example we take $\sigma=1,\epsilon=1, c=1$ in equation \ref{linear kinetic equation} and the spatial computational domain to be $[-1.5,1.5]\times[-1.5,1.5]$.
% All boundary conditions are vacuum.
%The original initial density distribution
%$I(x,y,{\bm\Omega},0)=\frac{1}{4\pi}\delta(x)\delta(y)$ is approximated by the following steep Gaussian distribution:
%$$
%I(x,y,{\bm\Omega},0) \approx \frac{1}{4\pi}~\left( \frac{1}{2\pi \varsigma^2} ~e^{-\frac{x^2+y^2}{2 \varsigma^2}}\right),
%$$
%where $\varsigma=0.03$.
The numerical simulation is run on $150\times 150$ spatial meshes with time step $\Delta t = 0.1\cdot\min\{\Delta x, \Delta y\}$,
which satisfies the positive preserving constraint conditions and the CFL condition.
The final simulation time is $t=1$. The spherical-spline filter function $f_{\text{SSpline}}(\lambda)=\frac{1}{1+\lambda^4}$
and $\sigma_f=45$ are applied. Figures \ref{Figure: contour of line source} and \ref{Figure: rho lineouts of line source}
show the contours of the solution $\rho$ and the line-outs along the nonnegative $x$-axis and $y=x$ at time $t=1$.
Moreover, the exact solution in \cite{Line-source-exact-solution-1999},
as well as the $S_{16}$ and $P_{11}$ solutions, and in particular, the solution of the first-order $PPFP_{11}$ scheme
(abbreviate as $PPFP_{11}^{1-\mathrm{order}}$ in the following) constructed in our previous work \cite{Xu-Jiang-Sun-2021}
with $300\times 300$ spatial meshes are also plotted in Figures \ref{Figure: contour of line source}
and \ref{Figure: rho lineouts of line source} for comparison.
%%%%%%%%%%%%%%%%%%%%%%%%%%%%
 \begin{table}[h]
\centering
 \caption{\small Problem specifications of {\sc Example 1}. }
\begin{tabular}{ll}%{c|cc}
 \hline
 Spatial domain: & \qquad   $[-1.5,1.5] \times [-1.5,1.5]$  \\
 Parameters: & \qquad  $\sigma=1,  c=1, \epsilon=1$\\
 Boundary conditions:& \qquad  all are vacuum  \\
 Initial condition:& $ \qquad
I(x,y,{\bm{\Omega}},0) \approx \frac{1}{4\pi}~\left( \frac{1}{2\pi \varsigma^2} ~e^{-\frac{x^2+y^2}{2 \varsigma^2}}\right)$\\
 & \qquad  with  $\varsigma=0.03$
\\ \hline
 \end{tabular}
 \label{Problem specification of the line source}
 \end{table}

It can be seen from Figures \ref{Figure: contour of line source}-\ref{Figure: rho lineouts of line source} that the $S_{16}$ solution
and $P_{11}$ solution perform poorly due to ray effects and oscillations, respectively. Furthermore, negative $\rho$ can be
clearly observed in the $P_{11}$ approximation. On the other hand, the first-order $PPFP_{11}$ solution,
the current second-order $PPFP_{11}$ solution and $PPFP_{11}$-$S$ solution not only have higher accuracy, but also can suppress
the oscillations of the $P_{11}$ solution and maintain the positivity of the solution.

In Table \ref{table of line source}, we give the relative $\ell_2$ spatial errors of $\rho$ at $t=1$
(denoted by $re(\ell_{2}^{\rho})$) and the computation time of these schemes. The numerical tests are carried out on a Lenovo personal
computer with Intel(R) Core(TM) i7-8700 CPU. It can be seen from Table \ref{table of line source} that the current second-order schemes
indeed perform better in accuracy than $S_{16}$ and $P_{11}$. Furthermore, the current schemes are computationally efficient as well.
In particular, by comparing the $PPFP_{11}^{1-\mathrm{order}}$ scheme with the $PPFP_{11}$-$S$ scheme, we find that the $PPFP_{11}$-$S$ scheme
achieves a more accurate solution with less CPU time, which validates the simplified scheme, at least for this example.
In addition, comparing the computation time of the $PPFP_{11}$ scheme with that of the $P_{11}$ scheme, we can conclude that the additional filtering term and positive-preserving limiters do not add too much computation time.

\begin{table}[h]
\centering
\begin{tabular}{cccccc}%{c|cc}
\hline
  Schemes & $S_{16}$ &  $P_{11}$ %& $FP_{11}$
  &  $PPFP_{11}^{1-\mathrm{order}}$& $PPFP_{11}$ & $PPFP_{11}$-$S$ \\ \hline
   $re(\ell_{2}^{\rho}) $%  $error_{\rho}$
   &  5.31E-1 &   8.45E-1& % 6.66E-2&
   9.29E-2 & 6.78E-2 & 5.79E-2  \\ \hline
    CPU time (sec) &  210 &    868  % &1015
    & 247  & 882 & 220  \\ \hline
 \end{tabular}
 \caption{\small  {\sc Example 1.} %Relative $\ell^2$ spatial error
 Relative error $re(\ell_2^{\rho})$
  and computation time by the various methods.}
 \label{table of line source}
 \end{table}

\begin{figure}[htbp]
{
\begin{minipage}[h]{0.45\textwidth}
\centering
\centerline{\includegraphics[width=2.2in]{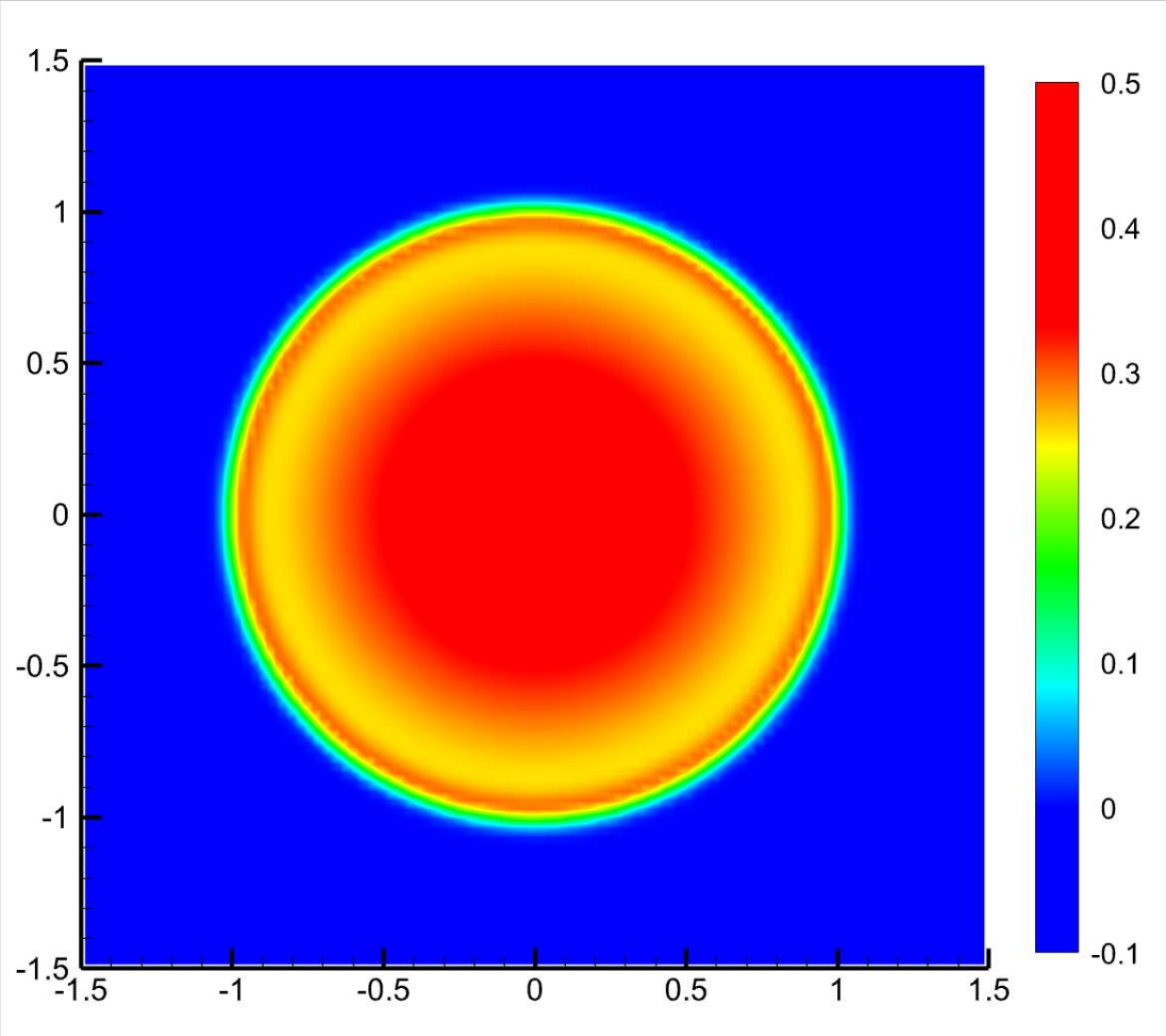}}
\centerline{\small (a) Exact solution}
\end{minipage}
\begin{minipage}[h]{0.45\textwidth}
\centering
\centerline{\includegraphics[width=2.2in]{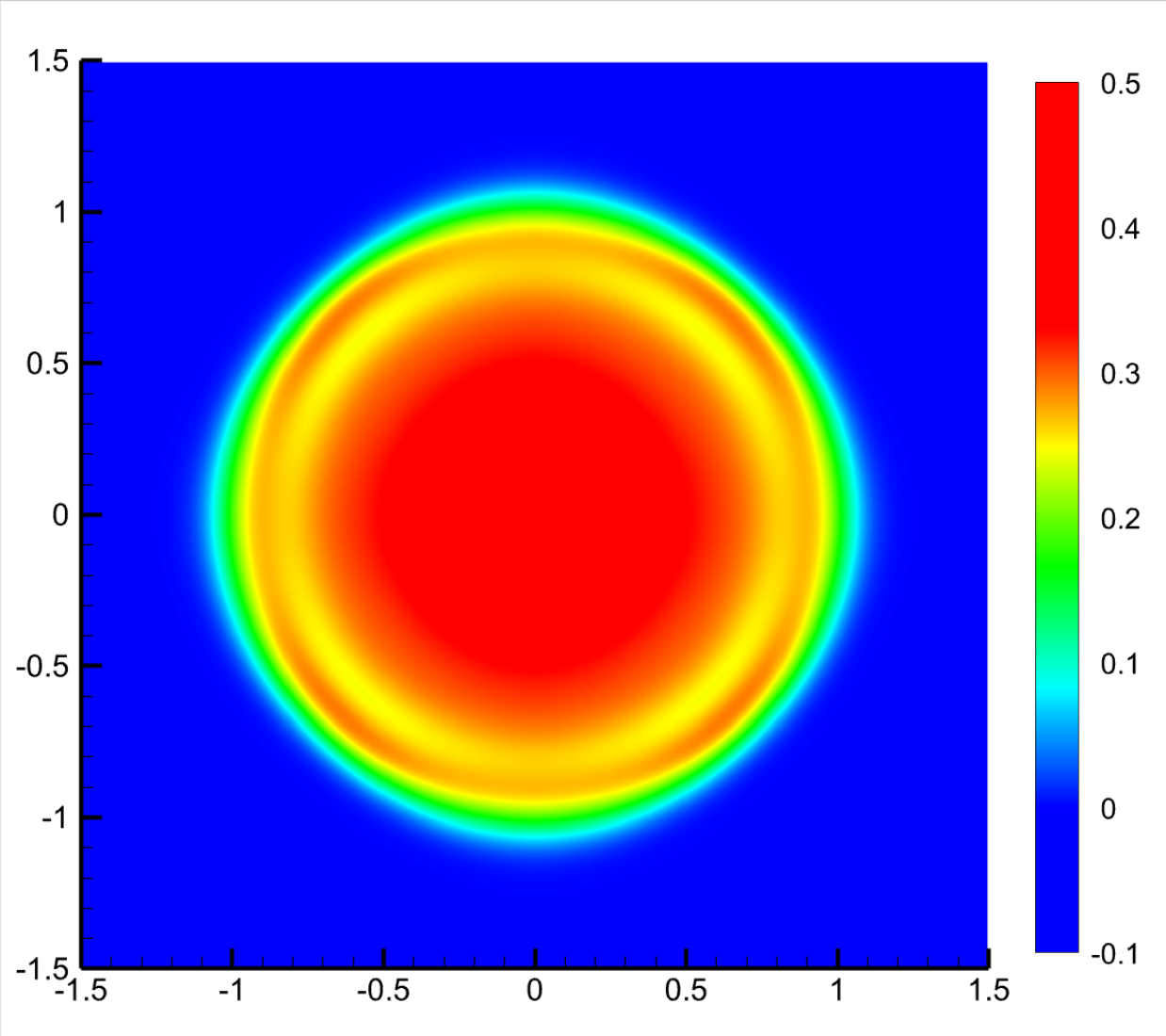}}%
\centerline{\small (b) $PPFP_{11}^{1-\mathrm{order}}$ solution}
\end{minipage}%
}\vskip -1.6pt
{
\begin{minipage}[h]{0.45\textwidth}
\centering
\centerline{\includegraphics[width=2.2in]{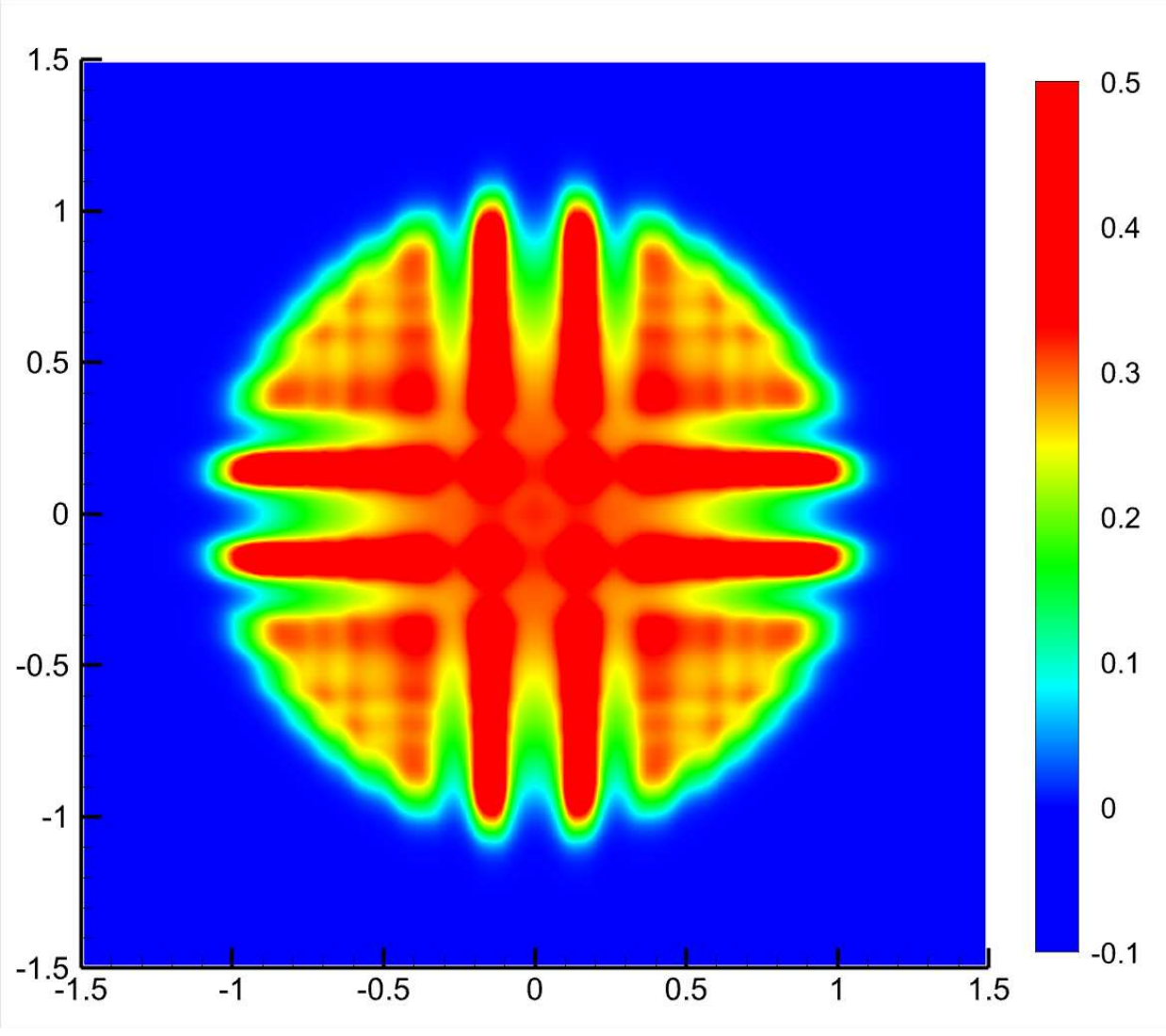}}
\centerline{\small (c) $S_{16}$ solution}
\end{minipage}
\begin{minipage}[h]{0.45\textwidth}
\centering
\centerline{\includegraphics[width=2.2in]{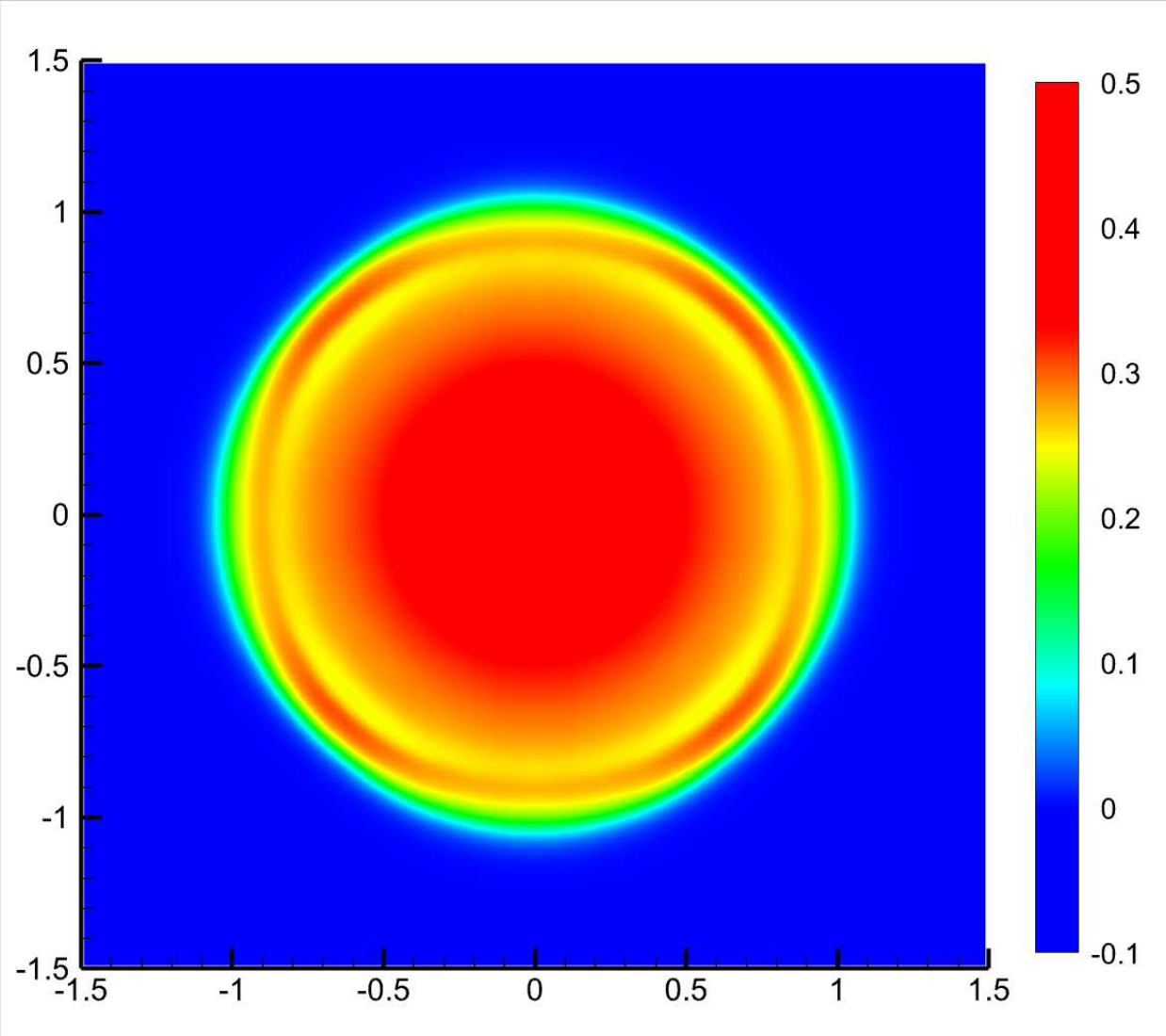}}
\centerline{\small (d) $PPFP_{11}$ solution}
\end{minipage}
}\vskip -1.6pt
{
\begin{minipage}[h]{0.45\textwidth}
\centering
\centerline{\includegraphics[width=2.2in]{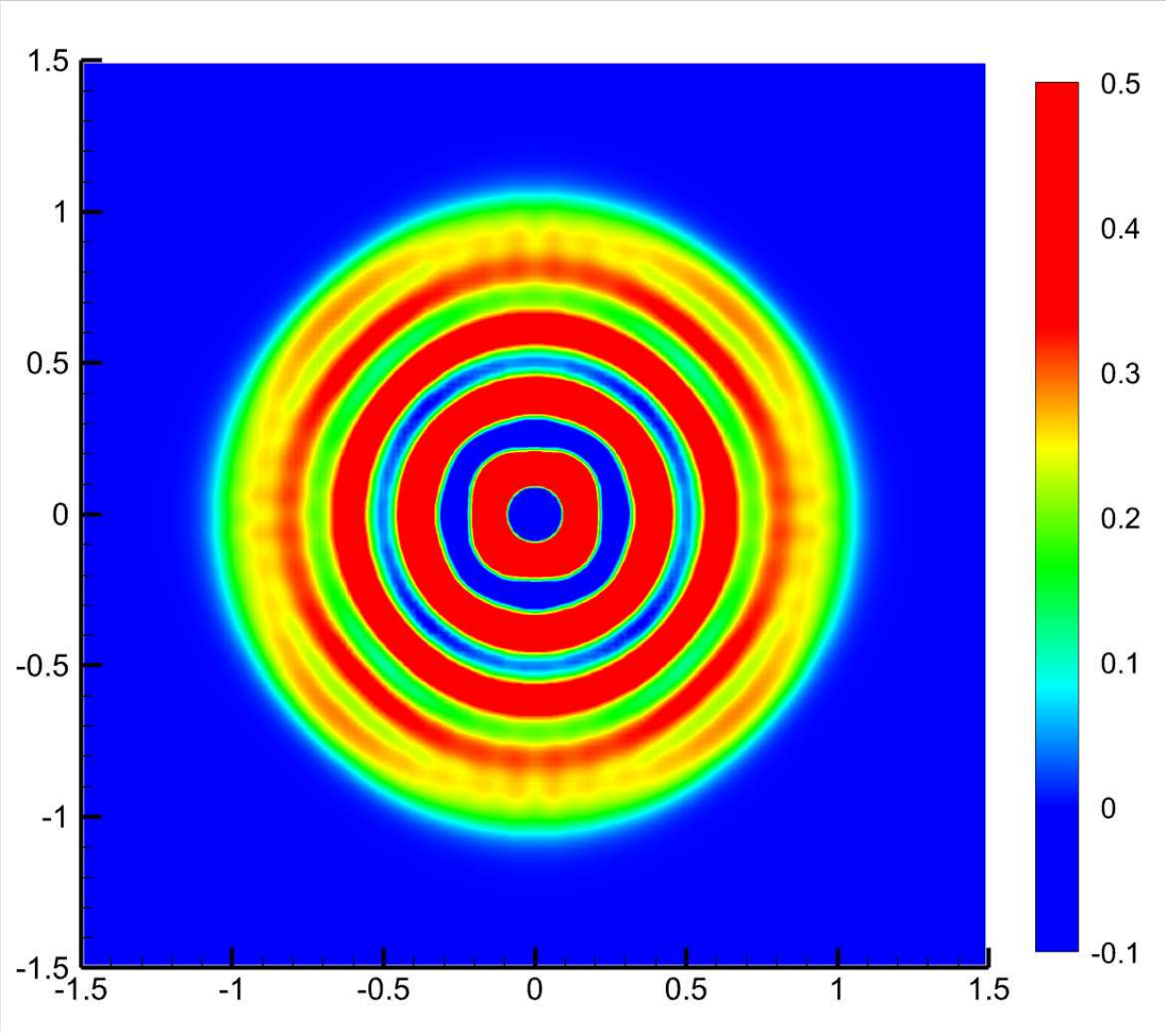}}
\centerline{\small (e) $P_{11}$ solution }
\end{minipage}
\begin{minipage}[h]{0.45\textwidth}
\centering
\centerline{\includegraphics[width=2.2in]{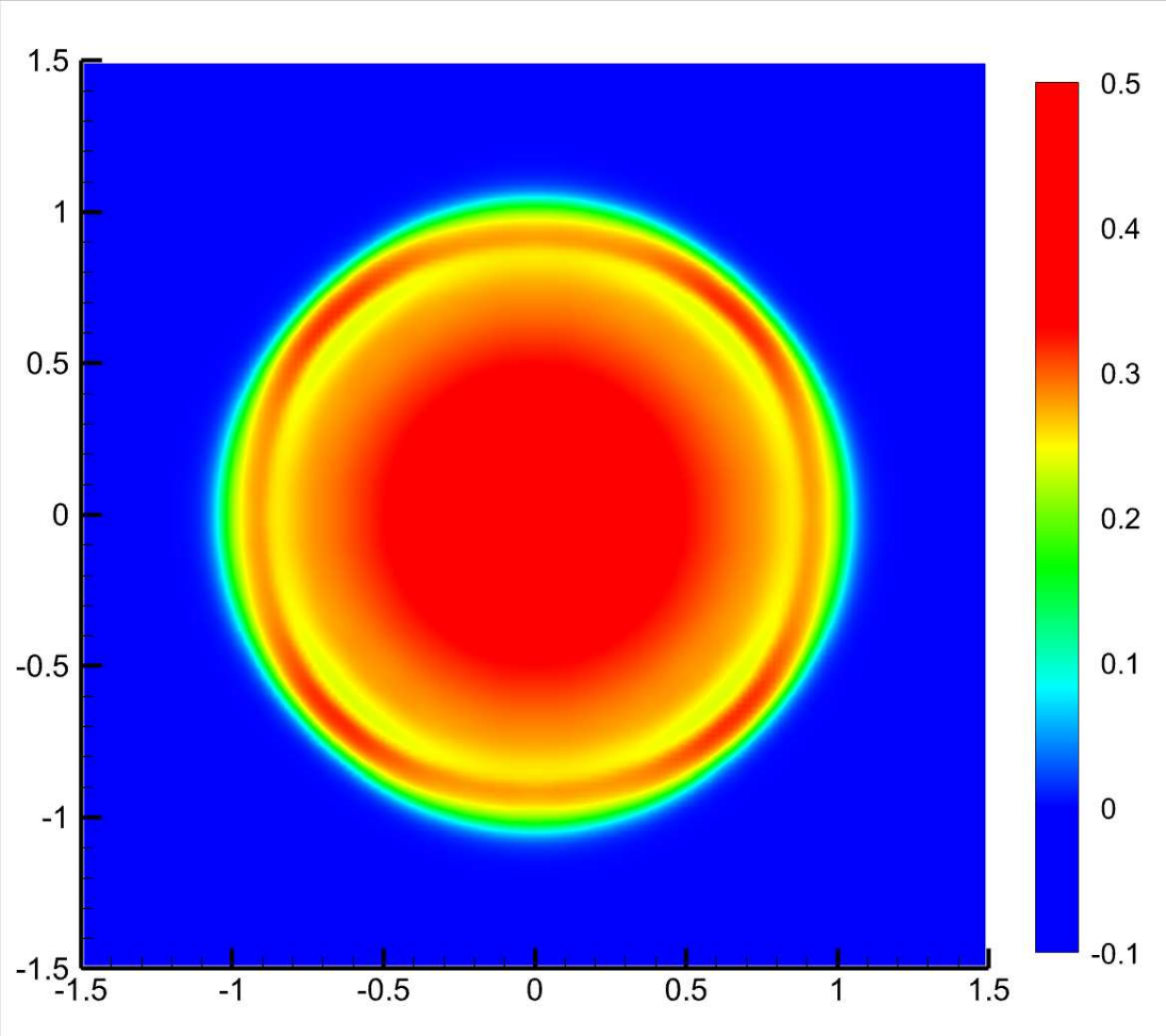}}
\centerline{\small (f) $PPFP_{11}$-$S$ solution}
\end{minipage}
}%\vskip -1.6pt
\vskip -8pt
\caption{\small {\sc Example 1}. The solution $\rho$ of the line source problem by the various methods.
}
\label{Figure: contour of line source}
\end{figure}

\begin{figure}[htbp]
{
\begin{minipage}[h]{0.45\textwidth}
\centering
\centerline{\includegraphics[width=2.5in]{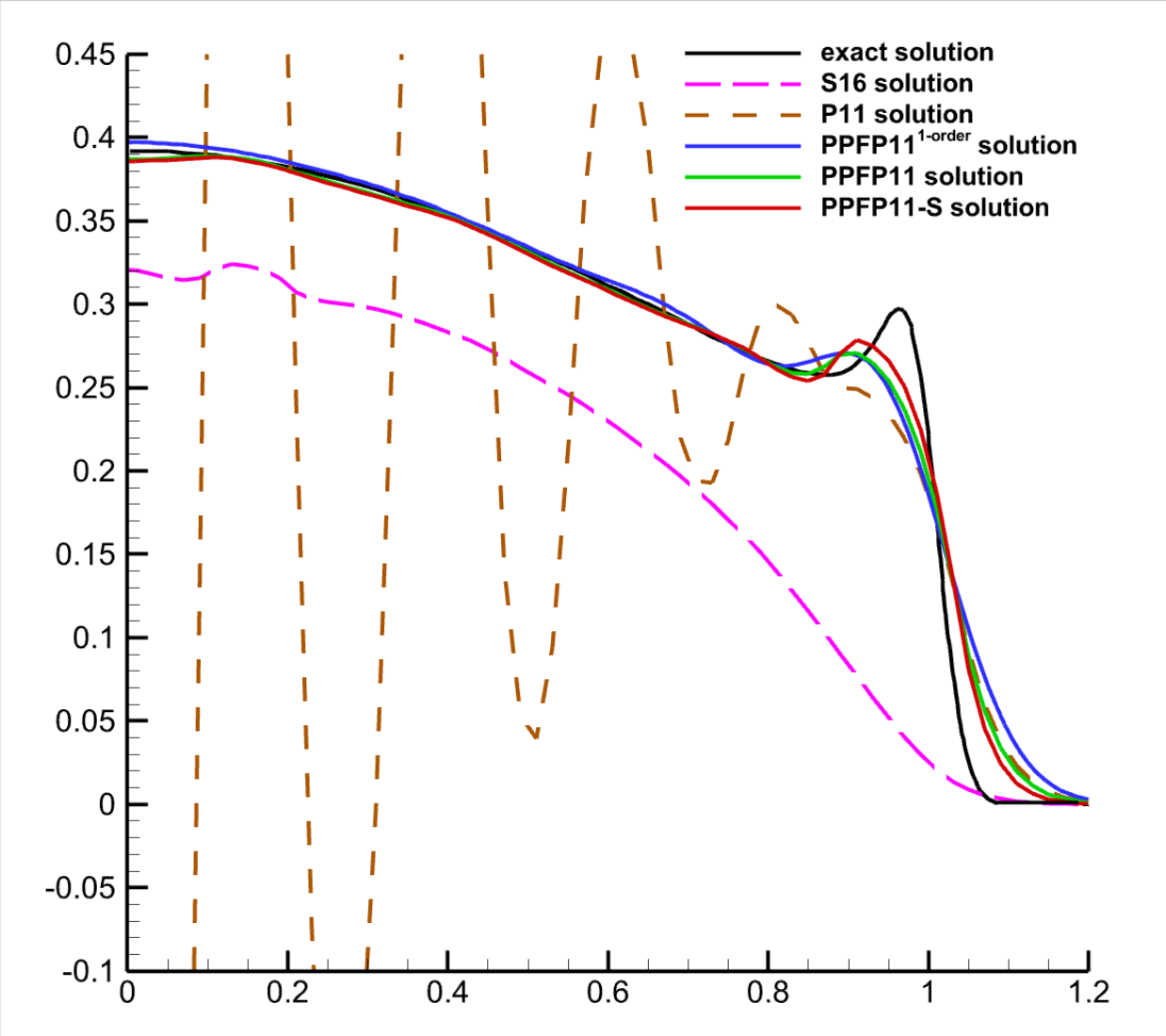}}
 \centerline{\small (a) along $x$ axis}
\end{minipage}
\begin{minipage}[h]{0.45\textwidth}
\centering
\centerline{\includegraphics[width=2.5in]{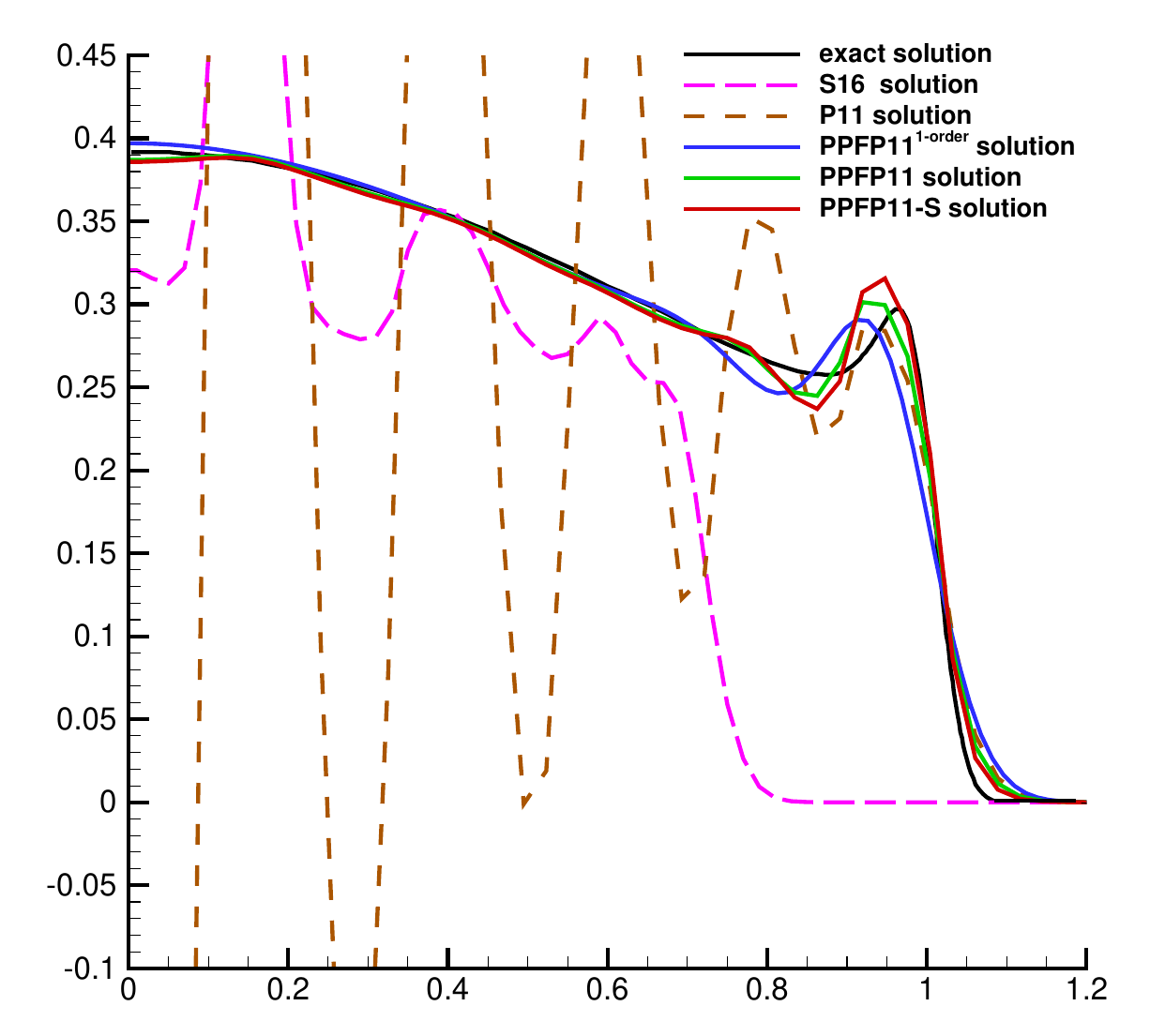}}%
\centerline{\small (b) along $y=x$}
\end{minipage}%
}%\vskip -1pt
 \vskip -6pt
\caption{\small {\sc Example 1.}
The line-outs of $\rho$ %along the nonnegative  $x$ axis  and $y=x$
by various methods.
}
\label{Figure: rho lineouts of line source}
\end{figure}
\vspace{0.1mm}
{\bf \noindent {\sc Example 2} (AP test).}
This example is used to test the asymptotic preserving property of the proposed schemes in 1D. The problem specifications of this example
are given in Table \ref{Problem specification of AP test}.
In this simulation, we take $200$ spatial meshes and time step $\Delta t = 0.1 \cdot \Delta x$.
Since the solution for the given coefficient $\epsilon$ is generally
not oscillatory, $PPFP_3$ and $PPFP_3$-$S$ with $\sigma_f=0$ are used to discretize the angular variable.
 \begin{table}[h]
\centering
 \caption{\small  Problem specifications of {\sc Example 2}. }
\begin{tabular}{ll}%{c|cc}
 \hline
 Spatial domain: & \qquad   $[0,1]$  \\
 Parameters: & \qquad  $\sigma=1, c=1, \epsilon=10^{-8}$  \\
 Boundary conditions:& \qquad  $I(0,\mu,t)=1$ for $\mu>0$ \\
                     & \qquad  $I(1,\mu,t)=0$ for $\mu<0$ \\
 Initial condition:&  \qquad $I(x,\mu,0)=0$
\\\hline
 \end{tabular}
 \label{Problem specification of AP test}
 \end{table}

Figure \ref{fig of linear kinetic 1D} shows the computed results of $\overline{\rho}$ with $\overline{\rho} = \rho/(4\pi)$
at $t=0.01,0.05,0.15,2.0$.
From Figure \ref{fig of linear kinetic 1D}, we see that both $PPFP_3$ and $PPFP_3$-$S$ solutions agree well with the reference solution. This
 implies that in the optically thick regime, both $PPFP_N$ and $PPFP_N$-$S$ schemes can recover the diffusion
solution accurately.
%Hence, this seems to verify the asymptotic preserving property of the current schemes.
Hence, this shows
the $AP$ property of the current schemes.

\begin{figure}[htbp]
\centering
\includegraphics[width=3.30in]{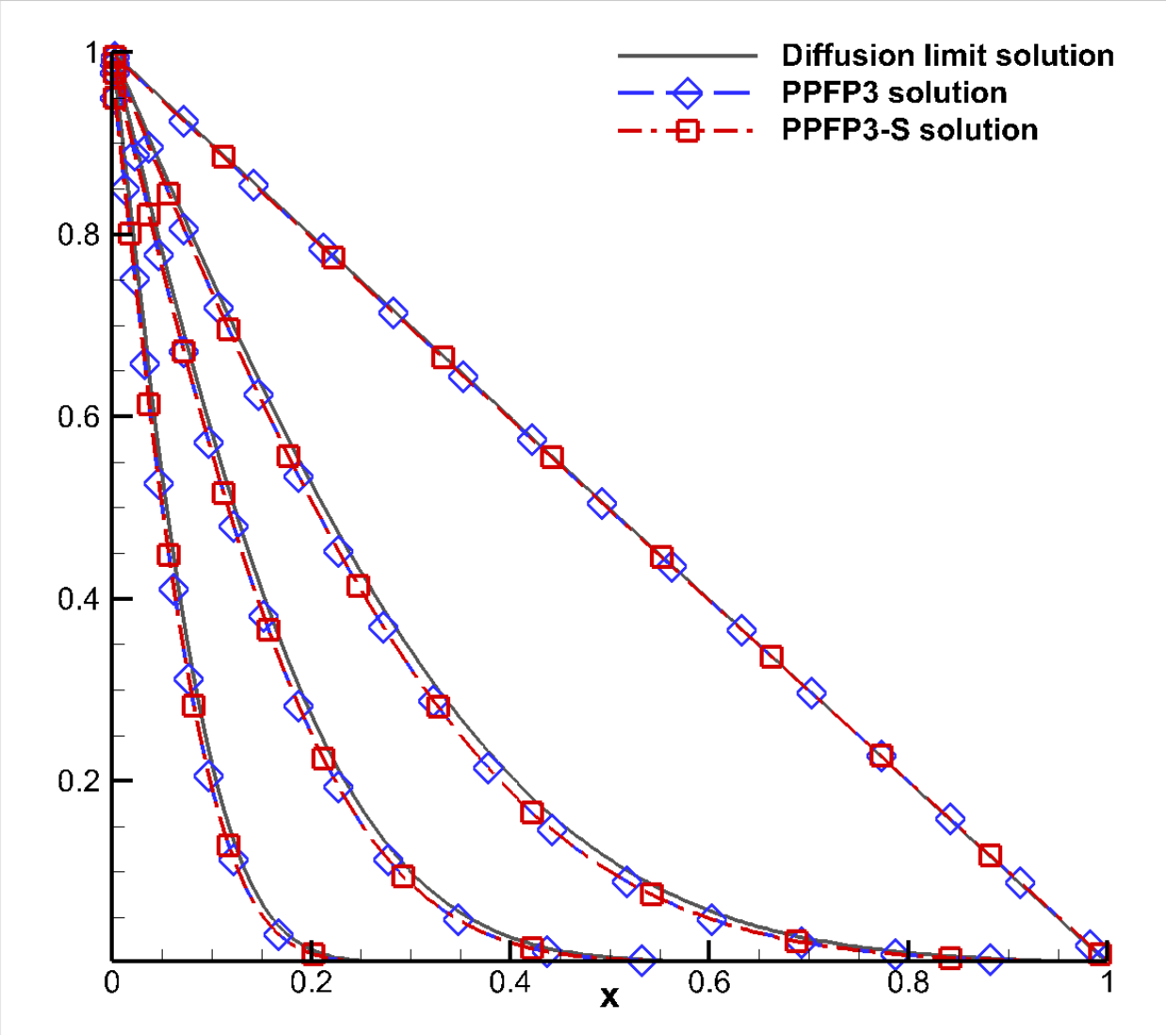}
  \caption{\small
  {\sc Example 2.}
 Linear kinetic transfer solution $\overline{\rho}$ at times 0.01, 0.05, 0.15, 2.0 (from left to right).
 }
\label{fig of linear kinetic 1D}
\end{figure}
\vspace{3mm}

The following Example 3 is given in the non-dimensional form, and
 is designed to test the accuracy of the proposed schemes to the nonlinear gray radiative transfer problem.
\vspace{3mm}

{\bf \noindent {\sc Example 3} (Accuracy test).}
This example aims at testing the spatial second-order accuracy of the proposed schemes in the regime $\epsilon\ll 1 $ and the regime $\epsilon=O(1)$.
Table \ref{Problem specification of Accurate test} shows the problem specifications of this example.
%It is a 1D problem.  The computational spatial domain
%is $[0,1]$. In this example, we take $\sigma=a=c=C_{\nu}=1$ and $\epsilon=1$, $10^{-1}$, $10^{-6}$, $10^{-9}$ respectively.
%The initial material temperature and radiation intensity are given by $T(x,0)=1+\frac{1}{2} sin(2\pi x)$ and $I(x,\Omega,0)= acT(x,0)^4$.
% Periodic  boundary condition is used.
And the final simulation time of this example is taken to be $t=0.2$.
 \begin{table}[h]
\centering
 \caption{\small  Problem specifications of {\sc Example 3}. }
\begin{tabular}{ll}%{c|cc}
 \hline
 Spatial domain: & \qquad   $[0,1]$  \\
 Parameters: & \qquad $\sigma=1, a=1, c=1, C_{\nu}=1$  \\
             & \qquad $\epsilon=1, 10^{-1}, 10^{-6}, 10^{-9}$ respectively \\
 Boundary conditions:& \qquad  periodic boundary   \\
 Initial condition:&  \qquad $T(x,0)=1+\frac{1}{2}sin(2\pi x)$ \\
  &\qquad   $I(x,\Omega,0)= acT(x,0)^4$
\\\hline
 \end{tabular}
 \label{Problem specification of Accurate test}
 \end{table}

$PPFP_7$ and $PPFP_7$-$S$ with $\sigma_f=0$ are applied to discretize the angular variable for all given values of $\epsilon$ in this simulation. Remind that by setting $N=7$ one
is already able to resolve the angular variable accurately enough, thus $\sigma_f=0$ is taken.
And also in this setting, the errors of the schemes mainly come from the spatial and temporal discretization.
Uniform spatial meshes are used with $N_x$ denoting the number of spatial cells.
In order to test the spatial accuracy, we take the time step to be $\Delta t=(\Delta x)^2$, which also
 satisfies the positive preserving constraint conditions and the CFL condition.
Here the numerical solution on a finer mesh $\Delta x /2$
is used as a reference solution to compute the relative $L_2$ and $L_{\infty}$ errors
of $\rho$ and $T$ on current mesh $\Delta x$.
Tables \ref{Table: accuracy test - rho}-\ref{Table: accuracy test - T} show the computed results.
From these Tables we clearly observe the second-order accuracy for given values of $\epsilon$.
 Thus, we have verify the spatial accuracy of the proposed schemes in the regimes $\epsilon\ll 1$ and $\epsilon=O(1)$.

 In addition to the accuracy test, to test the asymptotic preserving property in this case, we denote the asymptotic preserving error by
 $$
err_{ap}
=  \max_{i,j} \Big\{~ |I_{0,i,j}^0-\frac{\phi_{i,j}}{2\sqrt{\pi}}| + \sum_{ 1\leq\ell\leq N, -\ell \leq m \leq \ell } |I_{\ell,i,j}^m| ~\Big\}.
$$
To see the behavior of the proposed schemes as $\epsilon \rightarrow 0$, we present the time evolution of the error $err_{ap}$
in Figure \ref{Fig:  err_ap}. Note that the results are computed by using the $PPFP_7$ and $PPFP_7$-$S$ schemes
with $\sigma_f=0$, $N_x=100$ and $\Delta t = 0.25\Delta x$. From Figure \ref{Fig:  err_ap} we know that for any given small $\epsilon$,
the error decreases with time first, and then reaches a fixed steady state,
the value of which is within machine error of the double precision operation used in the code.
Moreover, for any one of the two schemes, as $\epsilon$ decreases, the error gets smaller at the beginning.
It should be pointed out that the observed difference between $PPFP_7$ and $PPFP_7$-$S$ at early time
is within the numerical error of the schemes.
Thus, we demonstrate the asymptotic preserving property of  the proposed $PPFP_N$ and $PPFP_N$-$S$ schemes.
%Thus, we have demonstrated
%the $AP$ property of the proposed $PPFP_N$ and $PPFP_N$-$S$ schemes.
%Hence, this seems to verify the asymptotic preserving property of the current schemes.

\begin{table}[h]
{\footnotesize%scriptsize
\centering
\begin{tabular}{|c|c|c|c|c|c|c|c|c|c|}%{c|cc}
\hline \multicolumn{2}{|c|} {} & %& &
\multicolumn{4}{|c|}
{$PPFP_7$} & \multicolumn{4}{|c|}
{$PPFP_7$-$S$}
% & & & & &
   \\ \hline
   $\epsilon$ &  $N_x$ & re($L_2^{\rho}$)  & order  & re($L_{\infty}^{\rho}$)    & order
   & re($L_2^{\rho}$)  & order  &re($L_{\infty}^{\rho}$) & order  \\ \hline
   \multirow{5}{0.6cm}{ ~$1$}% {\rotatebox{90}{$10^{-6}$} }
    %& \hline
               &   20   & 1.01E-02 & -      & 1.31E-02   & -    & 9.51E-03 & -    & 1.30E-02 &  -    \\
               &   40   & 3.59E-03 & 1.49   & 4.75E-03   & 1.46 & 3.43E-03 & 1.47 & 4.74E-03 & 1.46  \\
               &   80   & 9.73E-04 & 1.88   & 1.42E-03   & 1.74 & 9.34E-04 & 1.88 & 1.41E-03 & 1.75  \\
               &   160  & 2.59E-04 & 1.91   & 4.28E-04   & 1.73 & 2.49E-04 & 1.91 & 4.26E-04 & 1.73  \\
               &  320   & 6.68E-05 & 1.95   & 1.13E-04   & 1.93 & 6.45E-05 & 1.95 & 1.12E-04 & 1.93  \\\hline
   \multirow{5}{0.6cm}{ $10^{-1}$ } % {\rotatebox{90}{$10^{-9}$} }
                &  20 & 4.89E-03 & -    & 8.44E-03 & -    &3.44E-02 &-    & 4.27E-02 & -    \\
                &  40 & 2.97E-03 & 0.72 & 4.36E-03 & 2.02 &7.87E-03 &2.13 & 9.80E-03 &  2.12  \\
                &  80 & 8.53E-04 & 1.80 & 1.28E-03 & 1.99 &1.98E-03 &1.99 & 2.46E-03 &  1.99  \\
                &  160& 2.27E-04 & 1.91 & 3.49E-04 & 1.99 &4.95E-04 &2.00 & 6.15E-04 &  2.00  \\
                & 320 & 5.84E-05 & 1.96 & 9.08E-05 & 2.00 &1.24E-04 &2.00 & 1.53E-04 &  2.00  \\ \hline
   \multirow{5}{0.6cm}{ {$10^{-6}$} }% {\rotatebox{90}{$10^{-6}$} }%& \hline
               &   20   & 3.78E-03 & -      & 6.42E-03   & -    & 3.78E-03 &-     & 6.42E-03 &-    \\
               &   40   & 3.45E-04 & 3.46   & 5.74E-04   & 3.48 & 3.45E-04 & 3.46 & 5.74E-04 & 3.48  \\
               &   80   & 8.62E-05 & 2.00   & 1.43E-04   & 2.00 & 8.62E-05 & 2.00 & 1.43E-04 & 2.00  \\
               &   160   & 2.16E-05 & 2.00  & 3.58E-05   & 2.00 & 2.16E-05 & 2.00 & 3.59E-05 & 2.00  \\
               &  320   & 5.36E-06 & 2.01   & 8.89E-06   & 2.01 & 5.46E-06 & 1.98 & 9.11E-06 & 1.98  \\ \hline
   \multirow{5}{0.6cm}{ $10^{-9}$ } % {\rotatebox{90}{$10^{-9}$} }
                &  20 & 3.78E-03 & -    & 6.42E-03 & -    & 3.78E-03 & -    & 6.42E-03 & -    \\
                &  40 & 3.45E-04 & 3.46 & 5.74E-04 & 3.48 & 3.45E-04 & 3.46 & 5.74E-04 &  3.48  \\
                &  80 & 8.62E-05 & 2.00 & 1.43E-04 & 2.00 & 8.62E-05 & 2.00 & 1.43E-04 &  2.00  \\
                &  160& 2.16E-05 & 2.00 & 3.58E-05 & 2.00 & 2.16E-05 & 2.00 & 3.58E-05 &  2.00  \\
                & 320 & 5.39E-06 & 2.00 & 8.96E-05 & 2.00 & 5.39E-06 & 2.00 & 8.96E-06 &  2.00  \\ \hline
 \end{tabular}
 \caption{\small {\sc Example 3.} Relative errors of $\rho$ at $t=0.2$.}
 \label{Table: accuracy test - rho}
 }
 \end{table}

\begin{table}[h]
{\footnotesize
\centering
\begin{tabular}{|c|c|c|c|c|c|c|c|c|c|}
\hline \multicolumn{2}{|c|} {} &
\multicolumn{4}{|c|}
{$PPFP_7$} & \multicolumn{4}{|c|}
{$PPFP_7$-$S$}
%% & & & & &
   \\ \hline
   $\epsilon$ &  $N_x$ & re($L_2^{T}$)  & order  & re($L_{\infty}^{T}$)    & order
   & re($L_2^{T}$)  & order  &re($L_{\infty}^{T}$) & order  \\ \hline
   \multirow{5}{0.6cm}{ ~$1$}% {\rotatebox{90}{$10^{-6}$} }
    %& \hline
               &   20   & 3.32E-03 & -      & 4.34E-03   & -    & 3.18E-03 & -    & 4.20E-03 &  -    \\
               &   40   & 9.69E-04 & 1.78   & 1.23E-03   & 1.82 & 9.27E-04 & 1.78 & 1.20E-03 & 1.81  \\
               &   80   & 2.50E-04 & 1.96   & 3.81E-04   & 1.69 & 2.39E-04 & 1.96 & 3.73E-04 & 1.69  \\
               &   160  & 6.39E-05 & 1.97   & 1.07E-04   & 1.83 & 6.12E-05 & 1.96 & 1.05E-04 & 1.83  \\
               &  320   & 1.62E-05 & 1.98   & 2.87E-05   & 1.90 & 1.55E-05 & 1.98 & 2.80E-05 & 1.90  \\\hline
   \multirow{5}{0.6cm}{ $10^{-1}$ } % {\rotatebox{90}{$10^{-9}$} }
                &  20 & 1.35E-03 & -    & 2.77E-03 & -    &8.97E-03 &-    & 1.37E-02 & -    \\
                &  40 & 7.97E-04 & 0.76 & 1.45E-03 & 0.93 &2.04E-03 &2.14 & 3.08E-03 &  2.16  \\
                &  80 & 2.28E-04 & 1.80 & 4.27E-04 & 1.77 &5.13E-04 &1.99 & 7.72E-04 &  2.00  \\
                &  160& 6.08E-05 & 1.91 & 1.17E-04 & 1.87 &1.28E-04 &2.00 & 1.92E-04 &  2.00  \\
                & 320 & 1.56E-05 & 1.96 & 3.05E-05 & 1.94 &3.21E-05 &2.00 & 4.81E-05 &  2.00  \\ \hline
   \multirow{5}{0.6cm}{ {$10^{-6}$} }% {\rotatebox{90}{$10^{-6}$} }%& \hline
               &   20   & 9.83E-04 & -      & 2.00E-03   & -    & 9.83E-04 &-     & 2.00E-03 &-    \\
               &   40   & 9.81E-05 & 3.33   & 2.46E-04   & 3.02 & 9.81E-05 & 3.33 & 2.46E-04 & 3.02  \\
               &   80   & 2.45E-05 & 2.00   & 6.16E-05   & 2.00 & 2.45E-05 & 2.00 & 6.16E-05 & 2.00  \\
               &   160  & 6.13E-06 & 2.00   & 1.54E-05   & 2.00 & 6.14E-06 & 2.00 & 1.54E-05 & 2.00  \\
               &  320   & 1.52E-06 & 2.01   & 3.82E-06   & 2.01 & 1.56E-06 & 1.98 & 3.92E-06 & 1.98  \\ \hline
   \multirow{5}{0.6cm}{ $10^{-9}$ } % {\rotatebox{90}{$10^{-9}$} }
                &  20 & 9.83E-04 & -    & 2.00E-03 & -    & 9.83E-04 & -    & 2.00E-03 & -    \\
                &  40 & 9.81E-05 & 3.33 & 2.46E-04 & 3.02 & 9.81E-05 & 3.46 & 2.46E-04 &  3.02  \\
                &  80 & 2.45E-05 & 2.00 & 6.16E-05 & 2.00 & 2.45E-05 & 2.00 & 6.16E-05 &  2.00  \\
                &  160& 6.13E-06 & 2.00 & 1.54E-05 & 2.00 & 6.13E-06 & 2.00 & 1.54E-05 &  2.00  \\
                & 320 & 1.53E-06 & 2.00 & 3.85E-06 & 2.00 & 1.53E-06 & 2.00 & 3.85E-06 &  2.00  \\ \hline
 \end{tabular}
 \caption{\small {\sc Example 3.} Relative errors of $T$ at $t=0.2$.}
 \label{Table: accuracy test - T}
 }
 \end{table}

 \begin{figure}[htbp]
\centering
\includegraphics[width=3.50in]{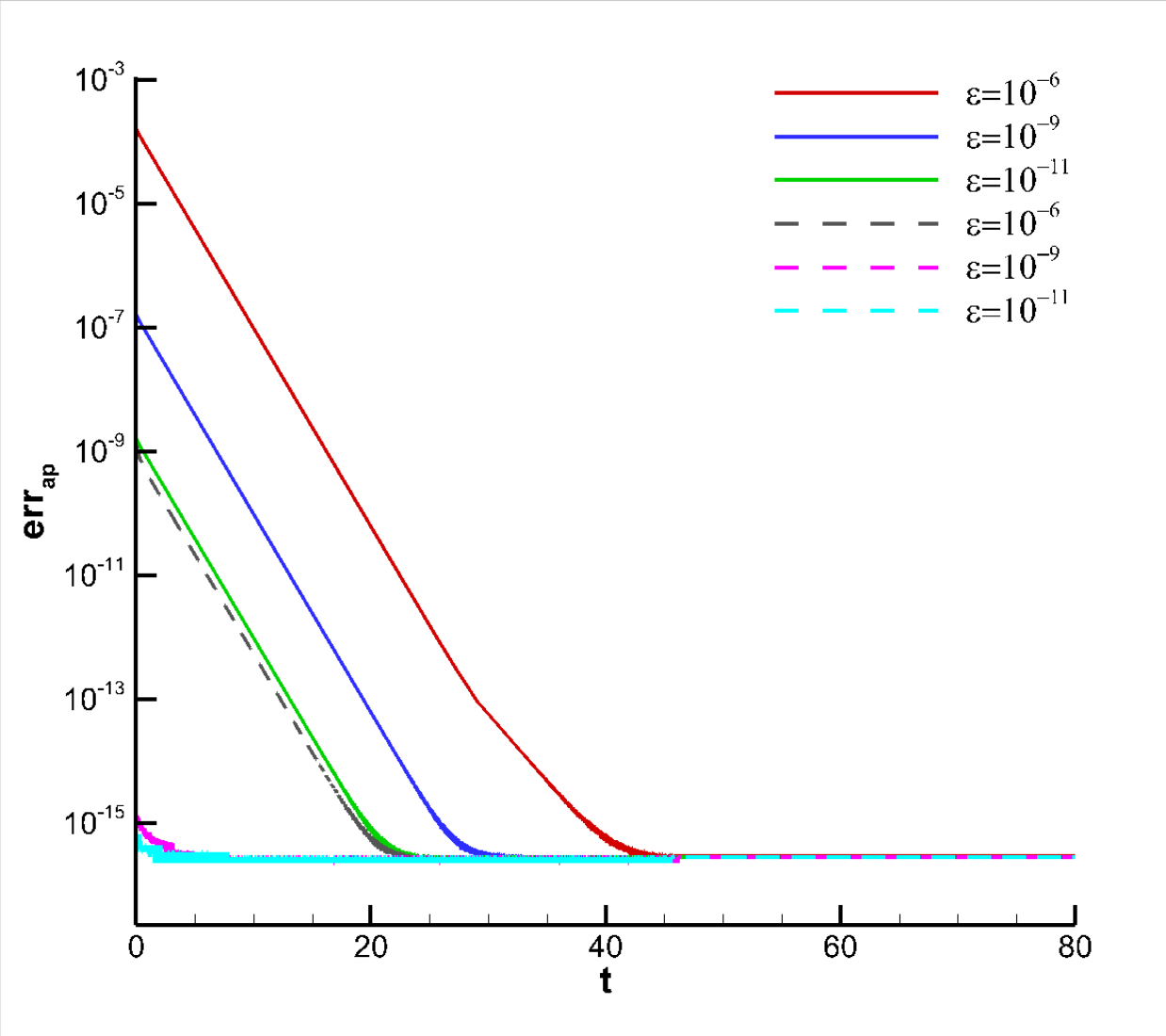}
  \caption{\small
  {\sc Example 3.}
  Time evolution of $err_{ap}$ by the $PPFP_7$ scheme (solid line) and $PPFP_7$-$S$ scheme (dashed line).
 }
\label{Fig:  err_ap}
\end{figure}
\vspace{3mm}

The remaining examples are all for the nonlinear gray radiative transfer equations in the dimensional form,
which are taken from Refs. \cite{N.A. Gentile-IMC-2001, Sun-Jiang-Xu-2015}.
In the following, the light speed $c=29.98$ cm/ns, the radiation constant $a=0.01372$ GJ/(cm$^3*$keV$^4$) and $\epsilon=1$.
\vspace{2mm}

{\bf \noindent {\sc Example 4} (Marshak wave-2B)}.
In this example, we take the temperature-dependent absorption/emission coefficient to be $\sigma=\frac{100}{T^3}$ cm$^2$/g,
the specific heat to be $0.1$ GJ/g/keV, and the density to be $3.0$ g/cm$^3$. The initial material temperature $T$ is set to be $10^{-6}$ keV.
The computational domain is a two-dimensional slab $[0\text{cm}, 1\text{cm}]\times [0\text{cm},  0.01\text{cm}]$.
  A constant isotropic incident radiation intensity with a Planckian distribution at $1$keV is kept on the left boundary.
 In the simulation, we take the spatial meshes to be $200\times 1$.
Since the solution for the given coefficient $\sigma$ is generally not oscillatory,
 $PPFP_3$ and $PPFP_3$-$S$ with $\sigma_f=0$  are used to discretize the angular variable.
As for the time step, we take $\Delta t = 0.7\cdot\Delta x/c$ for the $PPFP_3$ scheme and $\Delta t = 0.05\cdot\Delta x/c$
for the $PPFP_3$-$S$ scheme.
%\begin{table}[h]
%\centering
% \caption{\small  Problem specifications of {\sc Example 4} }
%\begin{tabular}{ll}%{c|cc}
% \hline
% Spatial domain: & \qquad   $[0,1]\times[0,0.01]$  \\
% Parameters:  & \qquad $\sigma=\frac{300}{T^3} cm^{-1}$, ~ $a= 0.01372
%                GJ/(cm^3*keV^4) $ \\
%           & \qquad $c=29.98 cm/ns$, ~ $C_{\nu}=0.3 GJ/cm^3/keV $ \\
% Boundary conditions:& \qquad  $T(0,y,t) = 1 keV$,~ $T(1,y,t) = 0 keV$    \\
%                     &\qquad $I(0,y,\xi,\eta,t) =\frac{ac T(0,y,t)^4}{4\pi}, \xi>0$, ~
%                     $I(1,y,\xi,\eta,t) = 0, \xi<0 $\\
%                     &\qquad the other boundaries are reflective boundary\\
% Initial condition:&  \qquad $T(x,y,0)=10^{-6} keV$ \\
%  &\qquad     $I(x,y,\Omega,0)= acT(x,y,0)^4$
%\\\hline
% \end{tabular}
% \label{Problem specification of Marshak wave-2B}
% \end{table}

%Figure \ref{Figure: Tr of Marshake 2B} plots  the computed radiation temperature $T_r \equiv(\frac{\rho}{a c})^{1/4}$  at times $15$, $30$, $45$, $60$, $74$ ns.  The $S_6$ solution  from \cite{Sun-Jiang-Xu-2015} is  presented as well in  Figure \ref{Figure: Tr of Marshake 2B}  for reference.
%We also show the computed material temperature $T$ at time $74$ ns in Figure \ref{Figure: diffusion limit of Marshake 2B}. Moreover, the diffusion solution

In Figure \ref{Figure: Tr of Marshake 2B}, we compare the radiation temperature $T_r \equiv(\frac{\rho}{a c})^{1/4}$
at times $15$, $30$, $45$, $60$, $74$ ns computed by the current $PPFP_3$ and $PPFP_3$-$S$ schemes
 with those in \cite{Sun-Jiang-Xu-2015} by $S_6$. The  numerical material temperatures $T$ at time $74$ ns are presented
 in Figure \ref{Figure: diffusion limit of Marshake 2B}. In addition, the diffusion solution is also given in
 Figure \ref{Figure: diffusion limit of Marshake 2B} for comparison.
From Figures \ref{Figure: Tr of Marshake 2B}-\ref{Figure: diffusion limit of Marshake 2B},
we clearly see that $T_r$ and $T$ computed by the current schemes show almost no difference with
the $S_6$ solution in \cite{Sun-Jiang-Xu-2015}. Moreover, the material temperatures agree well with the diffusion limit solution at $74$ ns.
This verifies that the proposed scheme can capture the right diffusive limit.

\begin{figure}[htbp]
\centering
\includegraphics[width=3.50in]{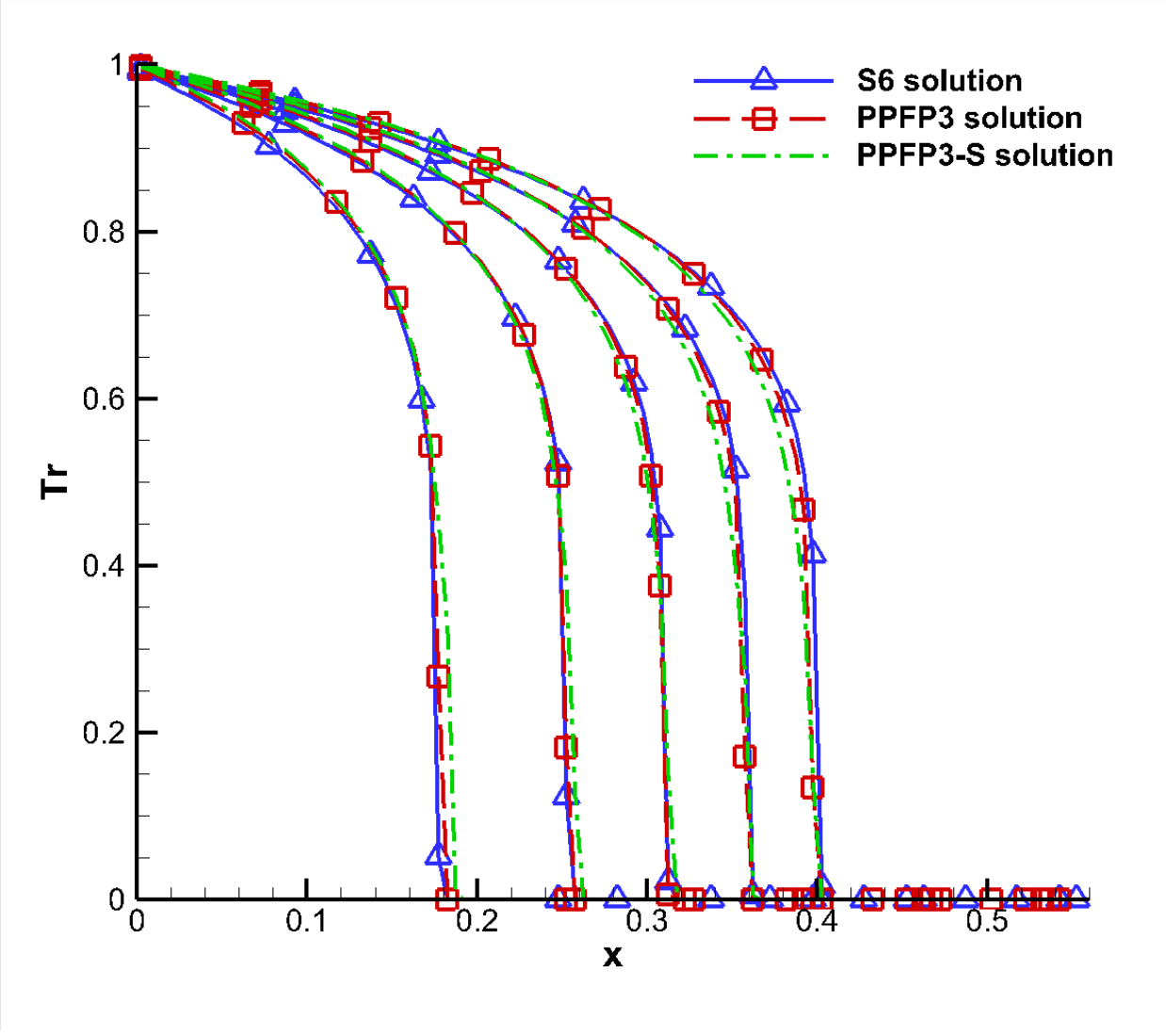}
 \vskip -6pt
\caption{\small {\sc Example 4.}
The radiation temperature $T_r$ at times 15, 30, 45, 60, 74 ns (from  left to right) for the Marshak wave-2B problem.
}
\label{Figure: Tr of Marshake 2B}
\end{figure}

\begin{figure}[htbp]
\centering
\includegraphics[width=3.50in]{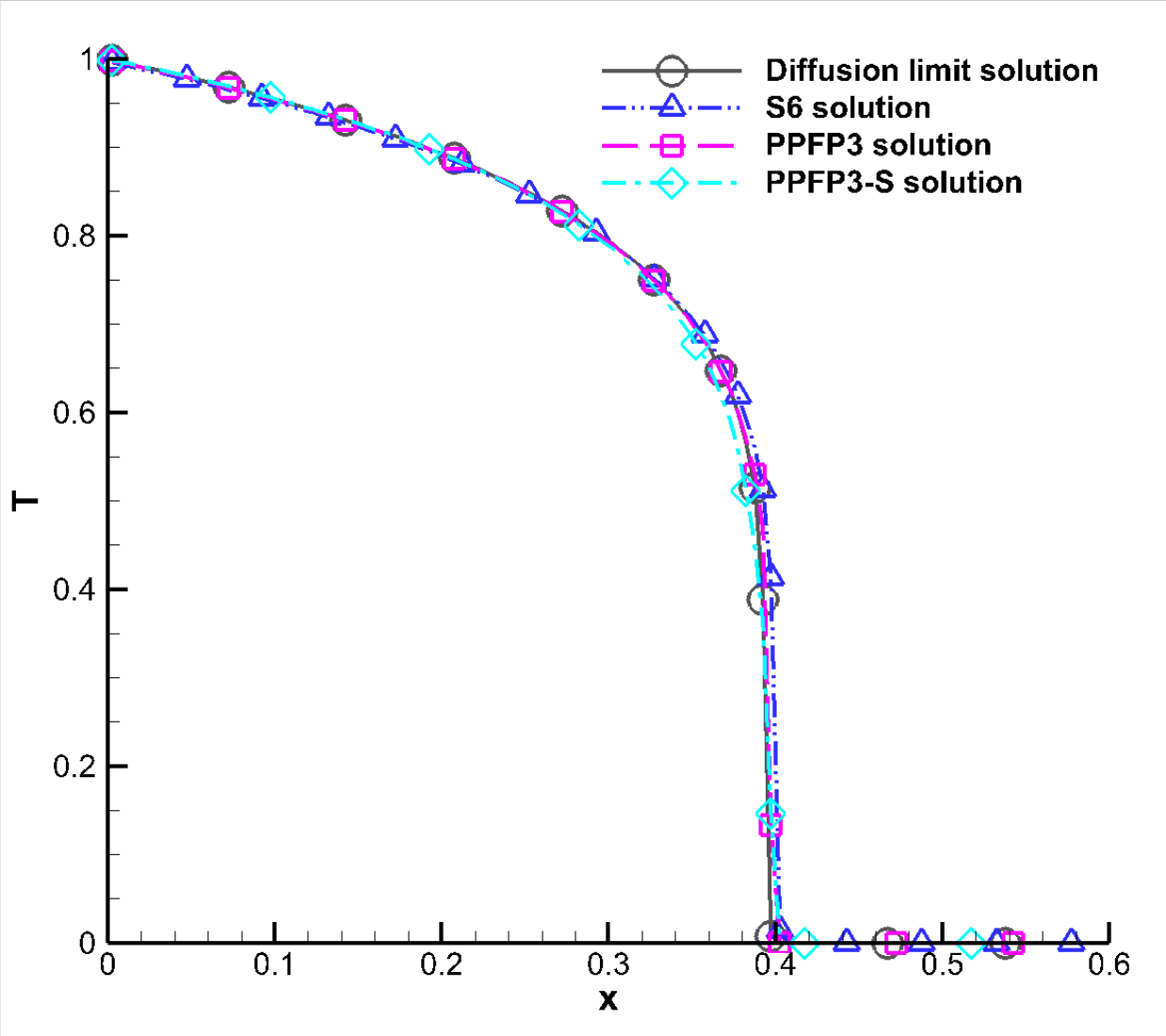}
 \vskip -6pt
\caption{\small {\sc Example 4.}
The numerical material temperature $T$ at time 74 ns for the Marshak wave-2B problem.
}
\label{Figure: diffusion limit of Marshake 2B}
\end{figure}
\vspace{3mm}
{\bf \noindent {\sc Example 5} (Marshak wave-2A).}
The Marshak wave-2A problem is almost the same as the Marshak wave-2B problem except that the coefficient $\sigma=\frac{10}{T^3}$ cm$^2$/g.
And also the same spatial meshes and angular discretization are used in the simulation of this problem. Here $\Delta t = 0.25\cdot\Delta x/c$
is taken.
Figures \ref{Figure: Tr of Marshake 2A} and \ref{Figure: diffusion limit of Marshake 2A} show the numerical results of $T_r$ and $T$.
It can be seen from Figures \ref{Figure: Tr of Marshake 2A}-\ref{Figure: diffusion limit of Marshake 2A} that the $PPFP_3$ solution
and $PPFP_3$-$S$ solution agree well with the $S_6$ solution in \cite{Sun-Jiang-Xu-2015}.
But the computed material temperatures $T$ of the current schemes are
quite different from the diffusion limit solution, which is indeed the case since the small absorption/emission coefficient
in this example violates the equilibrium diffusion approximation.

\begin{figure}[htbp]
\centering
\includegraphics[width=3.50in]{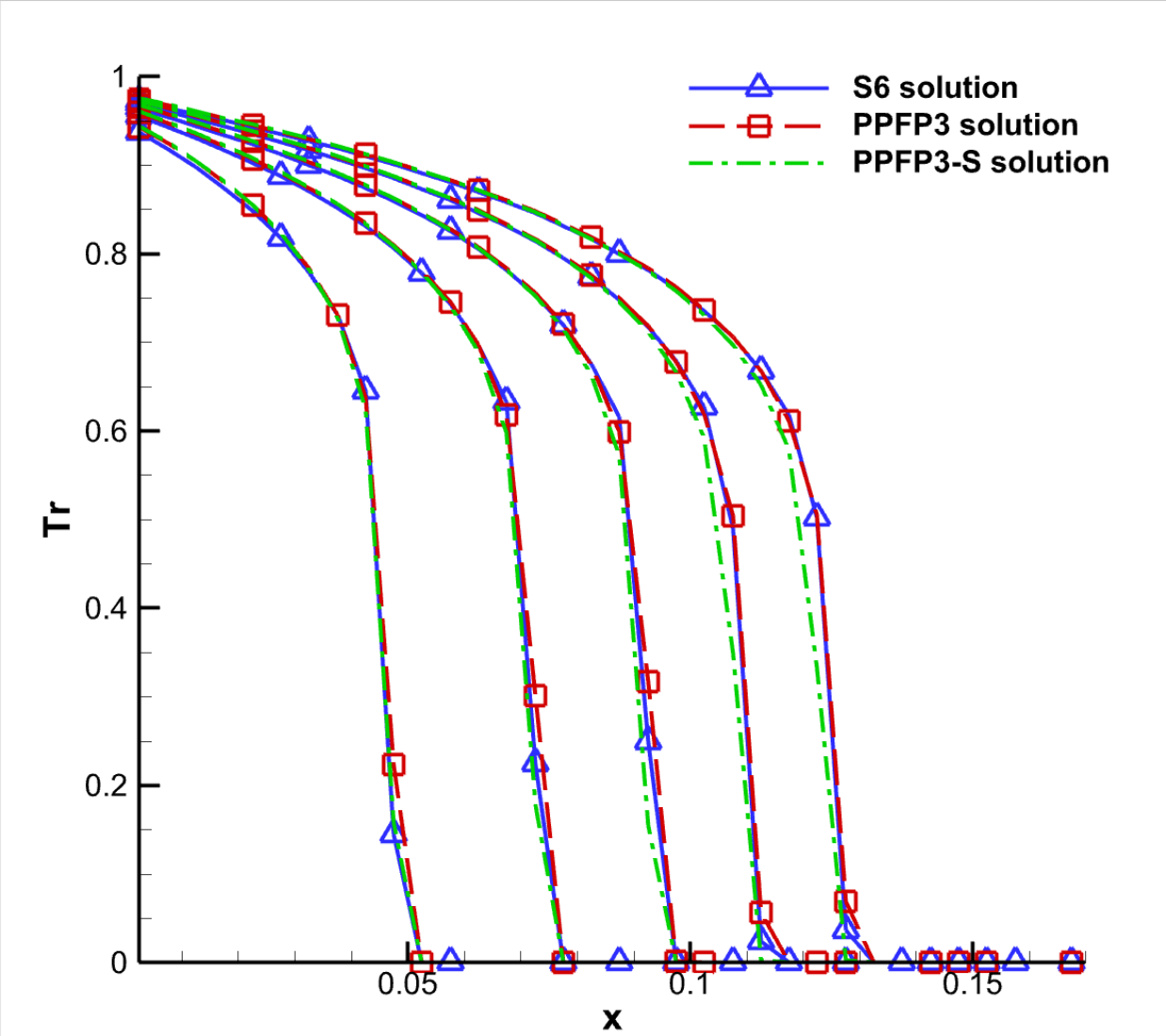}
 \vskip -6pt
\caption{\small {\sc Example 5.}
The radiation temperature $T_r$ at times 0.2, 0.4, 0.6, 0.8, 1.0 ns (from left to right) for the Marshak wave-2A problem.
}
\label{Figure: Tr of Marshake 2A}
\end{figure}

\begin{figure}[htbp]
\centering
\includegraphics[width=3.50in]{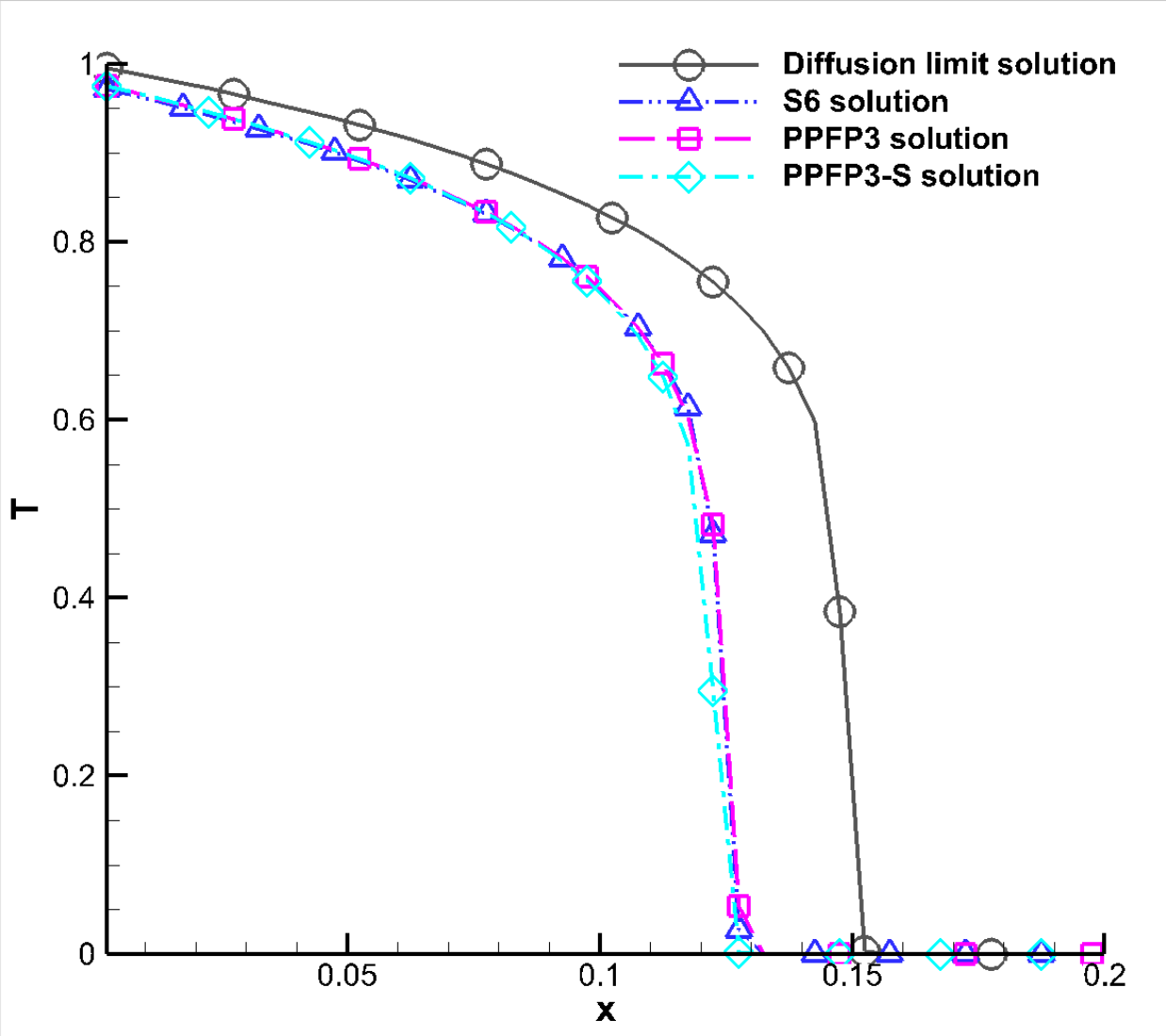}
 \vskip -6pt
\caption{\small {\sc Example 5.} The material temperature $T$ at time 1.0 ns for the Marshak wave-2A problem.
}
\label{Figure: diffusion limit of Marshake 2A}
\end{figure}
\vspace{3mm}

{\bf \noindent {\sc Example 6} (Tophat test).}
The initial configuration of this problem is given in Figure \ref{layout of Tophat test problem}. Initially,
the radiation temperature and material temperature are taken to be $0.05$keV everywhere.
A heating source with a fixed temperature $0.5$ keV is located on the left boundary for $-0.5< y < 0.5$.
We place five probes (A,B,C,D,E) in the thin opacity material, with the positions marked in Figure \ref{layout of Tophat test problem},
to monitor the change of the temperature.
%%%%%%%%%%%%%%%%%%%%%%%%%
%\begin{table}[h]
%\centering
% \caption{\small  Problem specifications of {\sc Example 6} }
%\begin{tabular}{ll}%{c|cc}
% \hline
% Spatial domain: & \qquad   $[0,7]\times[-2,2]$  \\
% Parameters: & \qquad $a= 0.01372 GJ/(cm^3*keV^4)$, $c=29.98 cm/ns$ \\
%           &  \qquad $C_{\nu}=0.1 GJ/g/keV $ \\
% Boundary conditions:& \qquad     \\
% Initial condition:&  \qquad $T(x,y,0) = 0.05 keV$ \\
%  &\qquad     $I(x,y,\Omega,0)= \frac{acT(x,y,0)^4}{4\pi}$
%\\\hline
% \end{tabular}
% \label{Problem specification of Tophat test}
% \end{table}

In the simulation, we take $128\times 64$ spatial meshes and let $\Delta t = 0.25\cdot min\{\Delta x, \Delta y\}/c$.
The $PPFP_7$ and $PPFP_7$-$S$ schemes, with Lanczos' filter function $f_{\mathrm{Lanczos}}(\lambda)=\frac{\sin (\lambda)}{\lambda}$
 and the space-dependent filter parameter $\sigma_f$ that is given in the table of Figure \ref{layout of Tophat test problem},
 are applied. In Figure \ref{Figure: T and Tr of five point in Tophat test}, we plot the time evolution of $T$ and $T_r$ at five probe points.
 The $S_{16}$ and $P_7$ solutions are presented in Figure \ref{Figure: T and Tr of five point in Tophat test}
 for reference and comparison.

It can be seen from Figure \ref{Figure: T and Tr of five point in Tophat test} that
 the material temperatures $T$ computed by $PPFP_7$, $PPFP_7$-S and $P_7$ agree well with the reference solution of $S_{16}$
 at all probes. As for the radiative temperatures $T_r$, they coincide at probes A,B,D and E. But for the probe C,
 the $P_7$ solution is oscillating heavily between $0.1$ns and $1$ns, while the $PPFP_7$ and $PPFP_7$-S schemes still give
  very accurate solutions.

In order to verify the positive preserving property of the proposed schemes,
we present the contours of $\rho$ and its minimal negative values in Figure \ref{Figure: contours of rho at t=0.5 and t=1.0 in Tophat test}
and Table \ref{Table: table of min rho in Tophat test}, respectively.
 From Figure \ref{Figure: contours of rho at t=0.5 and t=1.0 in Tophat test} and
 Table \ref{Table: table of min rho in Tophat test}, we can conclude that the $FP_7$ scheme indeed can damp oscillations in the $P_7$ approximation, and make the negative solution less likely.
   In addition, Figure \ref{Figure: contours of rho at t=0.5 and t=1.0 in Tophat test}
and Table \ref{Table: table of min rho in Tophat test} also show that,
with the time evolution, the absolute value of the negative quantity decreases.
  However, a negative solution still appears in the $FP_7$ scheme. On the other hand, the solutions of the current $PPFP_7$ and $PPFP_7$-$S$ schemes
  are always positive. Moreover, the accuracy of both $PPFP_7$ and $PPFP_7$-$S$ solutions can be assured,
  which can be seen in comparison with the reference solution of $P_{29}$.

  Figure \ref{Figure: T at different times in Tophat Test} shows the contours of the computed material temperature $T$
  by the $PPFP_7$ and $PPFP_7$-$S$ schemes, from which we can see that there is almost no difference between them. Moreover,
    comparing them with the numerical results from \cite{Sun-Jiang-Xu-2015},
  one can conclude that the current $PPFP_7$ and $PPFP_7$-$S$ schemes can still capture
  the interface between opacity thick and thin materials well.

\begin{figure}[h]
\centering
\begin{minipage}[h]{0.6\textwidth}
\centering
\includegraphics[width=0.9\textwidth]{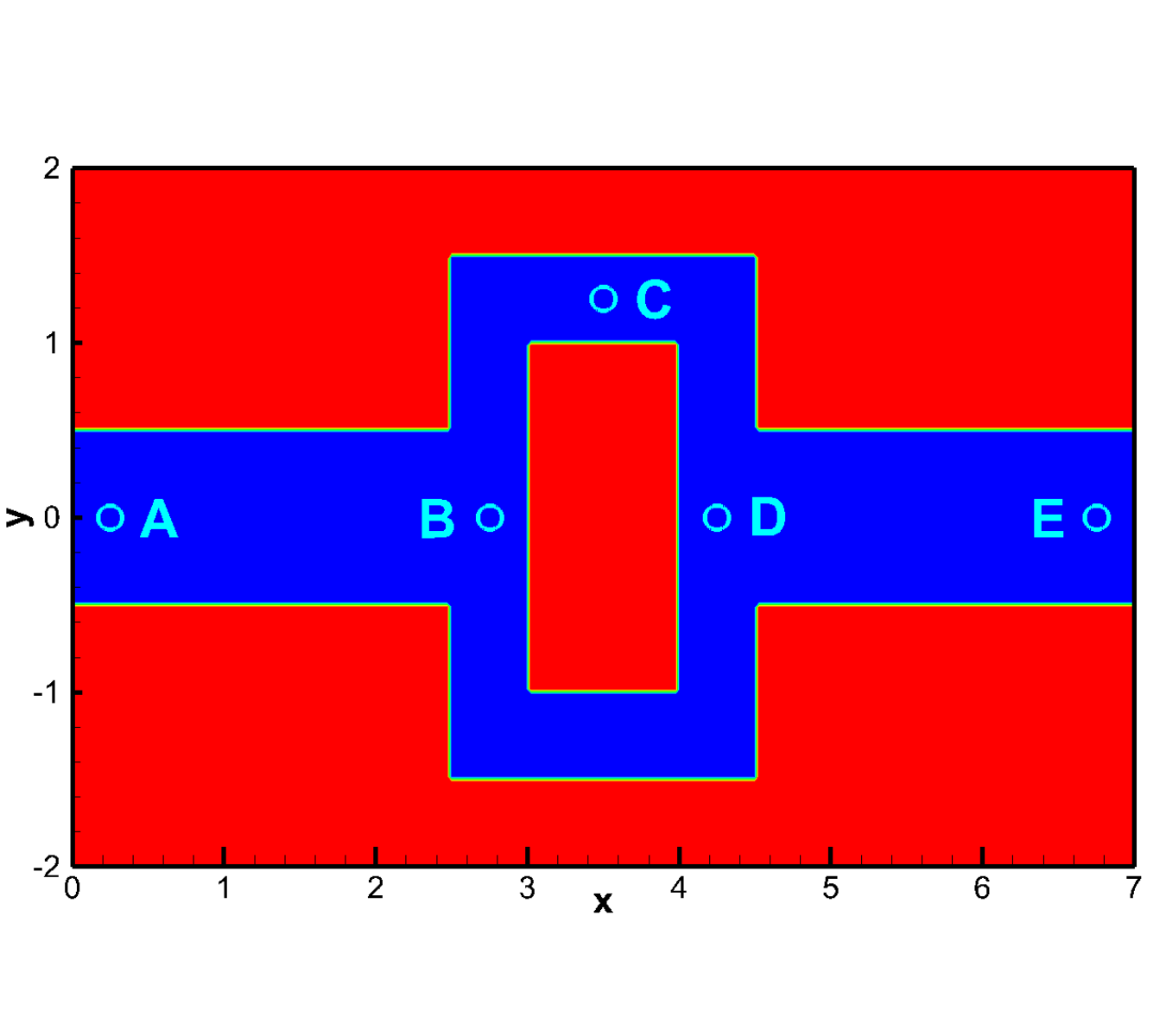}
\end{minipage}\\[-3mm]
\begin{minipage}[h]{1.0\textwidth}
\centering
\begin{tabular}{c|c|c|c}%{c|cc}
\hline
  ~& $\rho$ & $\sigma$  & $\sigma_f$\\ \hline
   blue region &  $0.01 $g/cm$^3$  & $0.2$ cm$^{-1}$ &  $2.0\times10^3$   \\ \hline
   red  regions & $10$ g/cm$^3$ & $2000$ cm$^{-1}$ & $0$       \\ \hline
 \end{tabular}
\end{minipage}\\%[-2mm]
 \caption{\small  {\sc Example 6.}
  Initial configuration of the Tophat test problem.} \label{layout of Tophat test problem}
 \end{figure}

%\begin{table}[h]
%\centering
%\begin{tabular}{ccccc}%{c|cc}
%\hline
% ~ min $\rho$  ~ &~~~~ $P_{7}$ ~~~~ &~~ $FP_{7}$ ~~ & ~ $PPFP_{7}$  ~ &  $PPFP_{7}$-$S$  \\ \hline
%  at $t=0.5$ ns  &  -2.32E-4    &  -3.32E-5  &  1.36E-5      &  1.28E-5     \\ \hline
%  at $t=1.0$ ns  &  -1.85E-4    &  -3.73E-6  &  1.36E-5      &  1.35E-5        \\ \hline
% \end{tabular}
% \caption{\small  {\sc Example 4.} Minimum negative values of the radiative  energy   density  $\rho$ . }
% \label{table of line source}
% \end{table}

 \begin{table}[h]
\centering
\begin{tabular}{ccccc}%{c|cc}
\hline
 ~ min $\rho$  ~ &~~~~ $P_{7}$ ~~~~ &~~ $FP_{7}$ ~~ & ~ $PPFP_{7}$  ~ &  $PPFP_{7}$-$S$  \\ \hline
  at $t=0.5$ ns  &  -2.32E-4    &  -3.32E-5  &  -      &  -     \\ \hline
  at $t=1.0$ ns  &  -1.85E-4    &  -3.73E-6  &  -      &  -        \\ \hline
 \end{tabular}
 \caption{\small  {\sc Example 6.} Minimal negative values of the radiative energy density $\rho$. }
 \label{Table: table of min rho in Tophat test}
 \end{table}

 \begin{figure}[htbp]
{
\begin{minipage}[h]{0.5\textwidth}
\centering
\centerline{\includegraphics[width=2.5in]{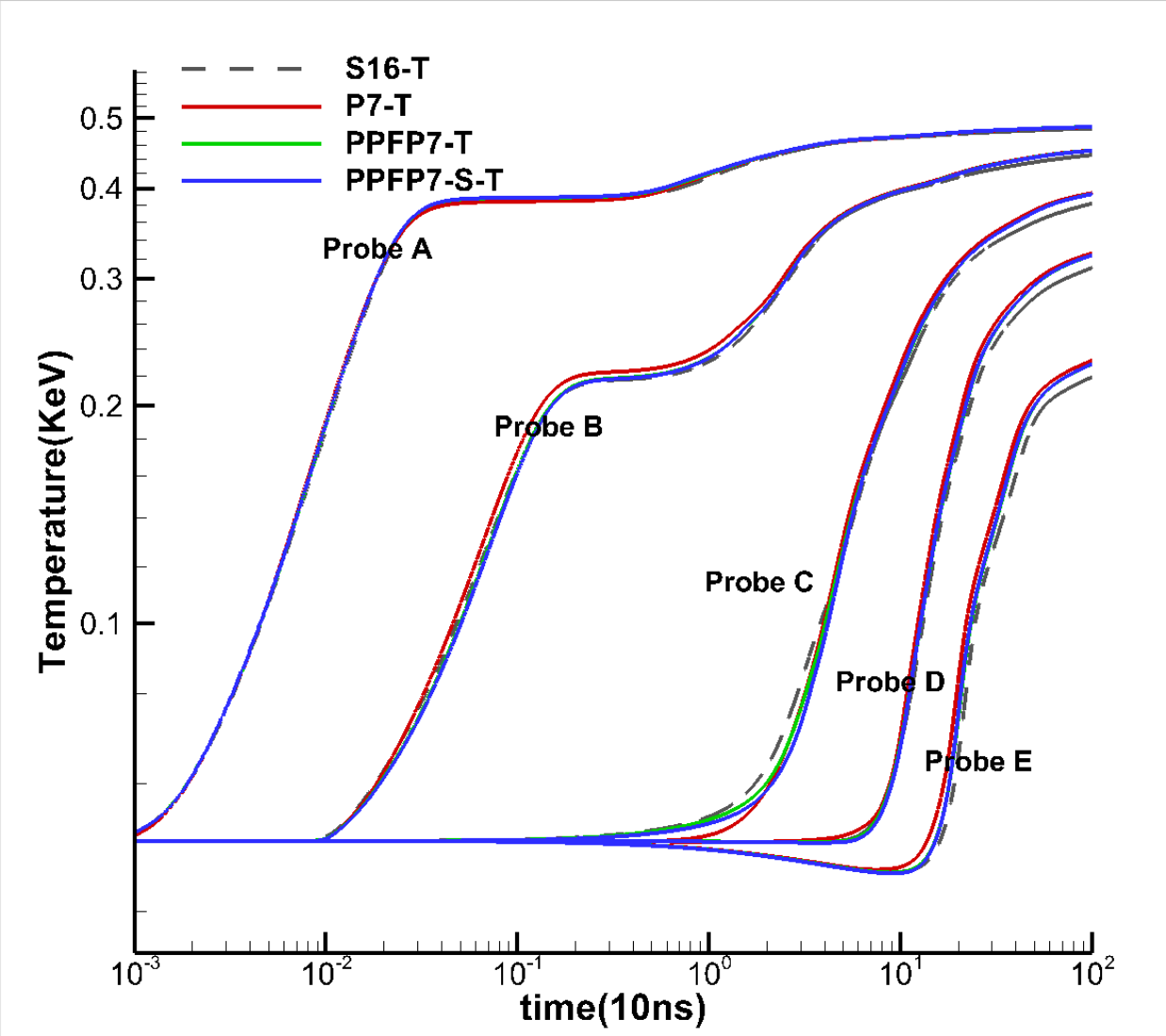}}
% \centerline{\small (a1) along $x$ axis}
\end{minipage}
\begin{minipage}[h]{0.5\textwidth}
\centering
\centerline{\includegraphics[width=2.5in]{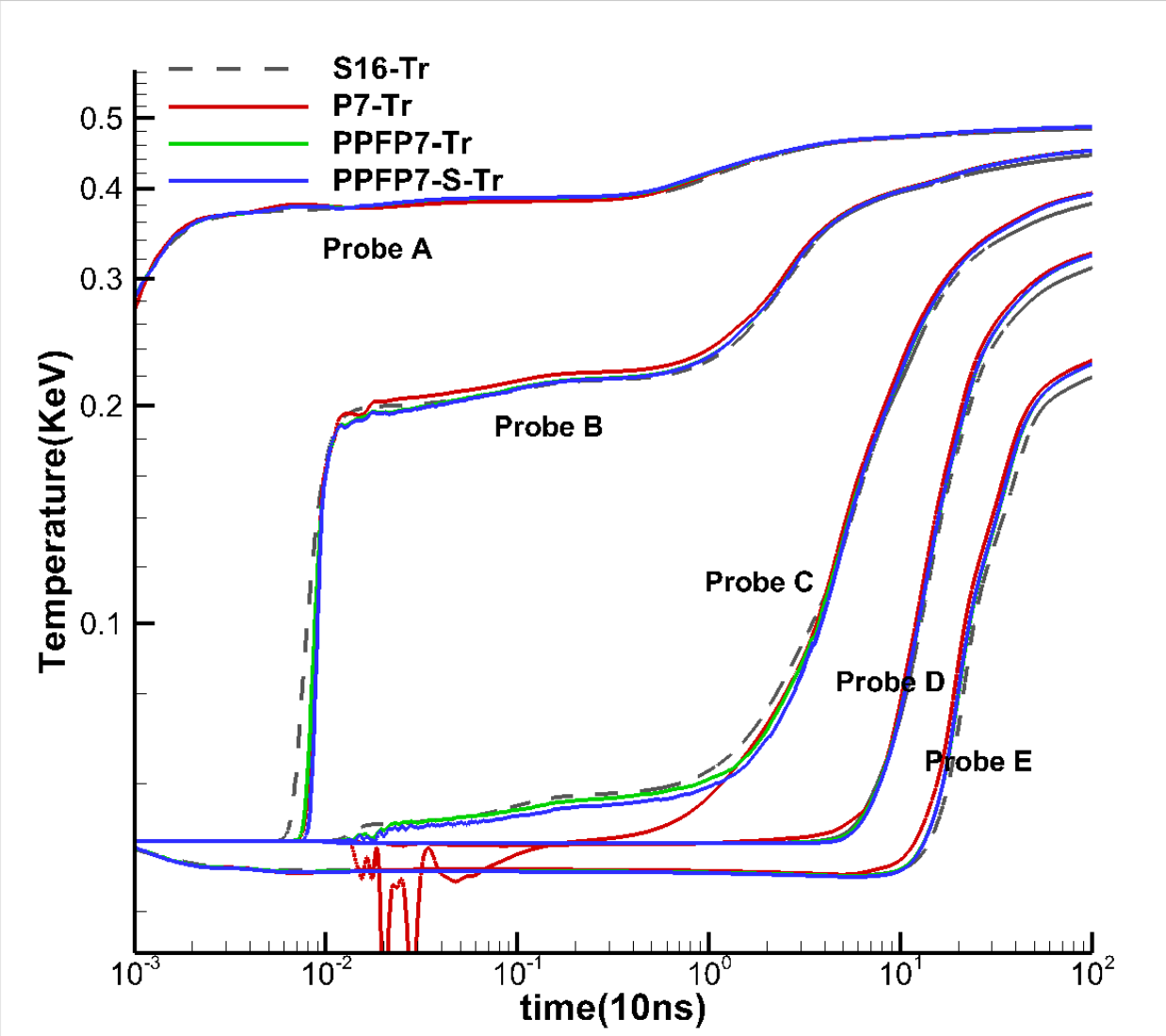}}%
%\centerline{\small (b1) along $y=x$}
\end{minipage}%
}%\vskip -1pt
 \vskip -6pt
\caption{\small {\sc Example 6.}
The time evolution of $T$ and $T_r$ at five probe points. The temperature unit is keV and the time unit is 10 ns.
}
\label{Figure: T and Tr of five point in Tophat test}
\end{figure}

\begin{figure}[htbp]
{
\begin{minipage}[h]{0.5\textwidth}
\centerline{ \hspace{25mm}
\includegraphics[width=1.25in]{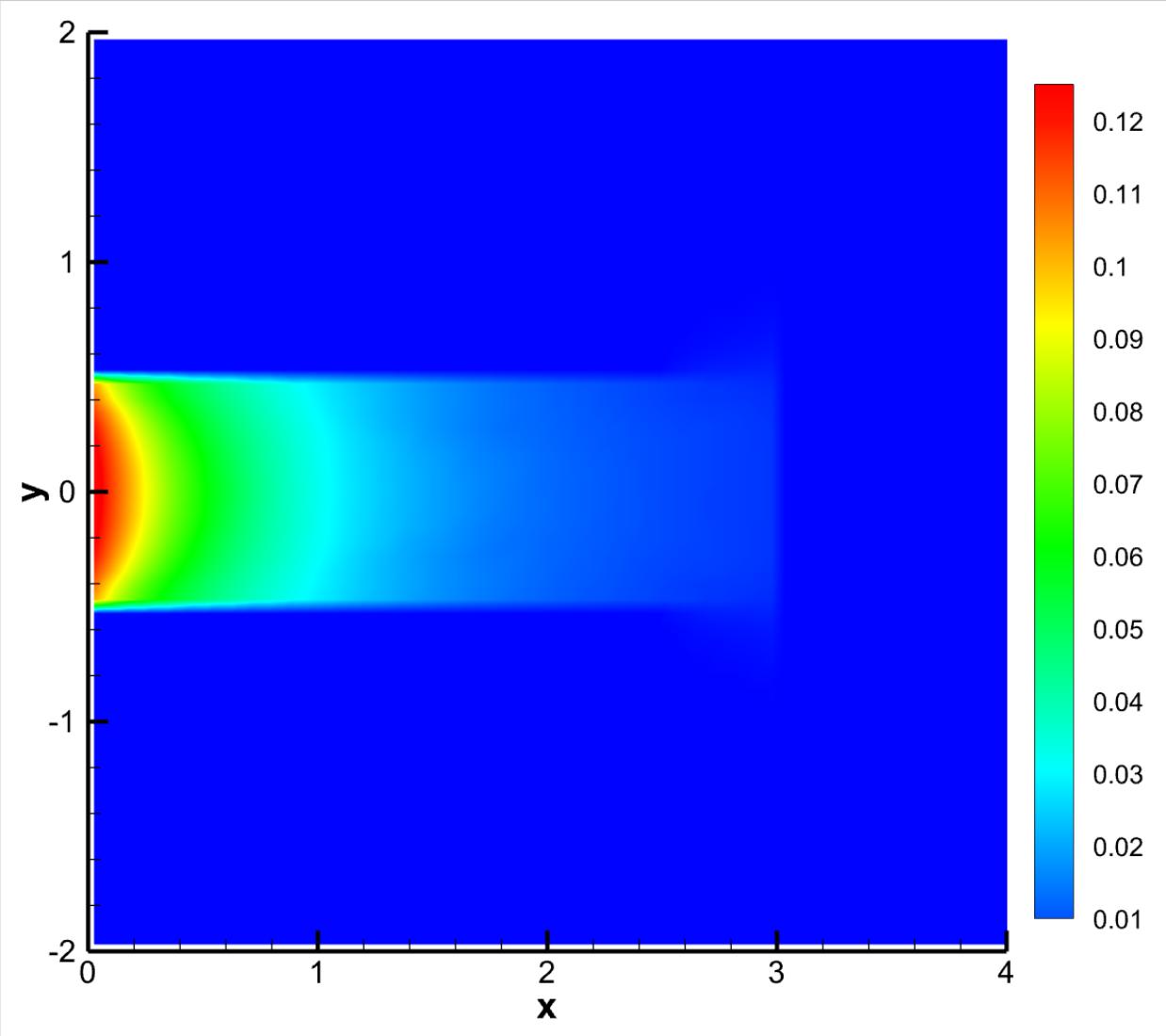}}
\vspace{-2mm}
 \centerline{  \small \hspace{32mm}  $P_{29}$ ~t=0.5ns }
\end{minipage}
\begin{minipage}[h]{0.5\textwidth}
\centerline{\hspace{-35mm}
\includegraphics[width=1.25in]{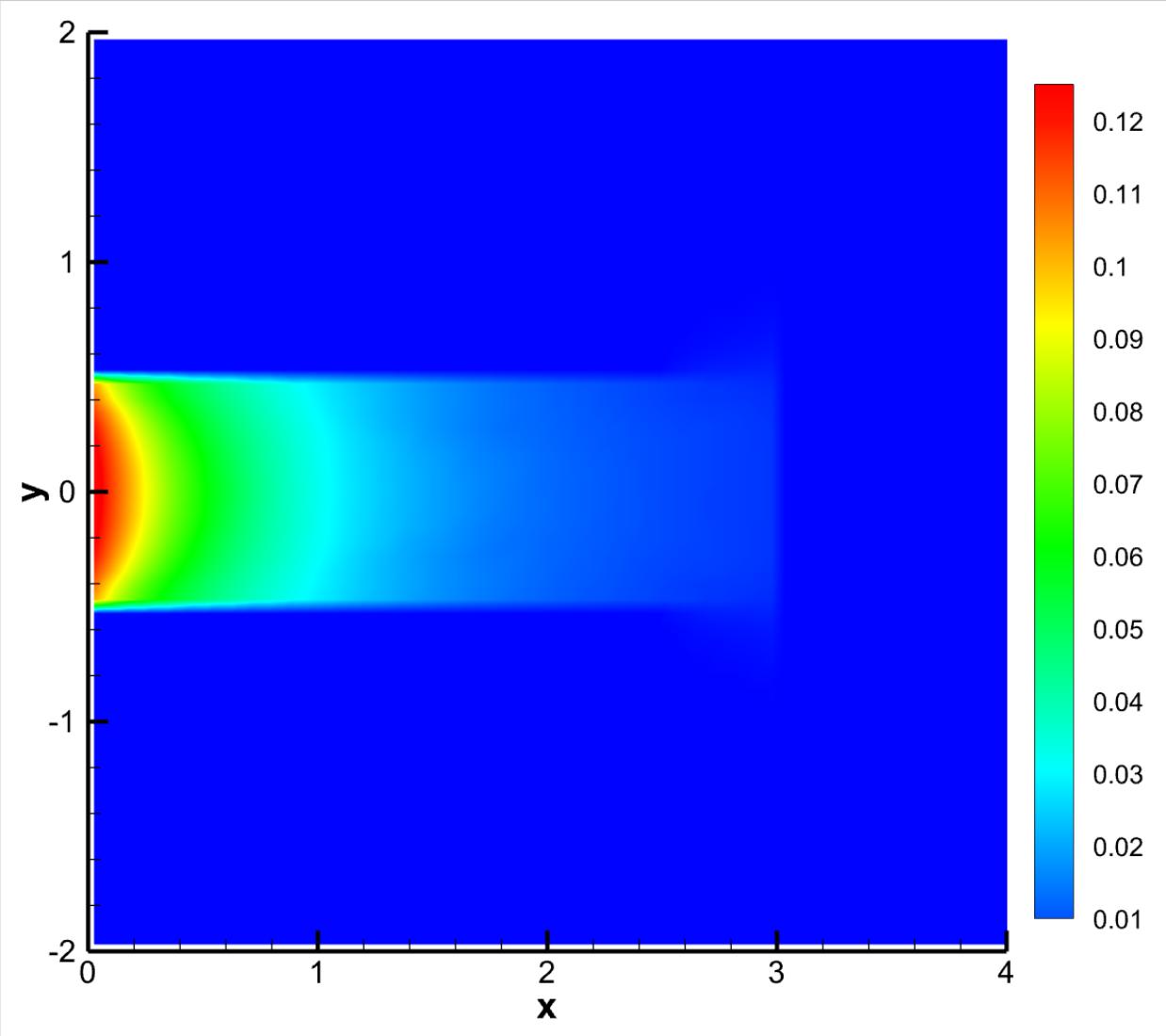}}
\vspace{-2mm}
\centerline{ \small \hspace{-43mm}  $P_{29}$ ~t=1.0ns }
\end{minipage}
}

{
\begin{minipage}[h]{0.24\textwidth}
%\centering
%\hspace{2mm}
\centerline{
\includegraphics[width=1.25in]{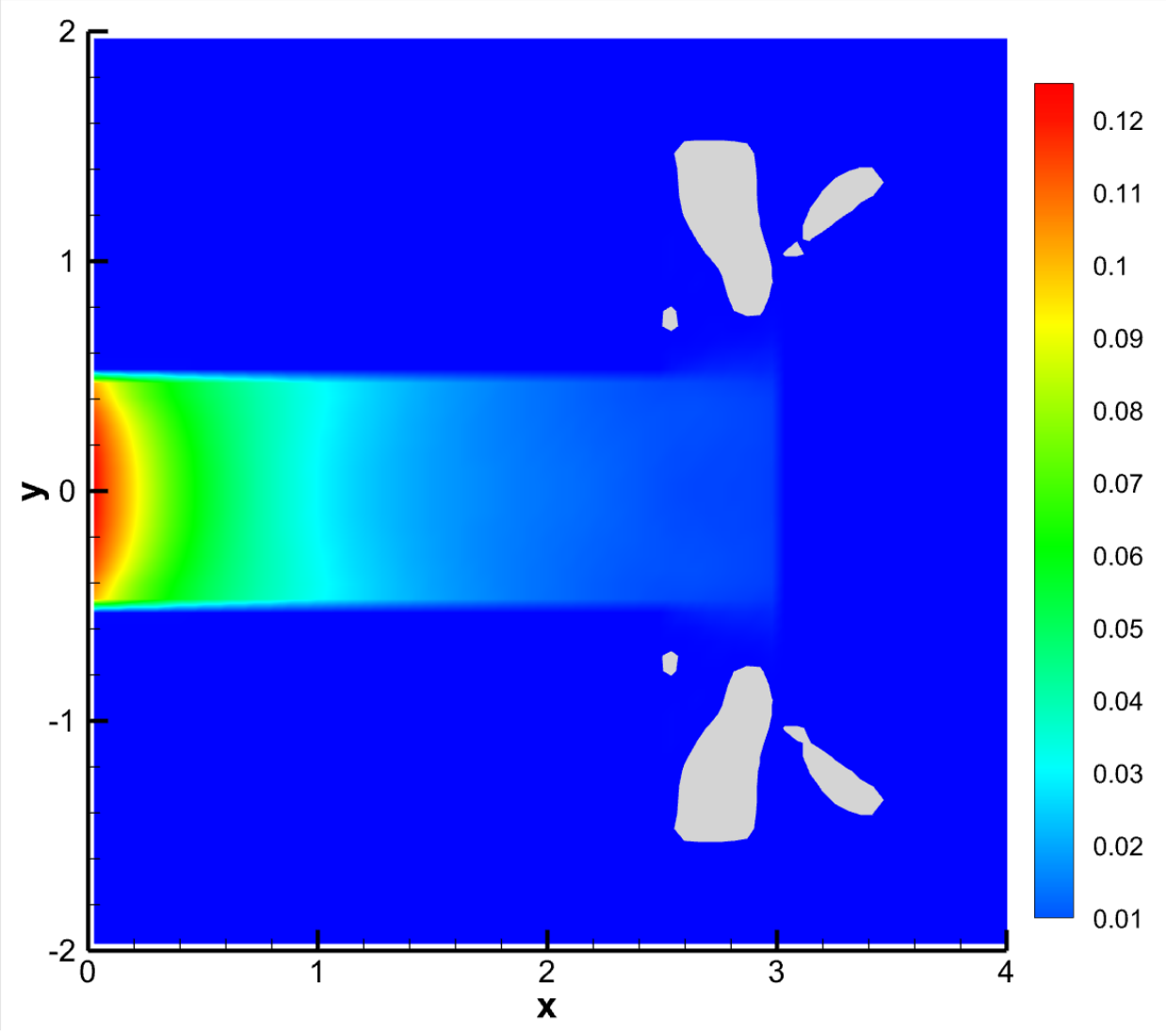}}
\vspace{-2mm}
 \centerline{  \small (a1) $P_7$ ~t=0.5ns}
\end{minipage}
\begin{minipage}[h]{0.24\textwidth}
%\centering
\centerline{%\hspace{3.5cm}
\includegraphics[width=1.25in]{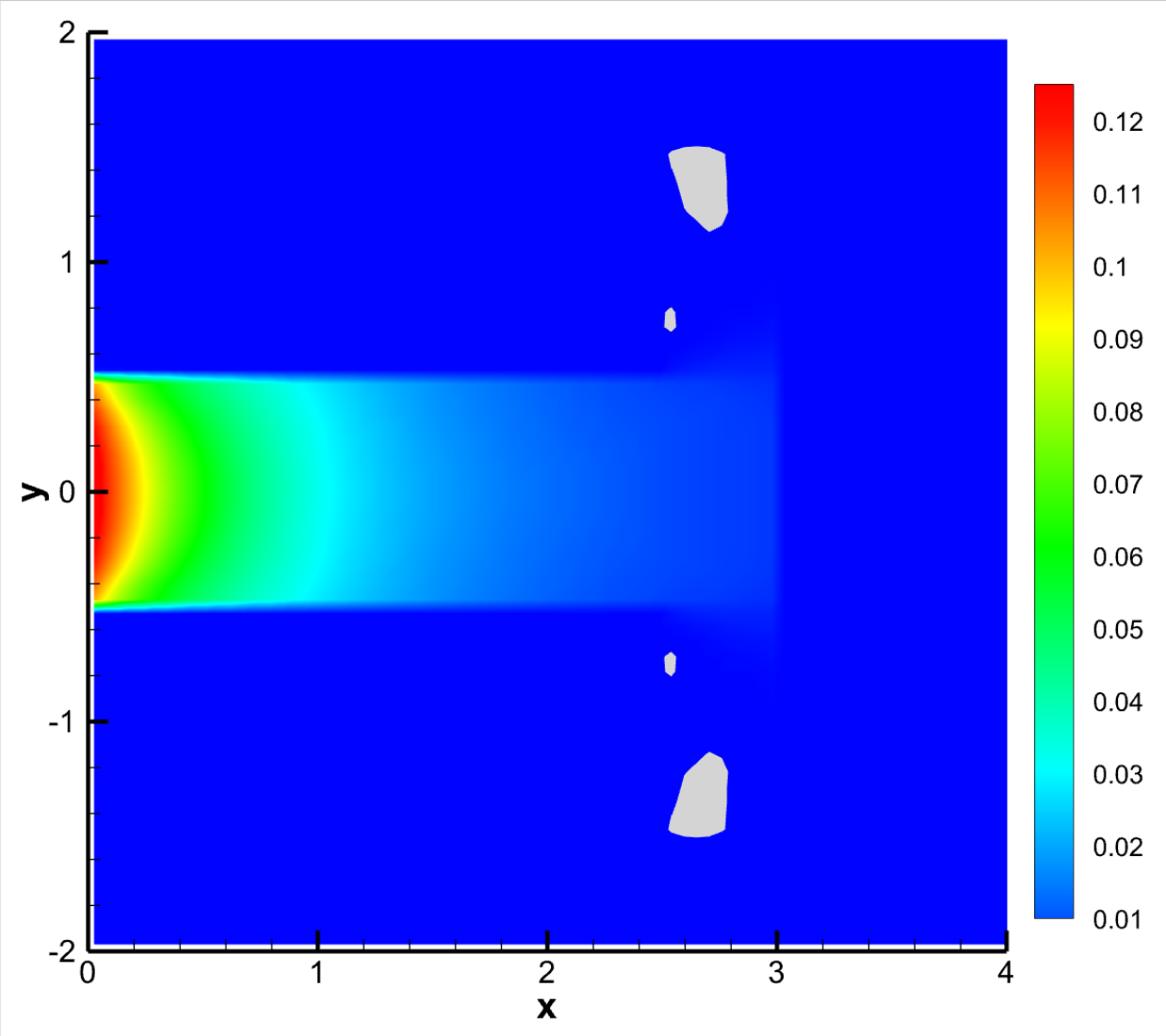}}
\vspace{-2mm}
\centerline{ \small (b1)  $FP_7$ ~t=0.5ns}
\end{minipage}
\begin{minipage}[h]{0.24\textwidth}
%\centering
\centerline{%\hspace{3.5cm}
\includegraphics[width=1.25in]{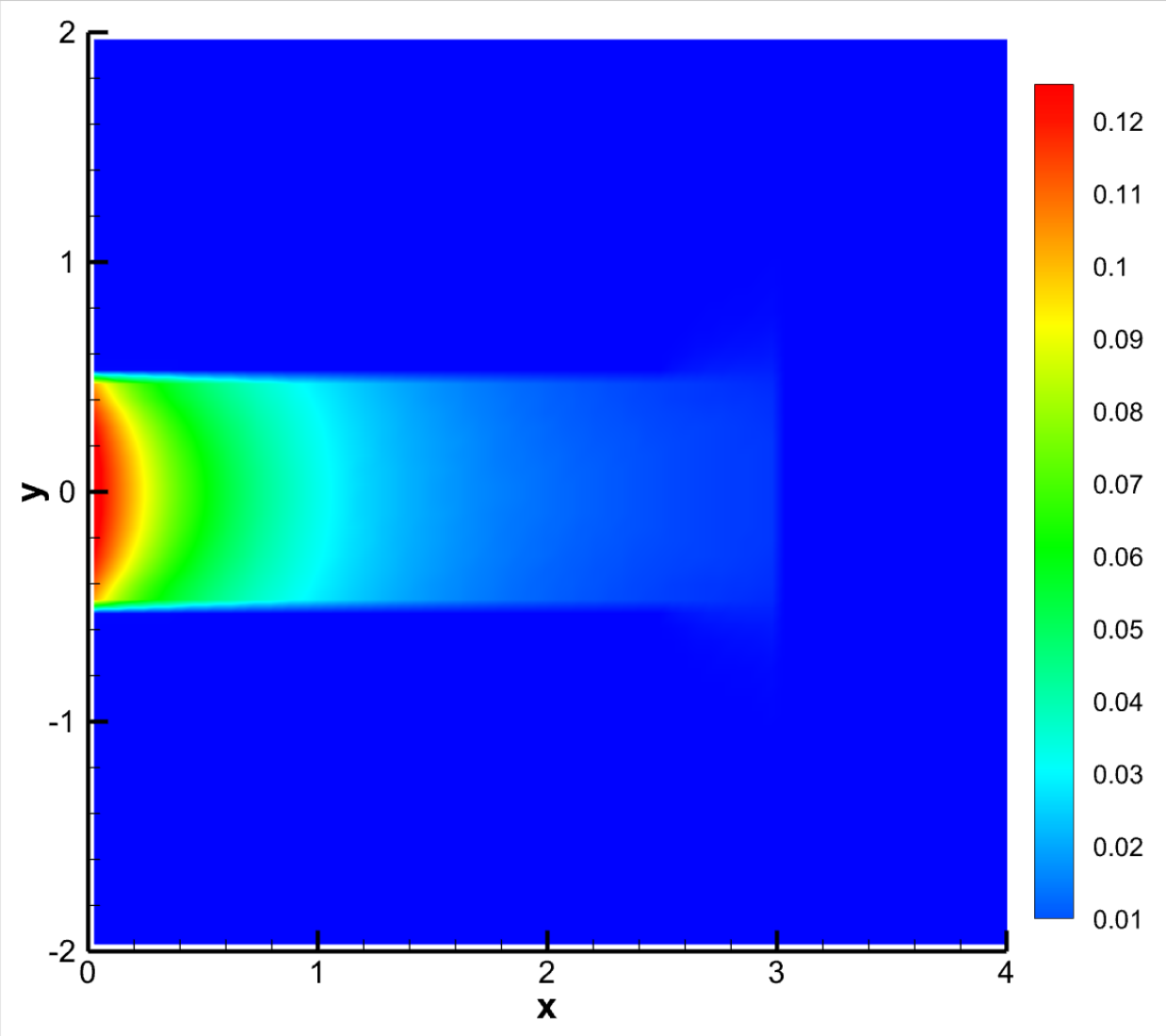}}
\vspace{-2mm}
\centerline{ \small (c1) $PPFP_7$ ~t=0.5ns}
\end{minipage}
\begin{minipage}[h]{0.24\textwidth}
%\centering
\centerline{%\hspace{3.5cm}
\includegraphics[width=1.25in]{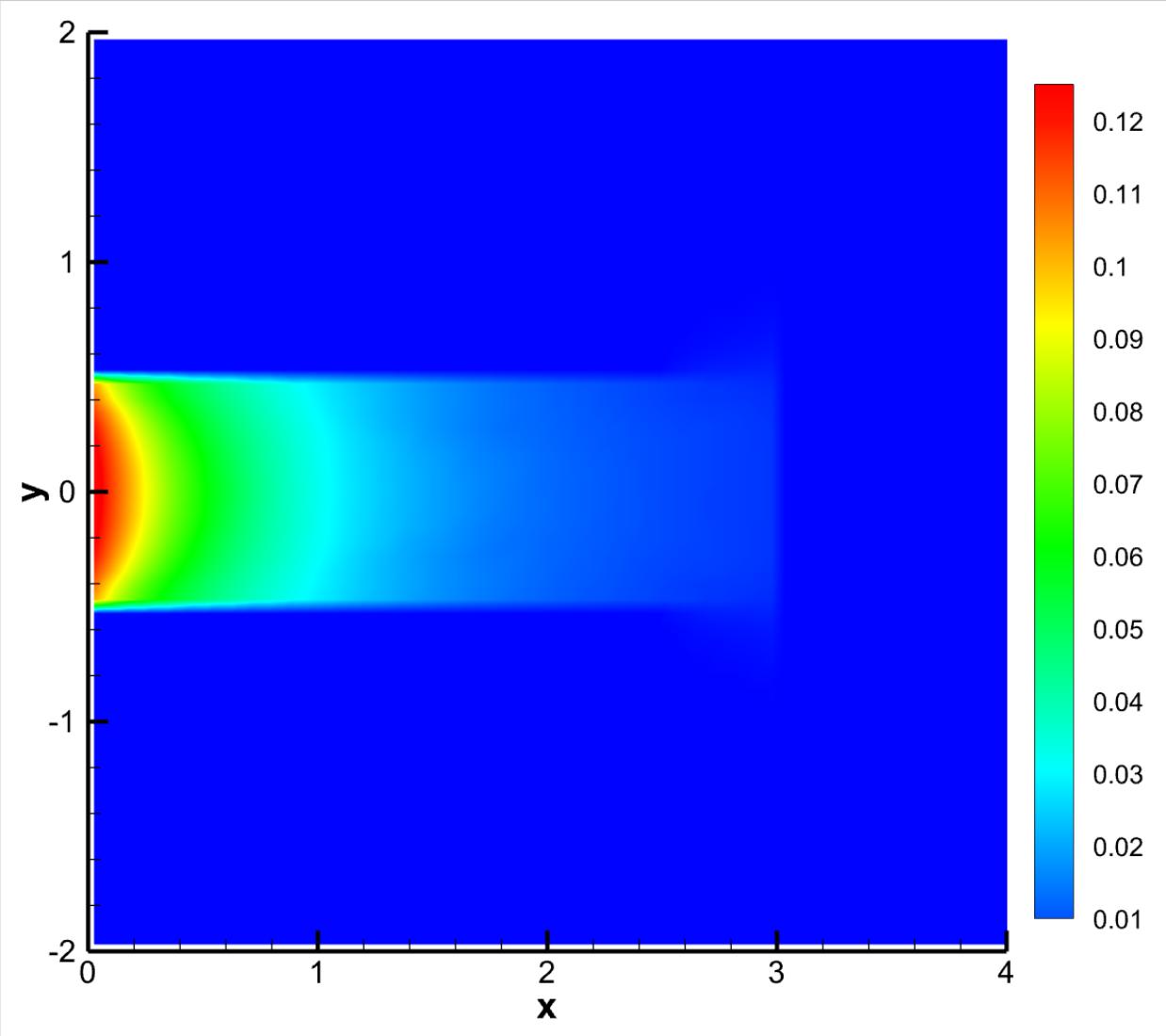}}
\vspace{-2mm}
\centerline{ \small (d1) $PPFP_7$-S t=0.5ns}
\end{minipage}
}

{
\begin{minipage}[h]{0.24\textwidth}
%\centering
%\hspace{2mm}
\centerline{
\includegraphics[width=1.25in]{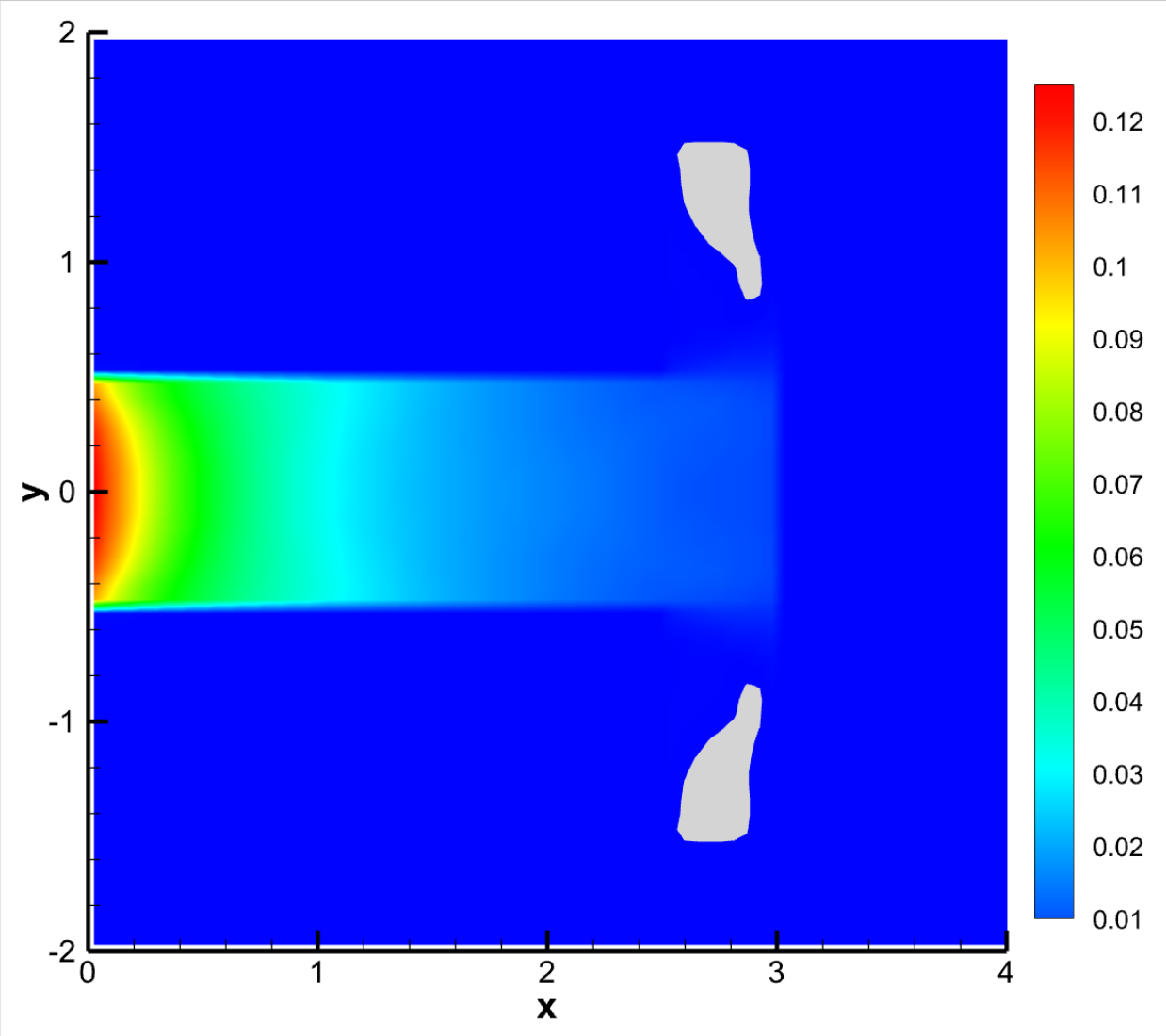}}
\vspace{-2mm}
 \centerline{  \small (a2) $P_7$ ~t=1.0ns}
\end{minipage}
\begin{minipage}[h]{0.24\textwidth}
%\centering
\centerline{%\hspace{3.5cm}
\includegraphics[width=1.25in]{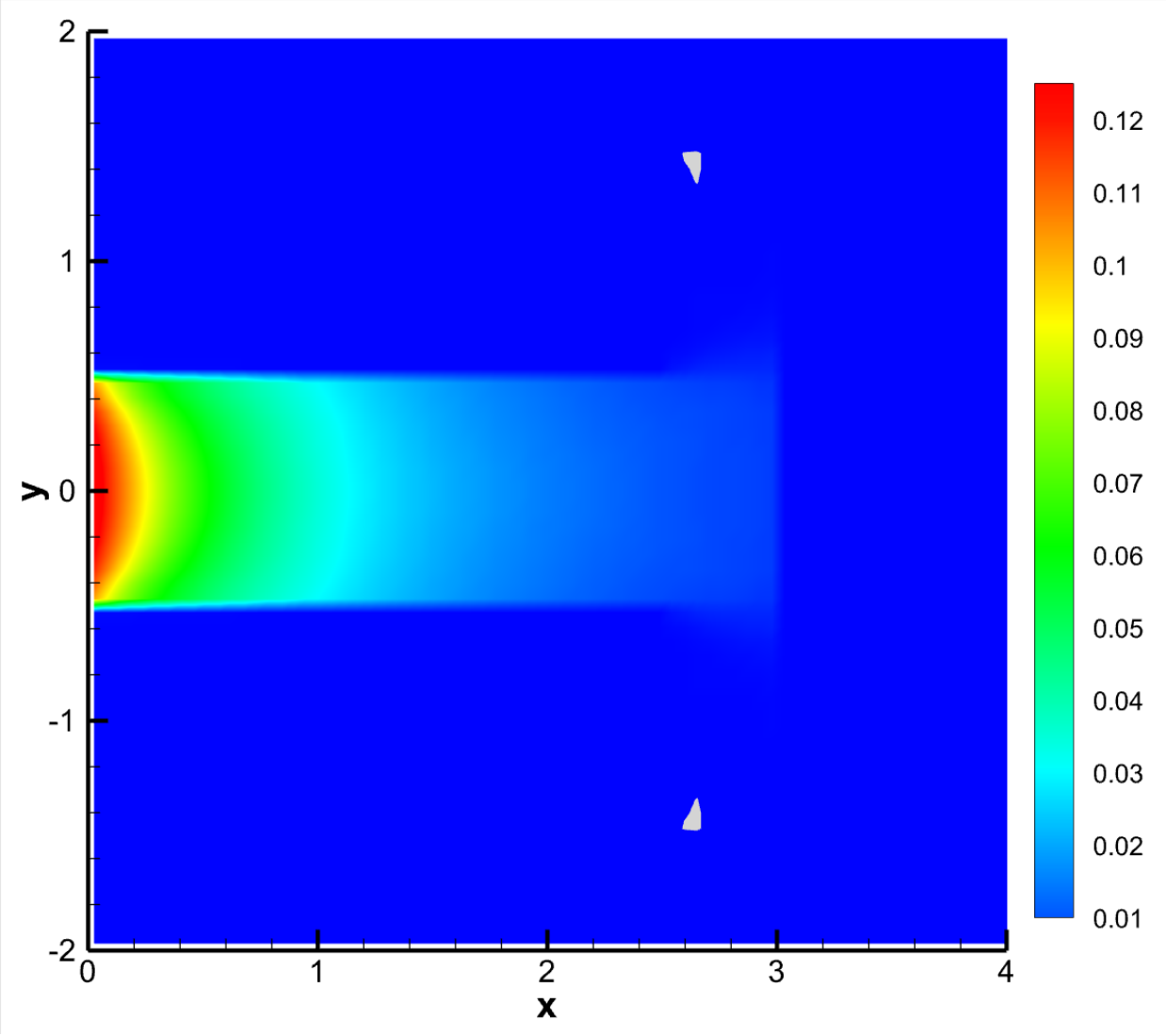}}
\vspace{-2mm}
\centerline{ \small (b2)  $FP_7$ ~t=1.0ns}
\end{minipage}
\begin{minipage}[h]{0.24\textwidth}
%\centering
\centerline{%\hspace{3.5cm}
\includegraphics[width=1.25in]{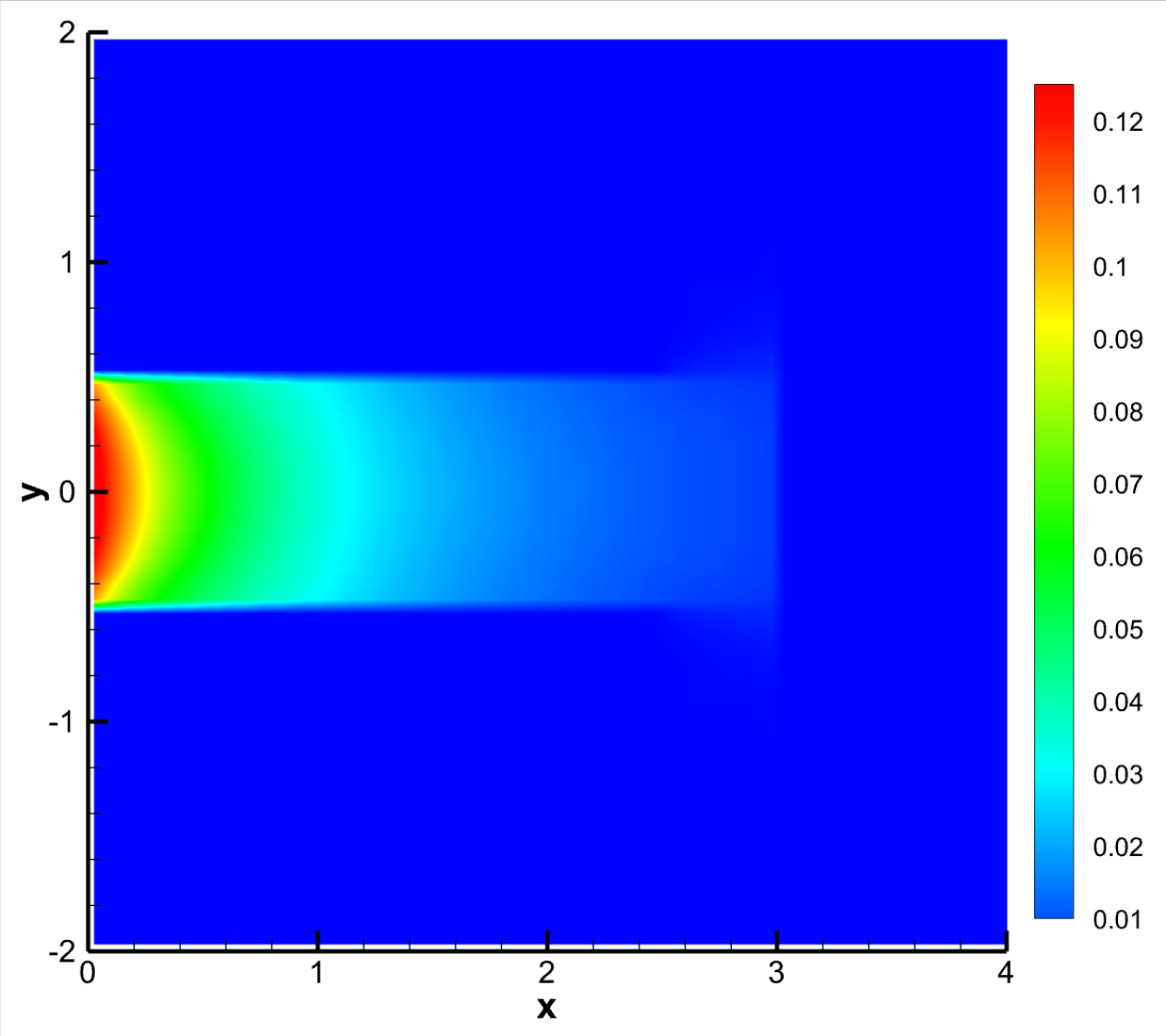}}
\vspace{-2mm}
\centerline{ \small (c2) $PPFP_7$ ~t=1.0ns}
\end{minipage}
\begin{minipage}[h]{0.24\textwidth}
%\centering
\centerline{%\hspace{3.5cm}
\includegraphics[width=1.25in]{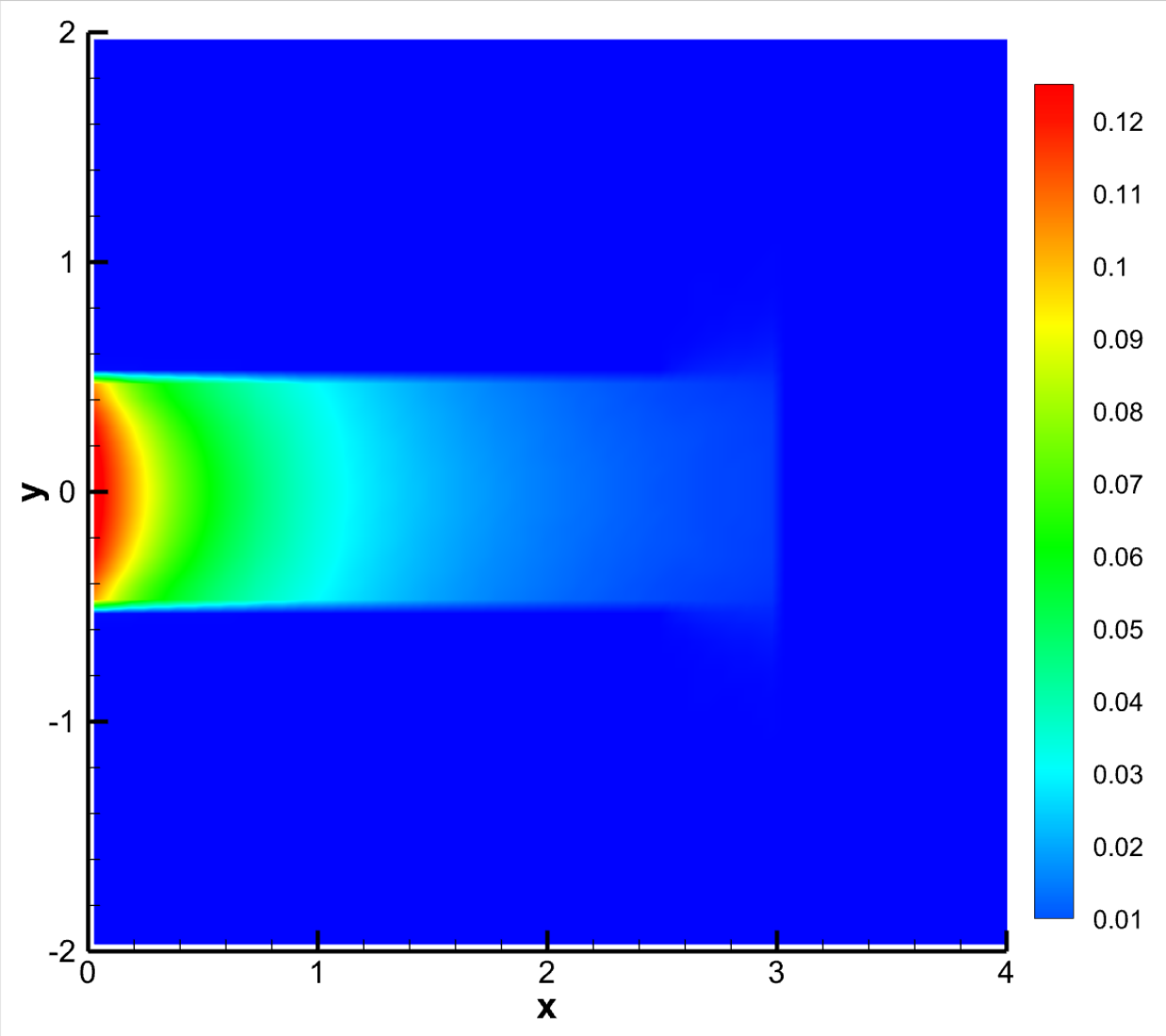}}
\vspace{-2mm}
\centerline{ \small (d2) $PPFP_7$-S t=1.0ns}
\end{minipage}
}

\caption{\small {\sc Example 6.} Contour of $\rho$. In the white regions $\rho$ becomes negative.
}
\label{Figure: contours of rho at t=0.5 and t=1.0 in Tophat test}
\end{figure}

\begin{figure}[htbp]
{
\begin{minipage}[t]{0.5\textwidth}
\centering
\centerline{\includegraphics[width=2.35in]{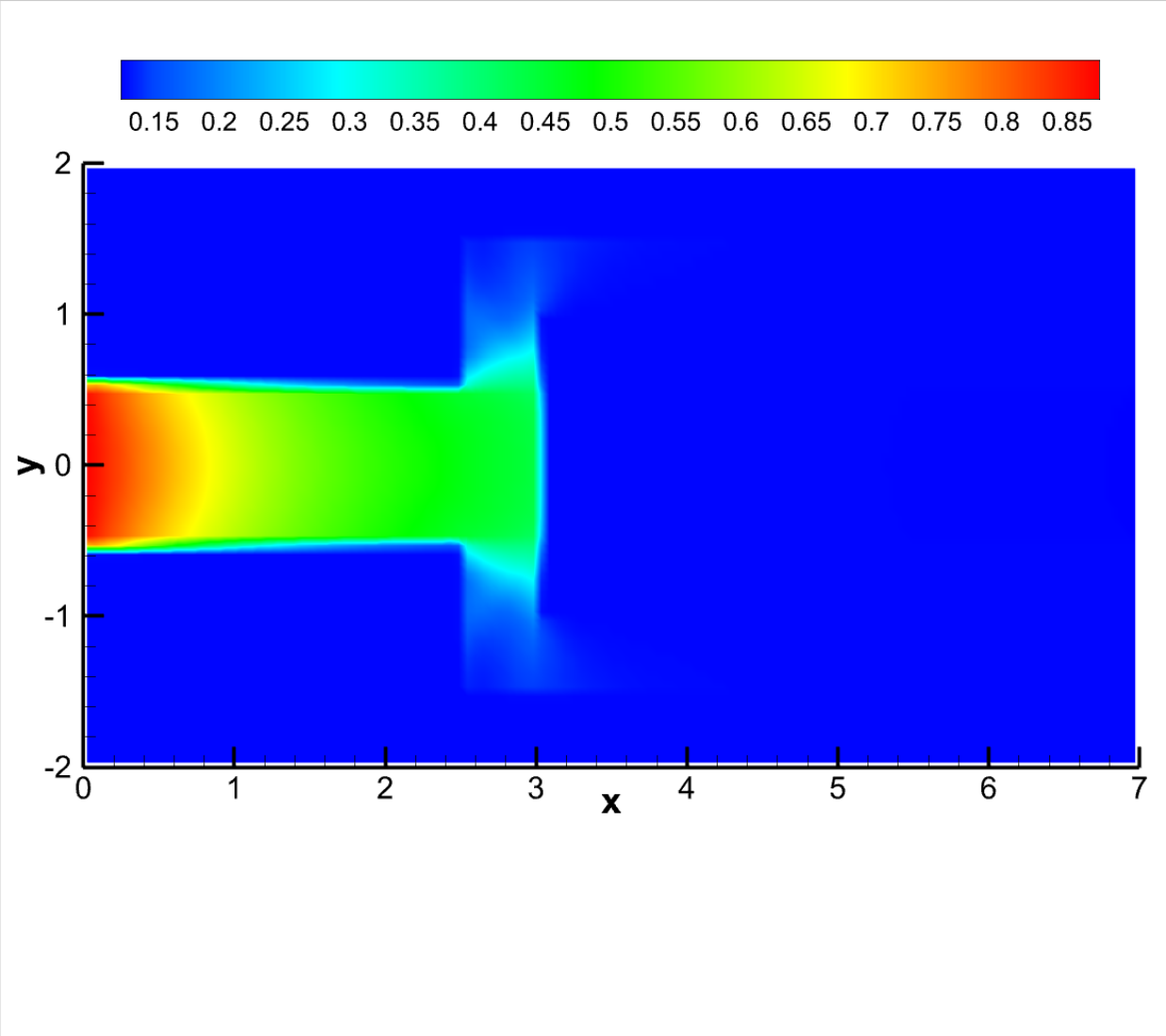}}%\caption{fig1}
\vspace{-10mm}
\centerline{$PPFP_7$~  t=8ns}
\end{minipage}
\begin{minipage}[t]{0.5\textwidth}
\centering
\centerline{\includegraphics[width=2.35in]{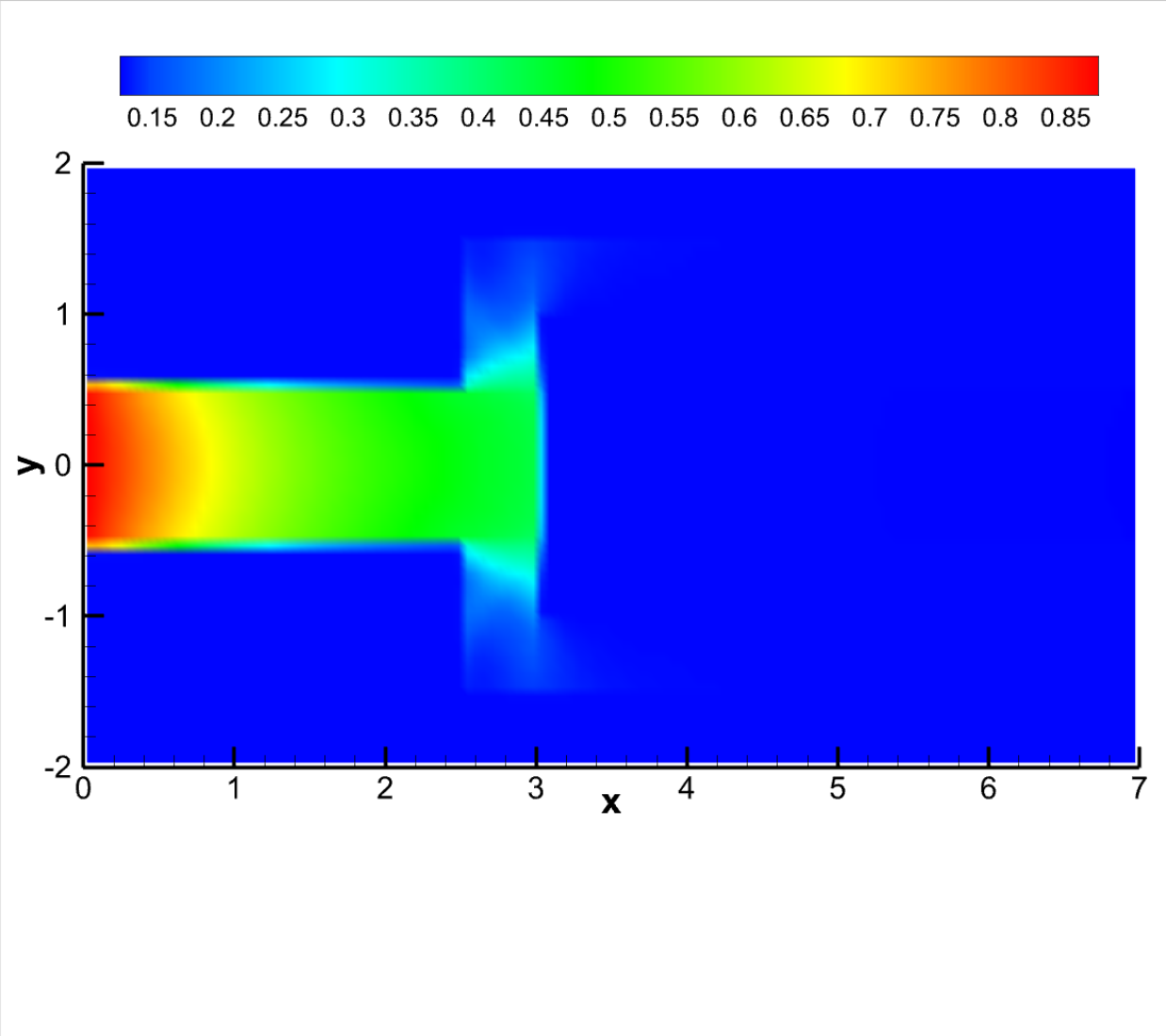}}%\caption{fig2}
\vspace{-10mm}
\centerline{$PPFP_7$-S~  t=8ns }
\end{minipage}%
}
{
\begin{minipage}[t]{0.5\textwidth}
\centering
\centerline{\includegraphics[width=2.35in]{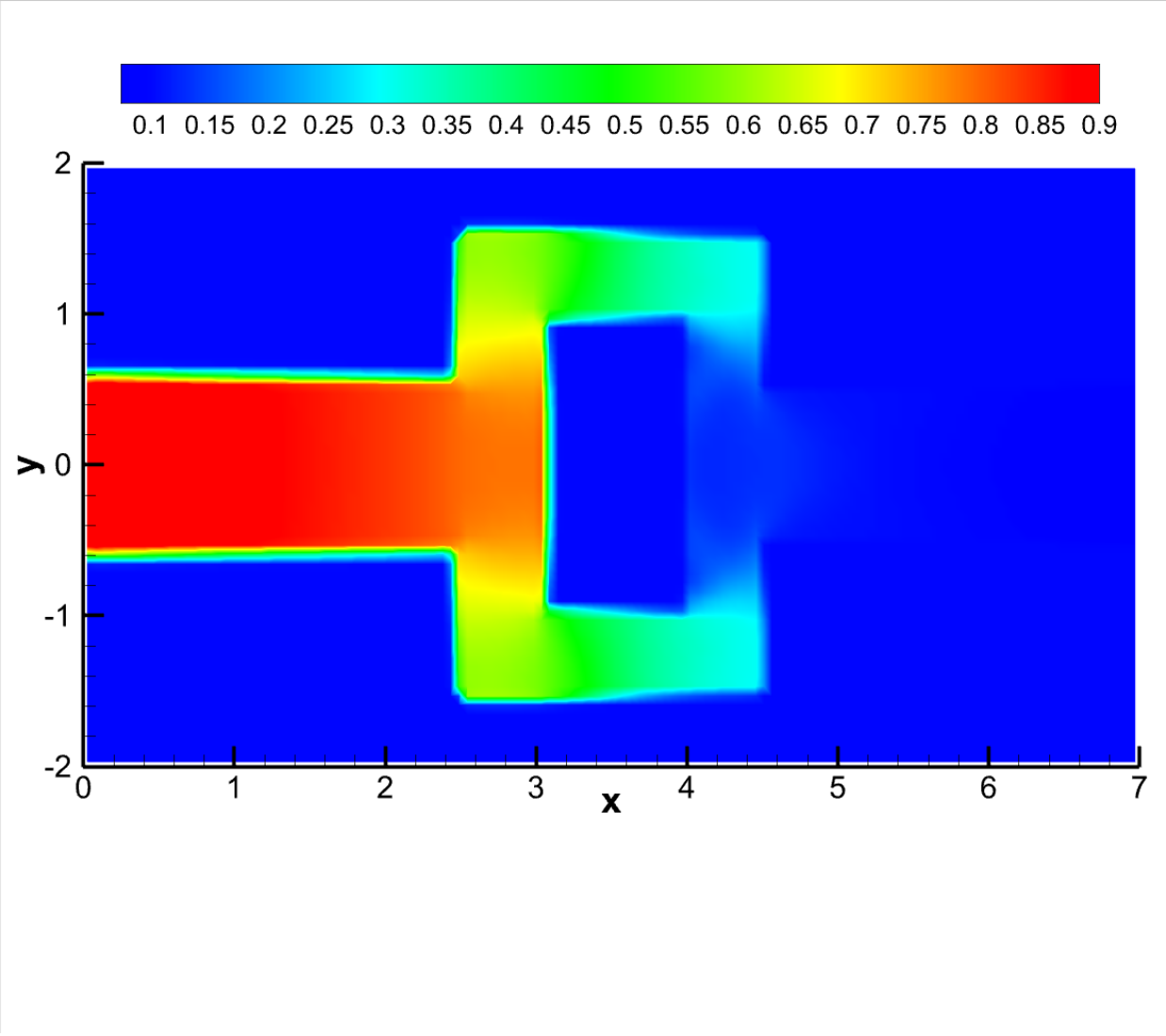}}%\caption{fig2}
\vspace{-10mm}
\centerline{$PPFP_7$~  t=94ns}
\end{minipage}
\begin{minipage}[t]{0.5\textwidth}
\centering
\centerline{\includegraphics[width=2.35in]{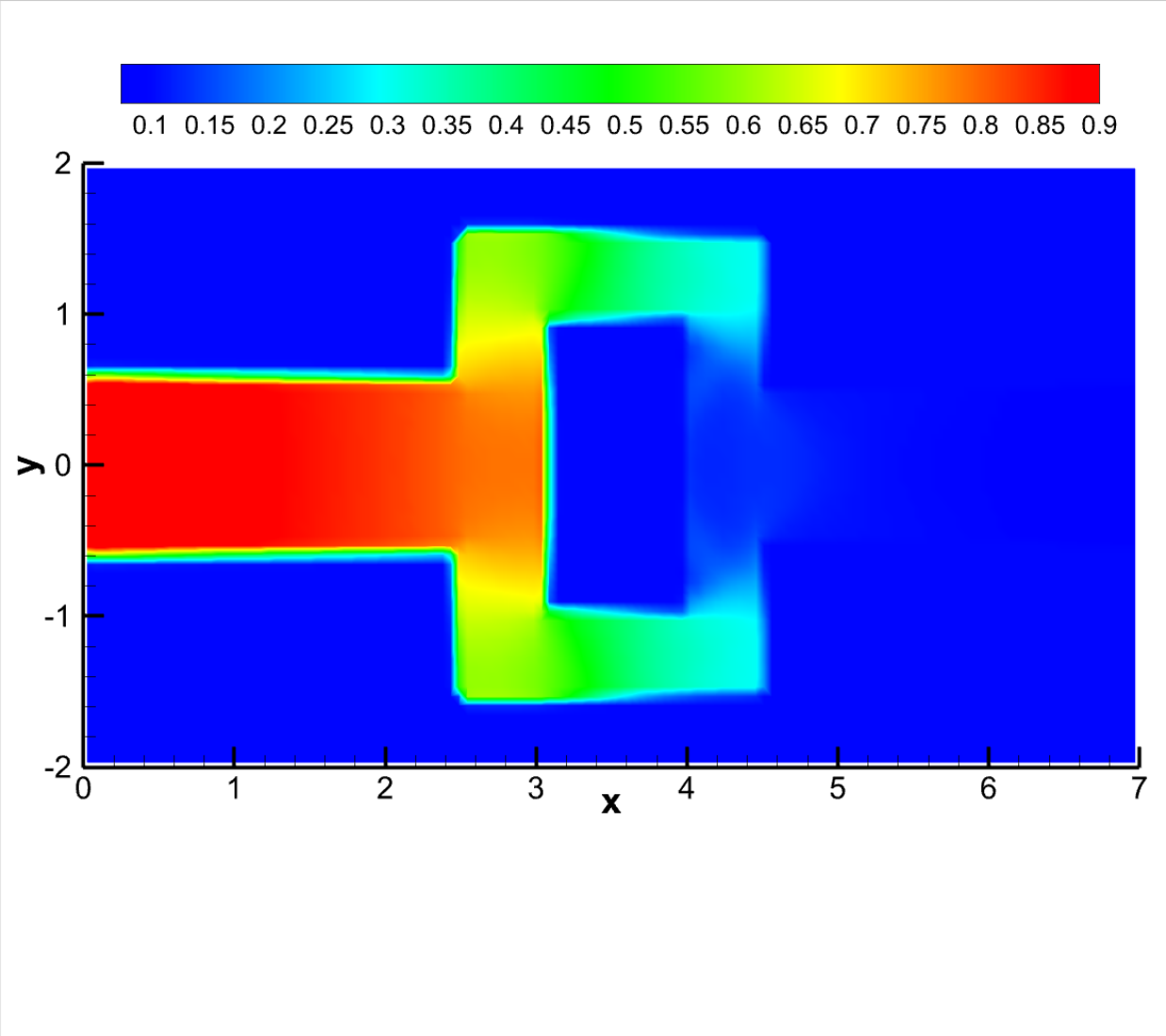}}%\caption{fig2}
\vspace{-10mm}
\centerline{$PPFP_7$-S~  t=94ns }
\end{minipage}
}
{
\begin{minipage}[t]{0.5\textwidth}
\centering
\centerline{\includegraphics[width=2.35in]{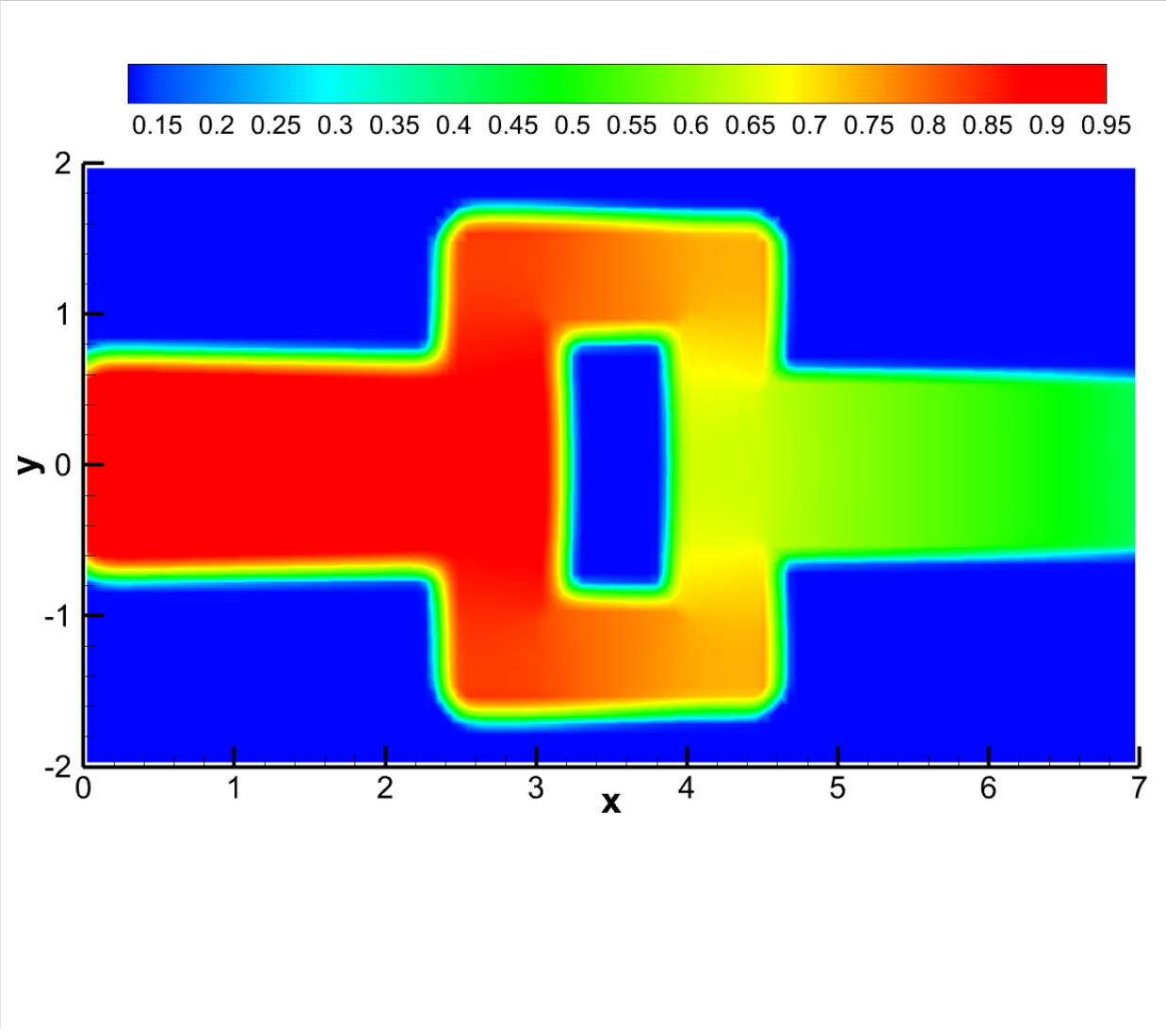}}%\caption{fig2}
\vspace{-10mm}
\centerline{$PPFP_7$~  t=1000ns}
\end{minipage}
\begin{minipage}[t]{0.5\textwidth}
\centering
\centerline{\includegraphics[width=2.35in]{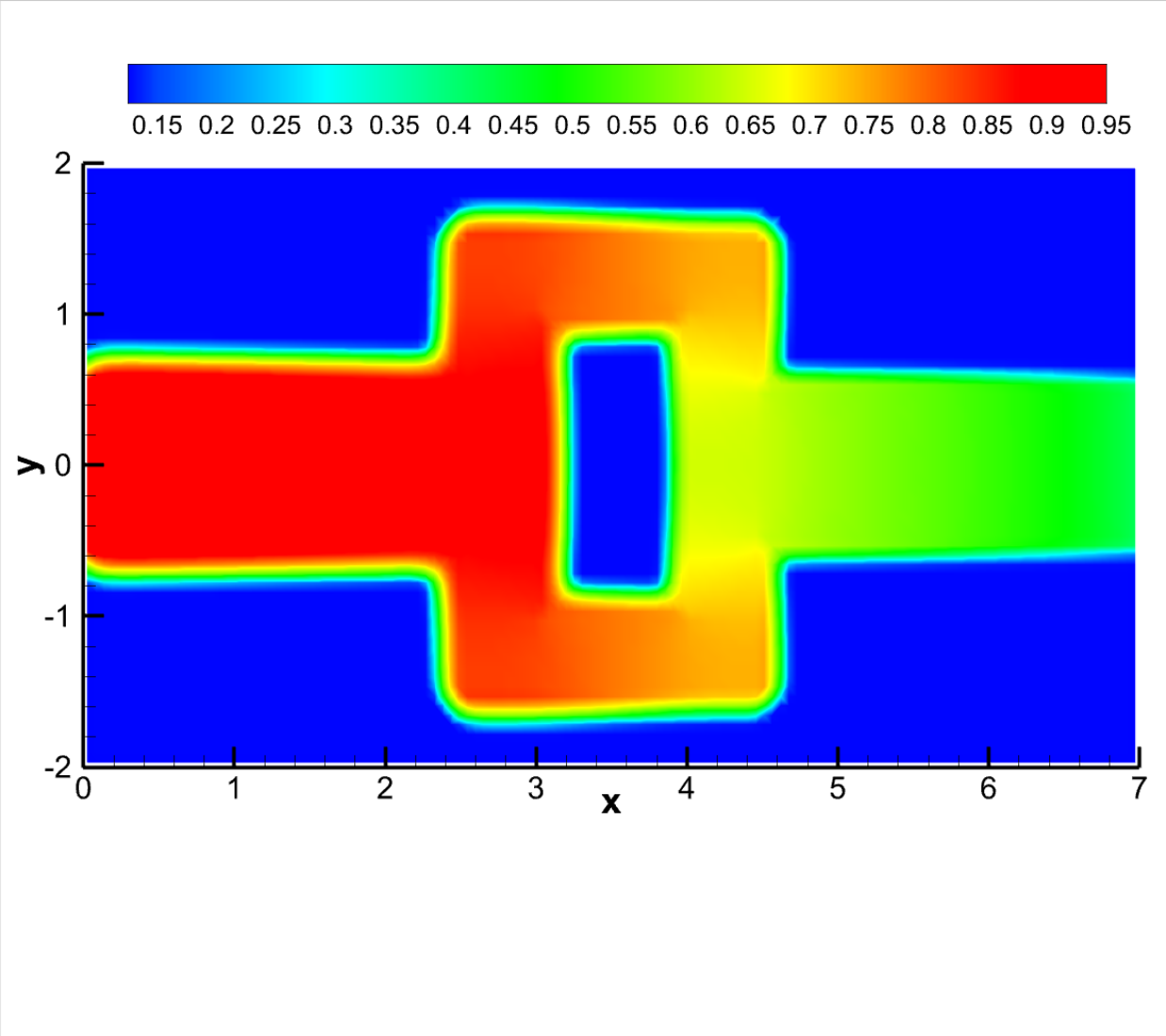}}%\caption{fig2}
\vspace{-10mm}
\centerline{$PPFP_7$-S~  t=1000ns}
\end{minipage}
}
\vspace{-8mm}
\caption{\small {\sc Example 6.} Contours of $T$. The temperature unit is $0.5$ keV.} \label{Figure: T at different times in Tophat Test}
\end{figure}

\section{Conclusion}
In this paper, we have extended the previous first-order positive and asymptotic preserving
$FP_N$ scheme for the nonlinear gray radiative transfer equations \cite{Xu-Jiang-Sun-2021} to a spatial second-order accuracy one.
The scheme applies the filtered spherical harmonics approach to discretize the angular variable, and UGKS to discretize the time and
spatial variables.

Due to the $FP_N$ angular discretization, the scheme is almost free of ray effects. It, meanwhile, can reduce the Gibbs phenomena
by the filter term. In the spirit of UGKS, it is easy to show that the scheme is AP.
Moreover, through a detailed analysis of the spatial second order fluxes which are obtained by using the IMC linearization method,
we are able to find the sufficient conditions that guarantee the positivity of the radiative energy density and material temperature.
Finally, by employing much cheap linear scaling limiters to make these conditions hold, the desired spatial second-order positive
and asymptotic preserving scheme, i.e., the $PPFP_N$-based UGKS, is obtained. The numerical examples validate the spatial second-order
accuracy of the scheme as expected. Furthermore, in order to reduce the computational costs of the fluxes in $PPFP_N$-based UGKS
in the regimes $\epsilon\ll 1$ and $\epsilon=O(1)$, a simplified version, called the $PPFP_N$-based SUGKS, is presented.

To our best knowledge, this is the first time that a spatial second-order positive, asymptotic preserving and almost free of ray effects
 scheme has been constructed for the nonlinear radiative transfer equations without operator splitting.
A number of standard (benchmark) problems have been tested to validate all the mentioned properties of the proposed schemes.

\section*{Acknowledgements}
The current research is supported by NSFC (Grant No. 12001451) for Xu, and by National Key R\&D Program (2020YFA0712200), and
National Key Project (GJXM92579), the Sino-German Science Center (Grant No. GZ 1465)
and the ISF-NSFC joint research program (Grant No. 11761141008) for Jiang, and by
%% CAEP foundation (No. CX20200026) and
National Key Project (GJXM92579) for Sun.

%\begin{appendices}
%
%\end{appendices}

\section{Appendix. Proof of Lemma \ref{lemma: analysis of fluxes}}\label{section: appendix}

In this appendix, we show the inequalities \eqref{eq1-analysis of fluxes}-\eqref{eq4-analysis of fluxes}
in Lemma \ref{lemma: analysis of fluxes}.
\vspace{2mm}

{\bf Proof of \eqref{eq1-analysis of fluxes}:}
 \vspace{2mm}

 After a straightforward calculation, we have
\begin{eqnarray}\label{Appendix: eq1}
\nonumber
& &
\Bigg| \frac{ \Phi^{n+1}_{i-1/2,j}(\tilde{b})- \Phi^{n+1}_{i+1/2,j}(\tilde{b}) }{\Delta x} \Bigg|
\\\nonumber
&=& \Bigg|
\frac{ \tilde{b}^{n+1}_{i-1/2,j}\left(
{\bm\langle}\xi^2\vec{\bm \psi} {\bm\rangle}_{1,4}\cdot\delta_x\vec{\bm I}^n_{i-1,j}
+{\bm\langle}\xi^2\vec{\bm \psi} {\bm\rangle}_{2,3}\cdot\delta_x\vec{\bm I}^n_{i,j}
\right)   }{\Delta x}\\\nonumber
& & \qquad -
\frac{ \tilde{b}^{n+1}_{i+1/2,j}\left(
{\bm\langle}\xi^2\vec{\bm \psi}{\bm\rangle}_{1,4}\cdot\delta_x\vec{\bm I}^n_{i,j}
+{\bm\langle}\xi^2\vec{\bm \psi} {\bm\rangle}_{2,3}\cdot\delta_x\vec{\bm I}^n_{i+1,j}
\right)    }{\Delta x}\Bigg|
\\\nonumber
&=&\Bigg|
{\bm\langle}\xi^2\vec{\bm \psi}{\bm\rangle}_{1,4}\cdot
\frac{
  \tilde{b}^{n+1}_{i-1/2,j}\delta_x\vec{\bm I}^n_{i-1,j}
-\tilde{b}^{n+1}_{i+1/2,j}\delta_x\vec{\bm I}^n_{i,j}
   }{\Delta x}
 +
 {\bm\langle}\xi^2\vec{\bm \psi} {\bm\rangle}_{2,3}\cdot \frac{
  \tilde{b}^{n+1}_{i-1/2,j}\delta_x\vec{\bm I}^n_{i,j}
-\tilde{b}^{n+1}_{i+1/2,j}\delta_x\vec{\bm I}^n_{i+1,j}
   }{\Delta x}\Bigg|
   \\\nonumber
&=&\Bigg|
{\bm\langle}\xi^2\vec{\bm \psi} {\bm\rangle}_{1,4}\cdot
\frac{
  \tilde{b}^{n+1}_{i-1/2,j}
  \left( \delta_x\vec{\bm I}^n_{i-1,j}-\delta_x\vec{\bm I}^n_{i,j} \right)
+\left(\tilde{b}^{n+1}_{i-1/2,j}-\tilde{b}^{n+1}_{i+1/2,j}\right)\delta_x\vec{\bm I}^n_{i,j}
   }{\Delta x}\\
& &\quad +
 {\bm\langle}\xi^2\vec{\bm \psi} {\bm\rangle}_{2,3} \cdot\frac{
  \tilde{b}^{n+1}_{i-1/2,j}\left( \delta_x\vec{\bm I}^n_{i,j}-\delta_x\vec{\bm I}^n_{i+1,j}\right)
+  \left(\tilde{b}^{n+1}_{i-1/2,j} -\tilde{b}^{n+1}_{i+1/2,j}\right)\delta_x\vec{\bm I}^n_{i+1,j}}{\Delta x}\Bigg|.
\end{eqnarray}
We assume that there are smooth vector functions  $\vec{\bm I}(x,y,t)$ (here we have slightly abused the notation), such that
$$\vec{\bm I}_{i,j}^n = \vec{\bm I}(x_i,y_j,t_n).$$
Substituting $\vec{\bm I}$ and the coefficient function $\widetilde{b}$ into \eqref{Appendix: eq1}, we use the differential mean value
theorem to find that for any component $I_{\ell}^m$ of $\vec{\bm I}$ and $\widetilde{b}$,
there are $(\zeta^1_x)_{\ell}^m, (\zeta^2_x)_{\ell}^m, \zeta_x$, such that
$$
\frac{\delta_x (I_{\ell}^m)^n_{i-1,j}-\delta_x(I_{\ell}^m)^n_{i,j}}{\Delta x} = -\frac{\partial^2 (I_{\ell}^m)^n((\zeta_x^1)_{\ell}^m, y_j) }{\partial x^2},$$
$$
\frac{\delta_x(I_{\ell}^m)^n_{i,j}-\delta_x(I_{\ell}^m)^n_{i+1,j}}{\Delta x} = -\frac{\partial^2 (I_{\ell}^m)^n((\zeta_x^2)_{\ell}^m, y_j) }{\partial x^2}, $$
$$
\frac{\tilde{b}^{n+1}_{i-1/2,j} -\tilde{b}^{n+1}_{i+1/2,j}}{\Delta x} = -\frac{\partial\tilde{b}^{n+1}(\zeta_x,y_j)}{ \partial x}.
$$
 For convenience, denote
$$
\frac{\partial^2 \vec{\bm I}^n(\zeta_x^1, y_j)}{\partial x^2}=
\left(\frac{\partial^2 (I_{0}^0)^n((\zeta_x^1)_{0}^0, y_j) }{\partial x^2},\frac{\partial^2 (I_{1}^{-1})^n((\zeta_x^1)_{1}^{-1}, y_j) }{\partial x^2},..., \frac{\partial^2 (I_{N}^N)^n((\zeta_x^1)_{N}^N, y_j) }{\partial x^2}\right)',
$$
and analogously for $\frac{\partial^2 \vec{\bm I}^n(\zeta_x^2, y_j)}{\partial x^2}$.

Therefore, one has
\begin{eqnarray}\nonumber
&&\Bigg| \frac{ \Phi^{n+1}_{i-1/2,j}(\tilde{b})- \Phi^{n+1}_{i+1/2,j}(\tilde{b}) }{\Delta x} \Bigg|
\\\nonumber
&=&\Bigg|
{\bm\langle}\xi^2\vec{\bm \psi} {\bm\rangle}_{1,4} \cdot
\left(- \tilde{b}^{n+1}_{i-1/2,j}
  \frac{\partial^2 \vec{\bm I}^n(\zeta_x^1, y_j) }{\partial x^2}
- \frac{ \partial\tilde{b}^{n+1}(\zeta_x,y_j)}{ \partial x}  \delta_x\vec{\bm I}^n_{i,j}
\right)
\\\nonumber
& &\qquad +
 {\bm\langle}\xi^2\vec{\bm \psi} {\bm\rangle}_{2,3} \cdot \left( -
  \tilde{b}^{n+1}_{i-1/2,j}
  \frac{\partial^2 \vec{\bm I}^n(\zeta_x^2, y_j) }{\partial x^2}
- \frac{\partial\tilde{b}^{n+1}(\zeta_x,y_j)}{ \partial x}  \delta_x\vec{\bm I}^n_{i+1,j}
\right)\Bigg|
\\\nonumber
&\lesssim&
\left |\tilde{b}^{n+1}_{i-1/2,j}\right| + \left|\frac{\partial\tilde{b}^{n+1}(\zeta_x,y_j)}{ \partial x} \right|.
\end{eqnarray}
In the same manner, we can obtain
\begin{eqnarray}\nonumber
\Bigg| \frac{ \Phi^{n+1}_{i,j-1/2}(\tilde{b})- \Phi^{n+1}_{i,j+1/2}(\tilde{b}) }{\Delta y} \Bigg|
&\lesssim&
\left |\tilde{b}^{n+1}_{i,j-1/2}\right| + \left|\frac{\partial\tilde{b}^{n+1}(x_i,\zeta_y)}{ \partial y} \right|.
\end{eqnarray}

Now, it remains to analyze $\widetilde{b}, \frac{\partial \widetilde{b} }{\partial x}, \frac{\partial \widetilde{b} }{\partial y}$.
After a straightforward calculation, one obtains
\begin{eqnarray}\nonumber
\frac{\partial \widetilde{b}}{\partial x}
= \frac{2c^2}{\epsilon^2 \Delta t \nu^3} \frac{\partial \nu}{\partial x}
\left( 1- e^{-\nu\Delta t} -\nu\Delta t e^{-\nu\Delta t}  \right) -
\frac{c^2\Delta t}{\epsilon^2\nu} \frac{\partial \nu}{\partial x} e^{-\nu\Delta t},
\end{eqnarray}
where we recall $\nu=\frac{c\sigma}{\epsilon^2}$. Similarly, we have an identity for $\frac{\partial \widetilde{b}}{\partial y}$.

Considering the regime $\epsilon \ll 1$ in which $\sigma=O(1)$, we can verify that
$$ |\widetilde{b} | = O( \frac{\epsilon^2}{\Delta t}), \qquad\quad
\Big| \frac{\partial \widetilde{b} }{\partial x} \Big|\ll \epsilon^2,\qquad\quad
\Big| \frac{\partial \widetilde{b} }{\partial y} \Big|\ll \epsilon^2. $$
On the other hand, for the regime $\epsilon = O(1)$, by the Taylor expansion, we have
$$ \widetilde{b}=O(\Delta t), \qquad \frac{\partial \widetilde{b} }{\partial x}=0
~\text{or}~ O( (\Delta t)^2),\qquad
\frac{\partial \widetilde{b} }{\partial y}=0
~\text{or}~ O( (\Delta t)^2). $$
Putting the above analyses together and noticing that $\epsilon \ll \Delta t$, we obtain \eqref{eq1-analysis of fluxes} immediately.
This completes the proof of inequality \eqref{eq1-analysis of fluxes}.
\vspace{2mm}

{\bf Proof of \eqref{eq2-analysis of fluxes}:}
\vspace{2mm}

In view of the expressions of $\Phi_{i-1/2,j}(\tilde{b})$ and $G_{i-1/2,j}(\tilde{b})$, one finds that
the proof of the inequality \eqref{eq2-analysis of fluxes} is almost the same as that of \eqref{eq1-analysis of fluxes},
and hence, we omit it here.
\vspace{2mm}

{\bf Proof of \eqref{eq3-analysis of fluxes} and \eqref{eq4-analysis of fluxes}:}
\vspace{2mm}

A straightforward calculation leads to
\begin{eqnarray} \label{Appendix: eq2}
\nonumber
& &\Bigg|\frac{ G_{i-1/2,j}(\tilde{c})- G_{i+1/2,j}(\tilde{c}) }{\Delta x}
\Bigg| \\\nonumber
&=&
\Bigg| \frac{\tilde{c}^{n+1}_{i-1/2,j}{\bm\langle}\xi \psi_{\ell}^m {\bm\rangle}~
\left( (\kappa^{n+1}\phi^n)_{i-1/2,j} +(1-\kappa_{i-1/2,j}^{n+1})\rho_{i-1/2,j}^{n+1} \right) }{\Delta x}
\\\nonumber
& &\qquad -
 \frac{\tilde{c}^{n+1}_{i+1/2,j}{\bm\langle}\xi \psi_{\ell}^m {\bm\rangle}~
\left( (\kappa^{n+1}\phi^n)_{i+1/2,j} +(1-\kappa_{i+1/2,j}^{n+1})\rho_{i+1/2,j}^{n+1} \right) }{\Delta x}\Bigg|
\\\nonumber
&=& \Bigg|  {\bm\langle}\xi \psi_{\ell}^m {\bm\rangle}~\left\{ \frac{\tilde{c}^{n+1}_{i-1/2,j} (\kappa^{n+1}\phi^n)_{i-1/2,j} - \tilde{c}^{n+1}_{i+1/2,j} (\kappa^{n+1}\phi^n)_{i+1/2,j}   }{\Delta x}\right. \\
& &\left. \qquad + \frac{ \tilde{c}^{n+1}_{i-1/2,j}(1-\kappa_{i-1/2,j}^{n+1})\rho_{i-1/2,j}^{n+1} - \tilde{c}^{n+1}_{i+1/2,j}(1-\kappa_{i+1/2,j}^{n+1})\rho_{i+1/2,j}^{n+1}  }{\Delta x}
\right\}
\Bigg|.
\end{eqnarray}
We suppose that there are smooth functions  $\phi(x,y,t), \rho(x,y,t)$, such that
$$\phi_{i,j}^n = \phi(x_i,y_j,t_n), \qquad\qquad \rho_{i,j}^n = \rho(x_i,y_j,t_n)  .$$
Substituting $\phi, \rho$ and the coefficient function $\widetilde{c}$ into \eqref{Appendix: eq2}, using the differential
mean value theorem, we see that there are $\zeta^1_x, \zeta^2_x$, such that
\begin{eqnarray}
\nonumber
& &\Bigg|\frac{ G_{i-1/2,j}(\tilde{c})- G_{i+1/2,j}(\tilde{c}) }{\Delta x}
\Bigg| \\\nonumber
&=& \Bigg| {\bm\langle}\xi \psi_{\ell}^m {\bm\rangle}~
 \frac{\partial (\tilde{c}^{n+1}\kappa^{n+1}\phi^n)}{\partial x} (\zeta^1_x, y_j) +
 \frac{\partial \left(\tilde{c}^{n+1}(1-\kappa^{n+1})\rho^{n+1} \right) }{\partial x}  (\zeta^2_x, y_j) \Bigg|
 \\\nonumber
 &\lesssim&
 \max_{k=1,2}
\left |\tilde{c}^{n+1}(\zeta^k_x, y_j)\right| + \max_{k=1,2} \left|\frac{\partial\tilde{c}^{n+1}}{ \partial x} (\zeta^k_x, y_j) \right|.
\end{eqnarray}
Similarly,
\begin{eqnarray}
\nonumber
\Bigg|\frac{ G_{i,j-1/2}(\tilde{c})- G_{i,j+1/2}(\tilde{c}) }{\Delta y}
\Bigg|
 &\lesssim&
 \max_{k=1,2}
\left |\tilde{c}^{n+1}(x_i, \zeta^k_y)\right| + \max_{k=1,2} \left|\frac{\partial\tilde{c}^{n+1}}{ \partial y} (x_i, \zeta^k_y) \right|.
\end{eqnarray}

Now, we focus on estimating $\widetilde{c}, \frac{\partial \widetilde{c} }{\partial x}$ and $\frac{\partial \widetilde{c} }{\partial y}$.
After a straightforward calculation, one has
\begin{eqnarray}
\nonumber
\frac{\partial \widetilde{c} }{\partial x}
&=& \left(  \frac{1}{2\pi\epsilon^3\nu} -\frac{1}{2\pi\Delta t \epsilon^3 \nu^2}
(1-e^{-\nu\Delta t})  \right) \frac{\partial \sigma_s}{\partial x} \\\nonumber
& & + \left( \frac{\sigma_s}{\pi\Delta t\epsilon^3\nu^3}(1-e^{-\nu\Delta t})
-  \frac{\sigma_s}{2\pi\epsilon^3\nu^2}(1+e^{-\nu\Delta t}) \right)  \frac{\partial \nu}{\partial x}.
\end{eqnarray}
A similar identity for $\frac{\partial \widetilde{c} }{\partial y}$ can be obtained in the same manner.

Consider the regime $\epsilon \ll 1$ where $\sigma=O(1)$. When $\sigma =$const., we see that
$\frac{\partial \sigma}{\partial x}=0$; otherwise, $\frac{\partial \sigma}{\partial x}=O(1)$. Thus, it is easy to verify that
$$ \widetilde{c} = O(\frac{1}{\epsilon}), \qquad \frac{\partial \widetilde{c} }{\partial x}= 0 ~\text{or}~ O(\frac{1}{\epsilon}),\qquad
\frac{\partial \widetilde{c} }{\partial y} =0 ~\text{or}~ O(\frac{1}{\epsilon}). $$
On the other hand, for the regime $\epsilon=O(1)$, we employ the Taylor expansion and argue similarly to arrive at
$$ \widetilde{c}= O(\Delta t), \qquad  \frac{\partial \widetilde{c} }{\partial x}= 0 ~\text{or}~  O(\Delta t),\qquad
\frac{\partial \widetilde{c} }{\partial y}= 0 ~\text{or}~ O(\Delta t). $$
Hence, recalling $\epsilon\ll\Delta t$, we get \eqref{eq3-analysis of fluxes} and \eqref{eq4-analysis of fluxes}.
This completes the proof of Lemma \ref{lemma: analysis of fluxes}.

%\bibliographystyle{abbrv}%{alpha}
%\bibliography{sample}

\end{document}